%% file: moments_08_09_11.tex
\documentclass{memo-l}%
\usepackage{amsfonts}
\usepackage{amsmath}
\usepackage{amsthm}
\usepackage[v2]{xy}
\usepackage{amssymb}
\usepackage{graphicx, color}%
\usepackage{enumerate}
\usepackage{setspace}
\usepackage{amscd}

\newtheorem{theorem}{Theorem}[chapter]
\newtheorem{lemma}[theorem]{Lemma}
\newtheorem{proposition}[theorem]{Proposition}
\newtheorem{corollary}[theorem]{Corollary}

\theoremstyle{definition}
\newtheorem{definition}[theorem]{Definition}
\newtheorem{example}[theorem]{Example}

\theoremstyle{remark}
\newtheorem{remark}[theorem]{Remark}

\numberwithin{section}{chapter}
\numberwithin{equation}{chapter}

\begin{document}
\frontmatter

\title{Iterated Function Systems, Moments, and Transformations of Infinite Matrices}
\author[P.E.T. Jorgensen]{Palle E. T. Jorgensen}
\address{Department of Mathematics, The University of Iowa, Iowa
City, IA 52242-1419, U.S.A.}
\email{jorgen@math.uiowa.edu}
\urladdr{http://www.math.uiowa.edu/\symbol{126}jorgen/}

\author[K.A. Kornelson]{Keri A. Kornelson}
\address{Department of Mathematics and Statistics,
Grinnell College, Grinnell, IA 50112-1690, U.S.A.}
\curraddr{Department of Mathematics, University of Oklahoma, Norman, OK 73019-0315, U.S.A.}
\email{kkornelson@math.ou.edu}
\urladdr{http://www.math.ou.edu/\symbol{126}kkornelson/}

\author[K.L. Shuman]{Karen L. Shuman}
\address{Department of Mathematics and Statistics,
Grinnell College, Grinnell, IA 50112-1690, U.S.A.}
\email{shumank@math.grinnell.edu}
\urladdr{http://www.math.grinnell.edu/\symbol{126}shumank/}
\date{September 11, 2008}



\thanks{This material is based upon work partially supported by the
U.S. National Science Foundation under grants DMS-0457581,
DMS-0503990, and DMS-0701164, by the University of Iowa Department
of Mathematics NSF VIGRE grant DMS-0602242, and by the Grinnell
College Committee for Support of Faculty Scholarship.  The second
author was supported in part by the Woodrow Wilson Fellowship
Foundation and the Andrew W. Mellon Foundation.}

\begin{abstract}
 We study the moments of equilibrium measures for
iterated function systems (IFSs) and draw connections to operator theory.  Our main object of study
is the infinite matrix which encodes all the moment data of a Borel
measure on $\mathbb{R}^d$ or $\mathbb{C}$.  To encode the salient features of a given IFS into precise moment data, we establish an interdependence between IFS equilibrium measures, the encoding of the sequence of moments of these measures into operators, and a new correspondence between the IFS moments and this family of operators in Hilbert space. For a given IFS, our aim is to establish a functorial correspondence in such a way that the geometric transformations of the IFS turn into transformations of moment matrices, or rather transformations of the operators that are associated with them.  

We first examine the classical
existence problem for moments, culminating in a new proof of the existence
of a Borel measure on $\mathbb{R}$ or $\mathbb{C}$ with a specified list
of moments.  Next, we consider moment problems associated with affine and
non-affine IFSs.  Our main goal is to determine
conditions under which an intertwining relation is satisfied by the moment
matrix of an equilibrium measure of an IFS.  Finally,
using the famous Hilbert matrix as our prototypical example, we study
boundedness and spectral properties of moment matrices viewed as
Kato-Friedrichs operators on weighted $\ell^2$ spaces.
\end{abstract}

\maketitle

\tableofcontents

\include{intro_08_09_11}

\chapter*{Acknowledgements}
The authors thank Professor Christopher French for several helpful
conversations and the proof of Proposition \ref{Prop:French}. In
addition, one or more of the authors had helpful conversations
with Professors Ken Atkinson, Dorin Dutkay, Erin Pearse, and
Myung-Sin Song.

The third author is pleased to thank the orthogonal polynomial
workshop students and graduate students at the University of Iowa
2007 REU for their energy and enthusiasm. She especially thanks
Bobby Elam, Greg Ongie, and Mark Tucker, whose questions led the
authors to the paper \cite{EST06}.

The second author is grateful to the Mathematics Department at the
University of Oklahoma - Norman for their hospitality during her research leave from Grinnell College.

\mainmatter

\include{notation_08_09_07}

\include{mproblem_08_09_11}

\include{affine_08_09_11}

\include{measurable_08_09_11}

\include{kato_08_09_11}

\include{integral_08_09_11}

\include{spectrum_08_09_11}

\include{extensions_08_09_11}



\backmatter
\bibliographystyle{amsalpha}
\bibliography{momentpaper}

\end{document}

%% file: intro_08_09_11.tex
\section*{Preface}

Moments of Borel measures $\mu$ on $\mathbb{R}^d$ have numerous
uses both in analysis and in applications.   For example, moments are used 
in the computation of orthogonal polynomials, in inverse spectral problems, in wavelets, in the analysis of fractals, in physics, and in probability 
theory.  In this paper, we study some well-known and perhaps
 not-so-well-known aspects of moment theory which have been motivated by the study of both the classical literature on moments and the newer literature on iterated function systems (IFSs).  
 
 Over the last hundred years, since the time of Lebesgue, mathematicians have adopted two approaches to measures on a topological (Hausdorff) space $X$.  In the first, measures are treated as functions on some sigma-algebra of ``measurable'' subsets of $X$.  In the second, measures are realized as positive linear functionals on a suitable linear space $C$ of functions on $X$.  If we take $C$ to be the compactly supported continuous functions of $X$, then Riesz's theorem states that the two versions are equivalent.  Starting with a measure $\mu$ on the Borel sigma algebra, integration with respect to $\mu$ yields a functional $L$ defined on $C$, and conversely (by Riesz), every positive linear functional $L$ on $C$ takes the form of integration against some Borel measure $\mu$.  The conclusion from Riesz's theorem asserts the existence of $\mu$.  
 
 The moment problem is an analog, taking $X$ to be $\mathbb{R}^d$ and replacing $C$ with the linear space of polynomials in $d$ variables.  The issue, as before, is to construct a measure from a functional.  Since the polynomials are spanned by the monomials, a functional $L$ is then prescribed by a sequence of moments.  Hence, we have the moment problem: determine a measure $\mu$ from the sequence of moments.  Given the moments, one can ask about the existence of a measure having those moments, the uniqueness of the measure if one exists, and the process for constructing such a measure and determining some of its properties.  We will touch on each of these topics in this Memoir.   

In Chapter \ref{Sec:Notation} we introduce our notation, definitions, and conventions.  We have collected here for the reader's convenience some of the tools from operator theory and harmonic analysis which we will use throughout.  In Chapter \ref{Sec:MomentTheory}, we review the classical moment existence problem, which asks the following question:  Given a list of moments $\{m_k\}_{k \geq 0}$, does there exist a Borel measure $\mu$ such that the $k^{\textrm{th}}$ moment of $\mu$ is $m_k$?   Alternately, we can arrange the moment sequence in a Hankel matrix in order to apply operator-theoretic techniques to this problem---a theme we carry throughout the Memoir.   Using tools of Kolmogorov, Parthasarathy, and Schmidt, we provide a new approach to this old problem,  which was first settled by Riesz (see \cite{Rie23}) nearly $100$ years ago.   Our main result here is a new proof for the moment existence problem in both $\mathbb{R}$ and $\mathbb{C}$.  We return to this same theme---the classical moment uniqueness problem---in Chapter \ref{Ch:Extensions}.  There, we describe in detail the theory of self-adjoint extensions of symmetric operators, which helps us determine precisely when a list of moments has more than one associated measure.

In Chapters \ref{Sec:Exist} and \ref{Sec:ComputeA}, we explore the difficult problem of computing moments directly for an equilibrium measure arising from an iterated function system (IFS).  An IFS is a finite set of contractive transformations in a metric space.  A theorem of Hutchinson \cite{Hut81} tells us that for each IFS, there exists a unique normalized equilibrium measure $\mu$ which is the solution to a fixed point problem for measures.  We show that every IFS corresponds to a non-abelian system of operators and a fixed point problem for infinite matrices.  We then prove that the moment matrix $M=M^{(\mu)}$ is a solution to this matrix fixed point problem, and in turn we exploit this fixed point property to compute or approximate the moments for $\mu$ in a more general setting than the affine cases studied in \cite{EST06}.  

As shown in \cite{EST06}, it is not straightforward to compute the moments for even the simplest Cantor equilibrium measures, but we \textit{can} compute the moments of an equilibrium measure $\mu$ directly in the affine IFS case.  However, we generally have no choice but to approximate the moments in non-affine examples.  In particular, our results can be applied to real and complex Julia sets, in which there is much current interest.  The non-affine moment approximation problem is surprisingly subtle, and as a result, we turn to operator-theoretic methods.  Affine IFSs are considered in Chapter \ref{Sec:Exist}, while non-affine IFSs and operator theory are considered in Chapter \ref{Sec:ComputeA}.  In addition, there are associated results about spectral properties of moment matrices for equilibrium measures in Chapter \ref{Sec:Spectrum}.

Infinite Hankel matrices cannot always be realized directly by operators in the $\ell^2$ sequence space.  We have been able to apply operator theoretic results more widely by allowing matrices to be realized as operators on a renormalized Hilbert space when necessary.  We turn to this problem in Chapter \ref{Ch:Kato}, where we introduce the operator theoretic extensions of quadratic forms by Kato and Friedrichs \cite{Kat80}.  The quadratic form we use in this context is induced by the moment matrix $M^{(\mu)}$.  Using the Kato-Friedrichs theorem, we obtain a self-adjoint operator with dense domain.  This generally unbounded Kato-Friedrichs operator can be used to obtain a spectral decomposition which helps us understand the properties of the quadratic form which gave rise to the operator.  Often, renormalized or weighted spaces allow us to use the Kato-Friedrichs theorem in greater generality.   

We continue to use Kato-Friedrichs  theory in the remaining chapters to understand the spectral properties of the moment matrix operator.  In Chapter \ref{Sec:IntOperators}, we use the classical example of the Hilbert matrix and its generalizations from Widom's work \cite{Wid66} in order to explore spectral properties of the moment matrix $M^{(\mu)}$ for general measures.  In particular, we show that the moment matrix is unitarily equivalent to a certain integral operator.   In Chapter \ref{Sec:Spectrum}, we further explore the spectral properties of the moment matrix and present some detailed examples.  Finally, Chapter \ref{Ch:Extensions} uses spectral theory as a tool to reexamine the classical moment problem, this time considering not only the existence of a measure having prescribed moments, but also the uniqueness.

Readers not already familiar with the theory of moments may find some of the classical references useful.  They treat
both the theory and the applications of moments of measures, and they
include such classics as \cite{ShTa43}, \cite{Sho47}, and \cite{Akh65}. The early uses of moments in mathematics were motivated to a large degree by
applications to orthogonal polynomials \cite{Sze75}.  More recent applications of orthogonal polynomials are numerical analysis, and
random matrix theory, random products, and dynamics. These
applications are covered
in \cite{Lan87a}, \cite{Lan87b}, \cite{Dei99}, while the edited and delightful volume
\cite{Lan87b} includes additional applications to geometry, to signal
processing, to probability and to statistics.

%% file: notation_08_09_07.tex
\chapter{Notation}\label{Sec:Notation}
In this section, we introduce some notation, conventions, and
definitions needed throughout the paper.

\section{Hilbert space notation}
We use the convention that the inner
product, denoted $\langle u|v \rangle$, is linear in the second
position $v$ and conjugate linear in the first position $u$.  This
choice results in fewer matrix transpositions in the type of
products we will be computing and therefore yields more
straightforward computations.

We will use Dirac's notation of ``bras'' and ``kets''.  For
vectors $u$ and $v$ in a Hilbert space, the inner product is
denoted ``bra-ket''  $\langle u | v \rangle$.  In contrast, the
``ket-bra'' notation $|u\rangle\langle v| $ denotes the rank-one
operator which sends a vector $x$ in the Hilbert space to a scalar
multiple of $u$. Specifically, it is the operator given by $x
\mapsto \langle v|x \rangle u$.  The Dirac notation for this
operation is $$ |u\rangle \langle v|\, x\rangle  = \langle v|x
\rangle u, $$ or also is sometimes written to preserve the order
$|u\rangle \langle v| \,x\rangle = u \langle v|x\rangle $.

If we consider these Hilbert space operations from the point of
view of matrices and Euclidean vectors, we denote a column vector
by $|u\rangle$, which we call a ``ket''.  Similarly, we denote a
row vector by $\langle u |$ and call it ``bra''.  The notation for
the inner product and rank-one operators above now are consistent
with the actual matrix operations being performed.  Further, the
algebraic manipulations with operators and vectors are done by
simply merging ``bras'' and ``kets'' in the position they
naturally have as we write them.

\section{Unbounded operators}\label{subsec:unboundedop}
We state here some definitions and properties related to unbounded
operators on a Hilbert space.  For more details, see \cite{ReSi80, Con90}.

\begin{definition} An \textit{operator} $F$  which maps  a Hilbert space $\mathcal{H}_1$ to another Hilbert space $\mathcal{H}_2$ is a linear  map from a subspace of $\mathcal{H}_1$ (the \textit{domain} of $F$) to $\mathcal{H}_2$.
\end{definition}

It will be assumed here that the domain of $F$ is dense in $\mathcal{H}_1$ with respect to the norm arising from the inner product on $\mathcal{H}_1$.   If $F$ is not a bounded operator, we call it an \textit{unbounded operator}.  

\begin{definition}[\cite{Con90}]\label{Def:FClosed} An operator $F$ from $\mathcal{H}_1$ to $\mathcal{H}_2$ is \textit{closed} if the graph of $F$, $\{( x , Fx ) \subset {\mathcal{H}_1 \times \mathcal{H}_2}\::\: x  \in \mathrm{dom}(F) \}$ is a closed set in the Hilbert space $\mathcal{H}_1 \times \mathcal{H}_2$ with inner product $\langle (x_1,x_2) | (y_1,y_2) \rangle_{\mathcal{H}_1 \times \mathcal{H}_2} = \langle x_1 | y_1 \rangle_{\mathcal{H}_1} + \langle x_2 | y_2 \rangle_{\mathcal{H}_2}$.  If there exists a closed extension to the operator $F$, then $F$ is called \textit{closable}.   In that case, there exists a smallest closed extension, which is called the \textit{closure} of $F$ and is denoted $\overline{F}$.
\end{definition}

\begin{definition}\label{Def:Adjoint} Let $F$ be an unbounded operator with dense domain $\mathrm{dom}\,(F)$ from $\mathcal{H}_1$ to $\mathcal{H}_2$.  Let $D$ be the set of all $y  \in \mathcal{H}_2$ such that there exists an $x \in \mathcal{H}_1$ such that \[\langle Fz|y \rangle_{\mathcal{H}_2 }=
\langle z| x \rangle_{\mathcal{H}_1} \] for all $z \in \mathrm{dom}\,(F)$.  We define the operator $F^*$ on the domain $D$ by $
F^*y = x$.  Note that $x$ is uniquely determined because $F$ is densely defined.  $F^*$ is called the \textit{adjoint} of $F$.
\end{definition}

It follows from these definitions (see \cite{ReSi80}) that $F$ is closable if and only if $F^*$ is densely defined, and that $F^*$ is always closed.

\begin{definition}\label{Def:SelfAdjoint} An operator $F$ on a Hilbert space $\mathcal{H}$ is \textit{symmetric} if $\mathrm{dom}(F) \subset \mathrm{dom}(F^*)$ and $Fx = F^*x$ for all $x \in \mathrm{dom}(F)$.  $F$ is \textit{self-adjoint} if $F$ is symmetric and  $\mathrm{dom}(F) = \mathrm{dom}(F^*)$; i.e.  $F^* = F$.
\end{definition}

Often, the term \textit{hermitian} is also used for a symmetric operator.  We see that a symmetric operator $F$ must be closable, since the domain of $F^*$ contains the dense set $\mathrm{dom}(F)$, and is therefore dense.  

\begin{definition} A symmetric operator $F$ is \textit{essentially self-adjoint} if its closure $\overline{F}$ is self-adjoint.
\end{definition}
If an operator $F$ is bounded on its dense domain and symmetric, it is essentially self-adjoint.  

\section{Multi-index notation}\label{subsec:index}

In $\mathbb{R}^d$, for $d > 1$, we will need to use multi-index
notation.  Here, $\alpha$ and $\beta$ denote multi-indices
belonging to $\mathbb{N}_0^d$, where $\mathbb{N}_0^d$ is the
Cartesian product
\begin{equation}
\mathbb{N}_0^d = \underbrace{\mathbb{N}_0\times\cdots\times\mathbb{N}_0}_{d \textrm{ times }}.
\end{equation}
Following standard conventions, the sum $\alpha + \beta$ is
defined pointwise:
\begin{equation}
\alpha + \beta = (\alpha_i + \beta_i)_{i=1}^d.
\end{equation}
Using this notation, we have the following integral expression:
\begin{equation}
 \int_{\mathbb{R}^d} x^{\alpha}\,\mathrm{d}\mu(x) =
\int_{\mathbb{R}^d} x_1^{\alpha_1}x_2^{\alpha_2}\cdots
x_d^{\alpha_d}\,\mathrm{d}\mu(x).
\end{equation}

\section{Moments and moment matrices}\label{Subsec:moments}
  Let $X$ be one of $\mathbb{R}, \mathbb{C}, \mathbb{R}^d$  with Borel
measure $\mu$.  When $X=\mathbb{R}$, we define the
$i^{\mathrm{th}}$ order moment with respect to $\mu$ to be
\begin{equation} m_i = \int_{\mathbb{R}} x^i \,\mathrm{d}\mu (x).
\end{equation}  If the moments of all orders are finite, we will
denote by $M^{(\mu)}$ an infinite matrix called the \textit{moment
matrix}. The moment matrix in the real case has entries
\begin{equation}\label{Eqn:MomMx} M^{(\mu)}_{i,j} = m_{i+j} = \int
x^{i+j} \,\mathrm{d}\mu(x).\end{equation}  Throughout this monograph, the real moment matrices will be indexed by $\mathbb{N}_0 \times \mathbb{N}_0$ -- in particular, the row and column indexing both start with $0$.

\begin{definition}\label{Defn:HankelN0d} An infinite real matrix $M$ whose entries are indexed by
$\mathbb{N}_0 \times \mathbb{N}_0$ is called a \textit{Hankel
matrix} if
\[ M_{i,j} = M_{i+k,j-k} = M_{i-k,j+k}\] for all values of $k$ for
which these entries are defined.
\end{definition}

The moment matrix defined in Equation (\ref{Eqn:MomMx}) is a Hankel
matrix.  This justifies our notation above referencing the $(i,j)^{\mathrm{th}}$
entry of $M^{(\mu)}$ with $i+j$.  The following examples will be used throughout the paper to illustrate our techniques and results.

\begin{example}\label{Ex:Lebesgue} Lebesgue measure on $[0,1]$. \end{example}
Let $\mu$ be the Lebesgue measure supported on
$[0,1]$.  Then the moment matrix for $\mu$ is \[M^{(\mu)}_{i,j} =
\int_0^1 x^{i+j} \,\mathrm{d}x = \frac{1}{i+j+1}.\]  This matrix
is often called the \textit{Hilbert matrix}.  We will examine the properties of the Hilbert 
matrix in greater detail in Section \ref{Sec:Hilbert}. \hfill $\Diamond$

\begin{example}\label{Ex:PointMass} Dirac point mass measure $\mu = \delta_1$. \end{example}  Let $\mu$ be the Dirac  probability mass $\delta_1$ on $\mathbb{R}$.  Then the moments for $\mu$ are \[ M^{(\mu)}_{i,j} = \int x^{i+j} \,\mathrm{d}\mu = 1^{i+j} = 1,\]  so the moment matrix entries are all $1$.  In this case, we see that the moment matrix does not represent a bounded operator on $\ell^2$.  This example will be studied further in Chapters \ref{Ch:Kato} and \ref{Sec:Spectrum}.  \hfill $\Diamond$

\begin{example}\label{Ex:Laguerre} The measure $\mu = e^{-x}\,\mathrm{d}x$. \end{example}   Let $\mu$ be $e^{-x}\mathrm{d}x$ on $\mathbb{R}^+ = (0,\infty)$.  Using integration by parts, a quick induction proof shows that the moments for $\mu$ are\[ M^{(\mu)}_{i,j} = (i+j)!.\]  Again, we see that the moments increase rapidly as $i,j$ increase, so $M^{(\mu)}$ cannot be a bounded operator on $\ell^2$.  We will discover more properties for this matrix in Chapters \ref{Ch:Kato} and \ref{Ch:Extensions}.  \hfill$\Diamond$

\begin{example}\label{Ex:Gaussian}  Measures with moments from the gamma function.\end{example}
Let \[\Gamma(k) = \int_0^{\infty} s^{k-1}e^{-s}\,\mathrm{d}s = (k-1)! \] when $k \in \mathbb{Z}_+$.  Set $\mu = e^{-p^2x^2}\mathrm{d}x$.  $\mu$ is a Gaussian measure with support on $\mathbb{R}$.  We have the odd moments equal to zero, and the even moments are given by
\[ \int_{\mathbb{R}} x^{2k}e^{-p^2x^2}\mathrm{d}x = \frac{\Gamma(k)}{2p^{2k+1}} \] for $k \in \mathbb{Z}_+$. \hfill $\Diamond$

If $X=\mathbb{C}$, the moments are now indexed by $\mathbb{N}_0
\times \mathbb{N}_0$ and are given by \begin{equation}m_{ij} =
\int_{\mathbb{C}} \overline{z}^iz^j
\mathrm{d}\mu(z).\end{equation}  If every moment is finite, then
the complex moment matrix $M^{(\mu)}$ is given by
\begin{equation}\label{Eqn:MomMxComplex} M^{(\mu)}_{i,j} = m_{ij} =
\int_{\mathbb{C}} \overline{z}^iz^j
\,\mathrm{d}\mu(z).\end{equation}  Notice that a complex moment
matrix will not in general have the Hankel property that arises in
real measures.  We also observe that the complex moments are equivalent to the inner
products of monomials in the Hilbert space $L^2(\mu)$:
\begin{equation} m_{ij} = \int_{\mathbb{C}} \overline{z}^i z^j
\mathrm{d}\,\mu(z)  = \langle z^i | z^j
\rangle_{L^2(\mu)}. \end{equation}   Moments and moment matrices arise naturally in the
study of orthogonal polynomials in $L^2(\mu)$.

\begin{example} If $\mu$ is the Lebesgue measure supported on the unit circle $\mathbb{T}$ in $\mathbb{C}$, then
the moment matrix is the identity matrix: \[ M^{(\mu)}_{j,k} =
\int_{\mathbb{T}} \overline{z}^jz^k \,\mathrm{d}\mu = \int_0^1
e^{2\pi i (k-j)x} \,\mathrm{d}x = \delta_{jk}.\]  As we noted
above, the identity matrix is not Hankel.
\end{example} \hfill $\Diamond$

In the case where $X = \mathbb{R}^d$, we index the moments using
the multi-index notation defined in Section \ref{subsec:index}.
Given $\alpha \in \mathbb{N}_0^d$, we have the $\alpha$-moment
\begin{equation}m_{\alpha} = \int_{\mathbb{R}^d} x^{\alpha}
\mathrm{d}\mu(x).\end{equation} The moment matrix for
$\mathbb{R}^d$ is actually indexed by $\mathbb{N}_0^d \times
\mathbb{N}_0^d$ and has entries
\begin{equation}\label{Eqn:multi-moment} M^{(\mu)}_{\alpha,\beta} = m_{\alpha+\beta} =
\int_{\mathbb{R}^d} x^{\alpha + \beta} \mathrm{d}\mu(x)
.\end{equation}

When $M$ has real entries indexed by $\mathbb{N}_0^d \times
\mathbb{N}_0^d$, we will also call $M$ a Hankel matrix if
\[ M_{\alpha, \beta} = M_{\alpha-\gamma, \beta+\gamma} = M_{\alpha + \gamma, \beta - \gamma}\]
for all $\gamma$ for which these entries are defined. We see from
Equation (\ref{Eqn:multi-moment}) that the moment matrix for a
measure on $\mathbb{R}^d$ ($d > 1$) also has the Hankel property.

\section{Computations with infinite matrices}\label{Sec:InfMatrices}
   In a number of applications throughout this paper, we will have occasion to use infinite matrices. They will be motivated by the familiar correspondence between linear transformations and matrices from linear algebra.  Given a linear transformation $T$ between two Hilbert spaces, assumed to be infinite dimensional, then a choice of ONBs in the respective Hilbert spaces produces an infinite matrix which represents $T$.  Moreover a number of facts from (finite-dimensional) linear algebra carry over: for example composition of two transformations (and a choice of ONBs) corresponds to multiplication of the associated two infinite matrices. Moreover by taking advantage of the orthogonality, one sees in Lemma \ref{Lem:Product} that the matrix multiplication is convergent.

     Given an infinite matrix $M$, it is a separate problem to determine when there exists a linear transformation $T$ between two Hilbert spaces and a choice of ONBs such that $M$ represents $T$. This problem is only known to have solutions in special cases.  We will address this further in Chapters \ref{Ch:Kato}  and \ref{Sec:IntOperators}. We will encounter infinite matrices which cannot be realized directly by operators in the $\ell^2$-sequence spaces, but for which an operator representation may be found after a certain renormalization is introduced.

  We use the following standard
language and notation: if $G = G_{i,j}$ is an infinite matrix
indexed by $\mathbb{N}_0 \times \mathbb{N}_0$, and $x$ is an
infinite vector indexed by $\mathbb{N}_0$, we will define the
matrix-vector product $Gx$ componentwise using the usual rule,
provided the sums all converge.  $$(Gx)_i = \sum_{j \in
\mathbb{N}_0} G_{i,j}x_j.$$ Similarly, for infinite matrices $G$
and $H$, the matrix product $GH$ is also defined componentwise
using the same summation formula used for finite matrices,
provided these infinite sums all converge in the appropriate
sense. Specifically, the $(i,j)^{\mathrm{th}}$ entry of the matrix
$GH$ is given by
\begin{equation}\label{Eqn:MatrixProduct} (GH)_{i,j} = \sum_{k \in \mathbb{N}_0} G_{i,k}H_{k,j}.
\end{equation}

  As an aside, there is a parallel notion for such products which
applies to matrices indexed by $\mathbb{N}^d_0 \times
\mathbb{N}^d_0$ and vectors indexed by $\mathbb{N}_0^d$.  Given
$M$ indexed by $\mathbb{N}_0^d \times \mathbb{N}_0^d$, a vector
$c$, and given $\alpha \in \mathbb{N}_0^d$, we have

\begin{equation}\label{Eqn:MatrixMult}
(Mc)_{\alpha} = \sum_{\beta \in \mathbb{N}_0^d}
M_{\alpha,\beta}c_{\beta},
\end{equation}
provided that this sum converges absolutely.  We particularly need
absolute convergence here so that the sum can be appropriately
reordered to sum over $\mathbb{N}_0^d$.

\begin{definition}
When the formal summation rules for a product of infinite matrices
or a matrix-vector product yield convergent sums for each entry,
we say that the matrix operations are \textit{well defined}.
\end{definition}

Using our matrices $G$ and $H$ above, if the sums $\sum_{k \in
\mathbb{N}_0} G_{i,k}H_{k,j}$ converge for all $i,j \in
\mathbb{N}_0$, then the matrix product $GH$ is well defined.

If $G$ and $H$ are operators on a separable Hilbert space, we can determine necessary conditions on
their corresponding matrix representations so that the matrix products are well defined.  If we
take the Hilbert space to be $\mathcal{H} = \ell^2(\mathbb{N}_0)$ and let $\{e_j\,:\, j \in
\mathbb{N}_0\}$ be the standard orthonormal basis in $\ell^2(\mathbb{N}_0)$, i.e., $e_j(k) =
\delta_{j,k}$ for $j,k \in \mathbb{N}_0$, then we must first assume that $G$ and $H$ are defined on
a dense domain which includes $\{e_j\}$.   This allows us to write the matrix representations of
$G$ and $H$.  From there, we have the following result.
\begin{lemma}\label{Lem:Product} Let $G$ and $H$ be linear
operators densely defined on $\ell^2(\mathbb{N}_0)$ such that
$Ge_j$, $G^*e_j$, and $He_j$ are well defined and in
$\ell^2(\mathbb{N}_0)$ for every element of the standard
orthonormal basis $\{e_j\}_{j\in \mathbb{N}_0}$.  Then $G=
(G_{i,j})$ and $H=(H_{i,j})$, the infinite matrix
representations of $G$ and $H$ respectively, are defined and the
matrix product $((GH)_{i,j})$ is well defined.
\end{lemma}

\begin{proof}  Recall
that the matrix representation of the operator $G$ is $G_{i,j} =
\langle e_i | Ge_j \rangle_{2}$ for $i,j \in \mathbb{N}_0$,
provided $Ge_j$ is in $\ell^2$ so that the inner product is
finite.  Our hypotheses imply, then, that the matrix
representations of $G$ and $H$ exist.

Given an operator $G$, the adjoint operator $G^*$  satisfies
$\langle G^*u|v\rangle = \langle u|Gv \rangle$ for all $u,v$ which
are in the appropriate domains and for which $G^*u$ and $Gv$ are
in $\ell^2(\mathbb{N}_0)$.  We now check for absolute convergence
of the matrix product sums from Equation
(\ref{Eqn:MatrixProduct}):
\begin{eqnarray*}  \sum_{k=0}^{\infty} \left| G_{ik}H_{kj}
\right| &=& \sum_{k=0}^{\infty} \Big| \langle e_i|Ge_k \rangle
\langle e_k|He_j \rangle\Big| \\ &=& \sum_{k=0}^{\infty} \left|
\langle G^*e_i|e_k \rangle \langle e_k|He_j \rangle\right|
\\&\leq& \left(\sum_{k=0}^{\infty} \left|\langle G^*e_i|e_k \rangle \right|^2
\right)^{1/2} \left(\sum_{l=0}^{\infty} \left|\langle e_l|He_j
\rangle \right|^2 \right)^{1/2} \\
&=& \|G^*e_i\|_2 \|He_j\|_2 < \infty.
\end{eqnarray*}
We use the Cauchy-Schwarz inequality above since we know the
vectors $G^*e_i$ and $He_j$ are in $\ell^2$.   Note that if $G,H$
are bounded operators, the sum above is bounded by
$\|G^*\|_{op}\|H\|_{op}$ for any choice of $i,j \in \mathbb{N}_0$.

Once we know the sums from Equation (\ref{Eqn:MatrixProduct}) are
absolutely convergent, they can be computed:

\begin{eqnarray*} \sum_{k=0}^{\infty} G_{i,k}H_{k,j} &=&
\sum_{k=0}^{\infty} \langle e_i|Ge_k \rangle \langle e_k|He_j
\rangle\\&=& \sum_{k=0}^{\infty} \langle G^*e_i|e_k \rangle
\langle e_k|He_j \rangle\\&=&  \langle G^*e_i|He_j \rangle \qquad
\text{by Parseval's Identity} \\&=& \langle e_i |(GH)e_j \rangle\\
&=& (GH)_{i,j}.
\end{eqnarray*}
\end{proof}

Note that the conditions in Lemma \ref{Lem:Product} imply that the
operator $G$ is automatically closable.  It is an immediate
corollary that bounded operators satisfy the hypotheses of Lemma
\ref{Lem:Product}, and thus their matrices will always have well
defined products.
\begin{corollary}  If $G$ and $H$ are bounded operators on
$\ell^2(\mathbb{N}_0)$, then the product of their matrix
representations is well defined.
\end{corollary}

 In some of the applications which follow -
for example the cases for which both $G$ and $H$ are lower
triangular matrices - the range of the summation index will be
finite. (See, for example,  the
proof of Lemma \ref{Lemma:Amatrix}.)  But we will also have occasion to
compute products $GH$ for pairs of infinite matrices $G$ and $H$
where the range of the summation index is infinite. In those
cases, we must check that the necessary sums are convergent.

          There are two approaches to working with the product $GH$ of infinite
matrices $G$ and $H$.   One is the computational approach we have
described above, and the other involves associating the matrices
to operators on appropriate Hilbert spaces (see Chapters
\ref{Ch:Kato} and \ref{Sec:IntOperators}).  For many applications,
the first method is preferred.  In fact, there are no known
universal, or canonical, procedures for turning an infinite matrix
into an operator on a Hilbert space (see e.g., \cite{Hal67} and
\cite{Jor06}), so often the computational approach is the only one
available.

\section{Inverses of infinite matrices}\label{Sec:Inverses}

     Computations in later sections will require the notion of inverse for infinite
matrices.  Let  $G$ be an infinite matrix indexed, as discussed in Section \ref{subsec:index}, by the set $\mathbb{N}_0^d  \times \mathbb{N}_0^d$.  We will now make
precise our (admittedly abusive) use of the notation $G^{-1}$ for
an inverse.  First,  in order to discuss computations, the index set for
rows and columns must be equipped with an order.  If $d > 1$, we
will use the order of the set $\mathbb{N}_0^d$ along successive finite diagonals.  Our goal is to find an  algorithm for the entries in the infinite matrix we denote
$G^{-1}$. This does turn out to be possible for the infinite matrices
we will be using in our analysis of moments and of
transformations.
 \begin{definition}

          Let $E$ be an infinite matrix.  We say that $E$ is an \textit{idempotent} if the
matrix product $E^2$ is well defined and if  $E^2 = E$.
\end{definition}

\begin{definition}\label{def:inverse}
     Let $G, H, E_1$ and $E_2$  be infinite matrices.  Assume that the matrix products $GH$  and $HG$ are well
defined.  We say that $G$ is a \textit{left inverse} of $H$  if there is an idempotent matrix $E_1$ such that $GH = E_1$.  $G$ is called a \textit{right inverse} of $H$ if there is an idempotent  $E_2$ such
that $HG = E_2$.  If $G$ is both a left and right inverse of $H$, $G$ is called an \textit{inverse} of $H$.
\end{definition}

While this notion of inverses for infinite matrices is not
symmetric, the following lemma does justify the use of the term
``inverse''.

\begin{lemma}\label{lem:inverse}
Let $G$ and $H$ be infinite matrices such that both
matrix products $GH$ and $HG$ are well defined. Suppose there is
an idempotent $E_1$ such that $GH = E_1$ and $E_1G = GE_1 = G$.
Then the infinite matrix $HG$ is an idempotent, denoted $E_2$,
which satisfies $E_2H = HE_1$ and $E_2E_1 = E_2$.
\end{lemma}

\begin{proof}  First, we see that $HG$ is idempotent: $$HGHG =
H(E_1)G = HG.$$ The formulas are also readily verified: $$ E_2H =
HGH = HE_1
$$ and
$$ E_2E_1 = HGE_1 = HG = E_2.$$
\end{proof}

A result of Lemma \ref{lem:inverse} is that if we know the matrix
product $GH = E_1$ is an idempotent such that $E_1G = GE_1 = G$,
then by Definition \ref{def:inverse}, $G$ and $H$ are inverses.

\begin{example}\label{Ex:Inverses} The following matrices arise in Example \ref{Ex:Rank2} with respect to measures which are convex combinations of Dirac masses.  \end{example}  We are given the matrices 
\begin{equation*}
G = \begin{bmatrix}
1         & 0         & 0 & \cdots\\
-1        & 2         & 0 & \cdots \\
0         & 0         & 0 &\cdots\\
\vdots & \vdots &       &\ddots
\end{bmatrix}
\quad\text{and}\quad H_1 = \begin{bmatrix}
1 & 0 & 0  & \cdots\\
\frac{1}{2} & \frac{1}{2}  & 0 &\cdots \\
\frac{1}{2} & \frac{1}{2}  & 0 &\cdots \\
\vdots        & \vdots        &\vdots &\ddots\\
\end{bmatrix}.
\end{equation*}
We see that $GH_1 = E$ where $E$ is the idempotent (in fact, projection):
\begin{equation*}
E = GH_1 = \begin{bmatrix}
1 & 0 & 0 & \cdots\\
0 & 1 & 0  & \cdots\\
0 & 0 & 0  & \cdots\\
\vdots & \vdots &\vdots & \ddots
\end{bmatrix}
\end{equation*}
Therefore, $H_1$ is an inverse of $G$.  Note that this inverse is not unique.  The matrix 
\begin{equation*}H_2=  \begin{bmatrix}  
1 & 0 & 0  & \cdots\\
\frac{1}{2} & \frac{1}{2}  & 0 &\cdots \\
\frac{1}{2} & \frac{1}{2}  & \frac12 &\cdots \\
\frac12        & \frac12       &\frac12 &\ddots\\
\vdots & \vdots & \vdots & \vdots
\end{bmatrix}
\end{equation*}
also satisfies $GH_2=E$ hence is also an inverse of $G$.   \hfill $\Diamond$

%% file: mproblem_08_09_11.tex
\chapter{The moment problem}\label{Sec:MomentTheory}

In this chapter we introduce the \textit{moment
problem}, in which we describe the various positivity conditions
that an infinite matrix $M$ must satisfy in order to imply that
there is a Borel measure $\mu$  on
an ambient space $X$ such that $M$ is the
moment matrix of $\mu$, i.e. $M = M^{(\mu)}$.  These conditions will be different in the
cases where $X$ is a subset of $\mathbb{R}$, $\mathbb{C}$, or
$\mathbb{R}^d$ for $d>1$.  We also examine whether or not the measures found will be unique.   The existence of the measure $\mu$ is discussed in Section \ref{Subsec:exist}, and a procedure by Parthasarathy using a Kolmogorov construction is used in Section \ref{Subsec:Parth} to discuss a more concrete construction of $\mu$.  The connections of the moment problem to other areas of interest to mathematicians, physicists, and engineers are mentioned in Section \ref{Subsec:history}.

\section{The moment problem $M = M^{(\mu)}$}\label{Subsec:exist}

Given an infinite matrix $M$, the moment problem addresses whether
there exist measures $\mu$ such that $M = M^{(\mu)}$.  There are known
conditions on $M$ which ensure that such a measure $\mu$ exists.  Our
presentation will be brief and we will omit the proofs of the results
stated here, as they are available in the literature, albeit scattered
within a variety of journal articles. Fuglede's paper \cite{Fug83}
offers a very readable survey of the literature on moments in several
variables, up to 1983. The two papers \cite{BeDu06, BeDu07} emphasize
the complex case and include new results.

Our intention in this section is primarily to give the definitions of
the positive semidefinite properties on the infinite matrices in the
real and complex cases.  We will only treat the existence part of the moment
problem, but there are a variety of interesting uniqueness results in
the literature as well. Here we shall only need the simplest version
of uniqueness: these are the cases when the measures $\mu$ are known
\textit{a priori} to be compactly supported. In this case, uniqueness
follows as a result of an application of the Stone-Weierstrass
Theorem.  An example of nonuniqueness is seen in Section
\ref{Sec:Nonunique}.

We begin with the real case $\mathbb{R}^d$, for $d \geq 1$.  Let
$\mathcal{D}$ be the space of all infinite sequences $c$ which are
finite linear combinations of elements of the standard orthonormal
basis $\{e_{\alpha}\}_{\alpha \in \mathbb{N}_0^d}$.  Clearly
$\mathcal{D}$ is a dense subspace of the Hilbert space
$\ell^2(\mathbb{N}_0^d)$, and matrix-vector products are well defined
on vectors in $\mathcal{D}$.

One of the ways to test whether a given sequence, or a given infinite
Hankel matrix, is composed of the moments of a measure $\mu$ is to
test for a positive semidefinite condition.  While there are several
such conditions in the literature, we isolate condition (\ref{Eqn:PD})
as it summarizes a variety of features that will be essential for our
point of view.  Note that (\ref{Eqn:PD}) entails a separate
verification for every finite system of numbers, hence it amounts to
checking that all the finite truncated square matrices have positive
spectrum.  This in turn can be done by checking determinants of finite
submatrices.

\begin{definition}\label{Defn:PD}
A matrix $M$ with real entries indexed by $\mathbb{N}_0^d \times
\mathbb{N}_0^d$ is said to be \textit{positive semidefinite} if
 \begin{equation}\label{Eqn:PD} \langle c | Mc
\rangle_{\ell_2} = \sum_{\alpha \in \mathbb{N}_0^d} \sum_{\beta \in
\mathbb{N}_0^d} \overline{c}_{\alpha}
M_{\alpha,\beta}c_{\beta} \geq 0\end{equation}
for all $c\in\mathcal{D}$.
\end{definition}

We say that a given infinite matrix $M$ is a $(\mathbb{N}_0^d,
\mathbb{R}^d)$-moment matrix if there exists a positive Borel
measure $\mu$ on $\mathbb{R}^d$ such that
\begin{eqnarray} M_{\alpha,\beta} &=& M^{(\mu)}_{\alpha,\beta} =  \int_{\mathbb{R}^d} x^{\alpha + \beta} \mathrm{d}\,\mu(x) \label{Eqn:MMuAlphaBeta}
\end{eqnarray}
for all $\alpha, \beta \in \mathbb{N}_0^d$.  Note that, as we observed in Section \ref{Subsec:moments},  if $M$ is a  moment matrix, it must have the Hankel property.

The solution to the existence part of the moment problem in the
$\mathbb{R}^d$ case is given in a theorem due to M. Riesz
\cite{Rie23} for $d=1$ and Haviland \cite{Hav35, Hav36} for
general $d$.

\begin{theorem}[Riesz, Haviland]\label{thm:real}
For any positive semidefinite real Hankel matrix $M = (M_{\alpha,\beta})$
indexed by $\mathbb{N}_0^d \times \mathbb{N}_0^d$, there is a
positive Borel measure $\mu$ on $\mathbb{R}^d$ having moment
matrix $M$.
\end{theorem}

If $M$ is an infinite matrix with complex entries, we say $M$ is a complex moment matrix if there exists a positive Borel measure $\mu$ on $\mathbb{C}$ such that
\begin{eqnarray} M_{i,j} = M^{(\mu)}_{i,j}
= \int_{\mathbb{C}} \overline{z}^i z^j \mathrm{d}\,\mu(z)
\end{eqnarray}
for all $i,j \in \mathbb{N}_0$. We define a property on the complex matrix $M$ which is stronger than the positive semidefinite property.
\begin{definition}[\cite{BeDu06, BeDu07}] \label{Defn:PDC}
We say a complex infinite matrix $M$ indexed by $\mathbb{N}_0 \times \mathbb{N}_0$ has \textit{Property PD$\mathbb{C}$} if given any doubly-indexed sequence $c = \{c_{i,j}\}$ having only finitely many nonzero entries,   we have
\begin{equation}\label{Eqn:PDC}
\sum_{i,j,k,\ell \in \mathbb{N}_0}
\overline{c}_{i,j} M_{i+\ell, j+k} c_{k,\ell} \geq 0.
\end{equation}
\end{definition}
Note that if we consider the
special case where $j,\ell=0$ in Definition \ref{Defn:PDC}, we get
exactly the positive semidefinite condition on $M$ given in Definition \ref{Defn:PD}.

\begin{theorem}[\cite{Fug83, BeDu06, BeDu07}]
\label{thm:complex} If $M$ is an infinite complex matrix indexed by
$\mathbb{N}_0 \times \mathbb{N}_0$  which satisfies Property PD$\mathbb{C}$, then there exists a positive Borel measure $\mu$ on $\mathbb{C}$
such that $M = M^{(\mu)}$.
\end{theorem}
 
\begin{remark}\label{Rem:RversusC}
  The condition in Theorem \ref{thm:real} that the matrix entries in $M$ are real numbers cannot be dropped.  To
see this, let $d=1$ and take $\xi$ be a fixed nonreal number.  Let
$M = ({\overline{\xi}}^j\xi^{k})_{j,k \in \mathbb{N}_0}$.  We see
that
\begin{eqnarray*} \langle c | Mc \rangle_{\ell^2}
&=& \sum_j \sum_k \overline{c}_j M_{j,k} c_k \\ &=& \left( \sum_j \overline{c}_j \overline{\xi}^j \right) \left(\sum_k c_k \xi^k \right) \\
&=& \left| \sum_j c_j \xi^j\right|^2 \\
&\geq & 0 \end{eqnarray*} for all $c  \in \mathcal{D}$. The
infinite matrix $M$ is therefore positive semidefinite.  Without
the condition that the components of $M$ be real, then, we would
conclude that there is a measure on $\mathbb{R}$ with $M$ as its
moment matrix.   We can also verify, however, that $M$ has
Property PD$\mathbb{C}$ from Theorem \ref{thm:complex}.   Let $c =
\{c_{i,j}\}$ be a doubly indexed sequence in $\mathcal{D}$ (i.e.
$c$ has only finitely many nonzero entries).  Then

\begin{eqnarray*} \sum_{i,j,k,\ell \in \mathbb{N}_0} \overline{c}_{i,j}M_{i+\ell,j+k} c_{k,\ell} &=&
\sum_{i,j,k,\ell \in \mathbb{N}_0}
\overline{c}_{i,j}\overline{\xi}^{i+\ell}\xi^{j+k} c_{k,\ell} \\
&=& \left(\overline{\sum_{i,j} c_{i,j}\overline{\xi}^j \xi^i
}\right) \left(\sum_{k,\ell} c_{k,\ell}\overline{\xi}^{\ell} \xi^k
\right) \\ &\geq& 0. \end{eqnarray*}

This leads to the conclusion that there is a measure on
$\mathbb{C}$ having moment matrix $M$.  In fact, given $\xi \in
\mathbb{C}$, we note that the Dirac measure $\delta_{\xi}$ at
$\xi$ has the matrix $M$ as its moment matrix, since the entries
of $M$ are the evaluation of the monomials $\overline{z}^jz^k$ at
$\xi$:
\begin{equation*} M_{j,k} = \int_{\mathbb{C}} \overline{z}^jz^k \mathrm{d}\delta_{\xi}(z) = \overline{\xi}^j
\xi^k.
\end{equation*}

It is shown in \cite{Fug83} that measures with compact support are
uniquely determined by their moment matrices.  The Dirac measure,
therefore, is the unique Borel measure on $\mathbb{C}$ having
moment matrix $M$.  If we have $\xi \in \mathbb{C}\setminus \mathbb(R)$, then even though $M$ satisfies the positive semidefinite condition from Definition \ref{Defn:PD}, there cannot
be a Borel measure $\mu$ on $\mathbb{R}$ such that $M =
M^{(\mu)}$.

In some sense, what we have shown is that there is a measure, but
its support is not restricted to $\mathbb{R}$.  One of the
classical moment questions is whether one can determine the
support of a measure $\mu$ from the geometric properties of its
moment matrix $M^{(\mu)}$. 
\end{remark}

  The simplest moment situation arises when the measure $\mu = \delta_{\xi}$ is a Dirac mass.  We will revisit this example frequently in the remainder of the Memoir.    Observe that the infinite moment matrix for $\delta_{\xi}$ has rank one.  We will later show that if $\mu$ is a finite convex combination of Dirac masses, then the associated moment matrix will be of finite rank.

\section{A Parthasarathy-Kolmogorov approach to the moment problem}\label{Subsec:Parth}

The traditional existence proofs of Theorems \ref{thm:real} and
\ref{thm:complex} are not constructive.  In this section, we carefully
explain how the Parthasarathy-Schmidt theorem, Theorem
\ref{Thm:ParKol} below, can be applied to the moment problems from
Theorems \ref{thm:real} and \ref{thm:complex}.  Theorem
\ref{Thm:OmegaHan} and Corollary \ref{Cor:Complex} are restatements of
the moment problem in the real and complex cases, respectively.  We
use Theorem \ref{Thm:ParKol} to produce the measures which appear in
Theorem \ref{Thm:OmegaHan} and Corollary \ref{Cor:Complex}.  In
Corollary \ref{Cor:Complex}, the condition PD$\mathbb{C}$ (Definition
\ref{Defn:PDC}) appears in a more natural fashion.  In addition, the
proof of Theorem \ref{Thm:OmegaHan} clearly shows why the Hankel
assumption for the infinite positive definite matrix $M$ is essential
in the real case.

The use of the Kolmogorov ideas is motivated by our applications
to iterated function systems (IFSs) which begin in Chapter
\ref{Sec:Exist}.  We will be interested in moments of IFS
measures.  A key tool in the analysis of these IFS measures will
be infinite product spaces, precisely such as those that arise in
the Parthasarathy-Schmidt construction.

To simplify the main idea, we carry out the details only in the
real case, and only for $d = 1$.  The reader will be able to
generalize to the remaining real cases in $\mathbb{R}^d$, $d > 1$.
We conclude with the application to the complex moment problem.

\begin{definition}
Let $S$ be a set, and let $M:S\times S \rightarrow \mathbb{C}$. We
say that $M$ is a \textit{positive semidefinite function} if
\begin{equation*}
\sum_{s \in S} \sum_{t \in S} \overline{c_s} M(s, t) c_t \geq 0
\end{equation*}
for all sequences $\{c_s\}_{s\in S}\in\mathcal{D}$.  (Recall this
means the sequences have only finitely many nonzero coordinates.)
In the real moment problem, the matrix $M$ is real, and we
restrict to sequences $\{c_s\}_{s\in S}$ with entries in
$\mathbb{R}$.
\end{definition}
\begin{theorem}\label{Thm:ParKol}\rm(\cite[Theorem 1.2]{PaSc72})
\it Suppose $S$ is a set, and suppose $M: S \times S \rightarrow
\mathbb{C}$ is a positive semidefinite function.  Then there is a
Hilbert space $\mathcal{H}$ and a function $X:S \rightarrow
\mathcal{H}$ such that $\mathcal{H} =
\overline{\mathrm{sp}}\{X(s)\,:\, s \in S\}$ and
\begin{equation}\label{Eqn:XInnerProd}
M(s,t) = \langle X(s) | X(t)\rangle_{\mathcal{H}}.
\end{equation}  Moreover, the
pair $(\mathcal{H},X)$ is unique up to unitary equivalence.
\end{theorem}

\begin{remark}
We outline here two choices for the pair $(\mathcal{H},X)$ in the
special case where $S = \mathbb{N}_0$ and $M$ is given by an
infinite positive semidefinite matrix.  The first pair is the
the one constructed in \cite{PaSc72} using Kolmogorov's extension
principle on an infinite product space.  This particular choice
will allow us in Theorem \ref{Thm:OmegaHan} to express concretely
the measure satisfying the moment problem in the special case
where $M$ is Hankel.

The second pair is formed by constructing a Hilbert space from a
quadratic form using the matrix $M$.  We will be using similar
completions later in the paper to work with moment matrices, so this
is a natural approach.  Note that Theorem \ref{Thm:ParKol} tells us that our two
Hilbert spaces are isometrically isomorphic.
\end{remark}

\noindent \textit{Proof $1$.}\;  Let $S$ be the natural numbers
$\mathbb{N}_0$ and let $\Omega$ be the set of all functions from
$\mathbb{N}_0$ into the one-point compactification
$\overline{\mathbb{R}}$ of $\mathbb{R}$: 
\[\Omega = \prod_{\mathbb{N}_0}\overline{\mathbb{R}}=(\overline{\mathbb{R}})^{\mathbb{N}_0}.\]
($\Omega$ must be compact
in order to use a Stone-Weierstrass argument in the construction.)

We now describe the Gaussian construction and its Kolmogorov
consistency.   For now, assume that $M$ is positive definite in
the strict sense.   Let $J=\{i_1, \ldots, i_p\}$ be a finite
subset of $\mathbb{N}_0$.  The positive definite function $M$
gives rise to a (strictly) positive definite matrix $M_J$ which is
formed by choosing elements from rows and columns in $M$ indexed
by $J$.
 Let $P_J$ be the Gaussian measure on
 $\Omega_J:=\overline{\mathbb{R}}^J$ with zero mean and covariance
 matrix $M_J$, such that
 $P_J$ has Radon-Nikodym derivative $f_J$ with respect to Lebesgue measure,
where $f_J$ is given by
\begin{equation}\label{Eqn:Gauss}
f_J(\omega_J) = \frac{1}{(\sqrt{2\pi})^p\sqrt{\det
M_J}}\exp\Bigl(-\frac{1}{2}\omega_J^t M_J^{-1}\omega_J \Bigr),
\end{equation}
where we denote $\omega_J = (\omega_{i_1}, \ldots, \omega_{i_p})$.

Note that Equation (\ref{Eqn:Gauss}) uses the strict positive
definite condition because the inverse $M_J^{-1}$ is required.  The density
in Equation (\ref{Eqn:Gauss}) is the standard multivariate normal
density for random variables $X(i_1), \ldots, X(i_p)$ with mean
$0$.

By Kolmogorov's extension theorem \cite[``Fundamental Theorem,''
p. 29]{Kol50}, in order for the family of measures $\{P_J\,:\, J
\, \textrm{finite}\}$ to define a measure $P$ on all of $\Omega$,
the measures $P_J$ must be consistent.  Suppose $J$ and $K$ are
both finite subsets of $\mathbb{N}_0$, where $J\subset K$. In this
context, \textit{consistency} means that if $g$ is a function on
$\Omega_J$ which is extended to a function $G$ on $\Omega_K$
depending only on the variables in $J$, then
\[
\int_{\Omega_J} g f_J d\omega_J = \int_{\Omega_K}Gf_Kd\omega_K.
\]
If we consider this problem in the coordinates which diagonalize
the matrix $M_K$, we see that the variables in $K\backslash J$
will integrate to $1$, and the measures $P_J$ are indeed
consistent. Note that consistency in one set of coordinates does
not imply consistency in another set of coordinates.

If $M$ is not strictly positive definite, then for any finite set
$J \subset \mathbb{N}_0$, we can change coordinates to diagonalize
 $M_J$ and write $\Omega_J = \mathrm{ker}(M_J) \oplus
\mathrm{ker}(M_J)^{\perp}$.  We take the measure $P_J$ to be the
Gaussian given by $f_J$ on $\mathrm{ker}(M_J)^{\perp}$ and
$\delta_0$, the distribution having mean $0$ and variance $0$, on
$\mathrm{ker}(M_J)$.  Then the covariance matrix of $P_J$ to
is exactly the diagonalized $M_J$, and Kolmogorov
consistency is maintained.

By the consistency condition, the measures $P_J$ defined on finite
subspaces of $\Omega$ can be extended to a measure $P$ on $\Omega$.
We now take the Hilbert space in Theorem \ref{Thm:ParKol} to be
$\mathcal{H} = L^2(\Omega, P)$ and the map $X: \mathbb{N}_0
\rightarrow \mathcal{H}$ to map $n \in \mathbb{N}_0$ to the
projection map onto the $n^{\mathrm{th}}$ coordinate of the element
$\omega\in\Omega$:
\[X(n)(\omega) = \omega(n).
\] Then we see that $\mathcal{H}$ is the closed linear span of the maps
$\{X(n)\,:\, n \in \mathbb{N}_0\}$ and \[ \langle X(n) | X(m)
\rangle_{\mathcal{H}} = \int_{\Omega} X(n)(\omega) X(m)(\omega)
\,\mathrm{d}P(\omega) = M(m,n). \]
We can now call $M$ the covariance function.

To show the uniqueness statement, suppose Hilbert spaces
$\mathcal{H}_1$ and $\mathcal{H}_2$ with corresponding functions
$X_1$ and $X_2$ satisfy $\langle X_i(s) |
X_i(t)\rangle_{\mathcal{H}_i} = M(s,t)$ and $H_i =
\overline{\mathrm{sp}}\{X_i(s)\,:\,s \in S\}$ for $i=1,2$. For
each $s \in S$, let $W$ be the linear operator such that $W(X_1(s)) =
X_2(s)$. $W$ defined this way is an operator, since if $\sum_i c_i
X_1(s_i) = 0$ then \begin{eqnarray*}  \left\|W\Big(\sum_i c_i
X_1(s_i)\Big)\right\|^2_{\mathcal{H}_2} &=& \left\langle W\Big(\sum_i c_i X_1(s_i)\Big) \Bigr|
W\Big(\sum_i c_i X_1(s_i)\Big) \right\rangle_{\mathcal{H}_2}\\& = &\sum_{i,j}
\overline{c_i}c_j \langle X_2(s_i)|X_2(s_j) \rangle_{\mathcal{H}_2}
\\&=&\sum_{i,j} \overline{c_i}c_j M(i,j) \\&=& \sum_{i,j} \overline{c_i}c_j \langle
X_1(s_i)|X_1(s_j) \rangle_{\mathcal{H}_1} \\ &=& \left\| \sum_i c_i X_1(s_i) \right\|^2_{\mathcal{H}_1} = 0. \end{eqnarray*} By
linearity and density, $W$ extends to a map from $\mathcal{H}_1$
onto $\mathcal{H}_2$, and by the association of the inner products
with the matrix $M$, we see that $W$ is unitary. \hfill $\Box$

\medskip

\noindent \textit{Proof $2$.}\;  As in the first proof, let $S =
\mathbb{N}_0$.  Let $\mathcal{D}$ be the set of all maps from $S$
to $\mathbb{C}$ (i.e. complex sequences) such that only finitely
many coordinates are nonzero.  Given the positive semidefinite
function $M$, define the sesquilinear form
\[ S(v_1,v_2) = \sum_{s \in \mathbb{N}_0} \sum_{t \in \mathbb{N}_0}
\overline{v_1(s)} M(s,t) v_2(t). \] This is clearly a sesquilinear
form and therefore yields a quadratic form $Q$ which is a seminorm
\[ Q(v) = \|v\|_M^2 = S(v,v). \] The set $ \mathrm{Null}_M = \{v \in
\mathcal{D}\,:\, Q(v) = 0 \}$ is a subspace of $\mathcal{D}$, so the
quotient space $\mathcal{D}/\mathrm{Null}_M$ is an inner product
space.  We complete this space to form a Hilbert space which we denote
$\mathcal{H}_Q$ to emphasize the dependence on the quadratic form $Q$
arising from the matrix $M$.  We note here that this construction of a
Hilbert space from a given positive semidefinite function $M$ is
unique up to unitary equivalence.  We will make use of this Hilbert
space completion of the quadratic form $Q$ again in Chapters
\ref{Ch:Kato} and \ref{Ch:Extensions}.

We next define the map $X: \mathbb{N}_0 \rightarrow
\mathcal{H}_Q$ by \[ [X(s)](t) = \delta_s(t) =
\left\{\begin{matrix} 1 & s=t\\ 0 & x \neq t \end{matrix} \right.
.\]  Clearly $\delta_s \in \mathcal{D} \subseteq \mathcal{H}_Q$
for all $s \in \mathbb{N}_0$, and in fact, $\mathcal{D}$ is the
linear span of $\{\delta_s\,:\, s \in \mathbb{N}_0\}$.  Therefore,
$\mathcal{H}_Q = \overline{\mathrm{sp}}\{X(s)\,:\, s \in
\mathbb{N}_0\}$.  We also see that
\begin{eqnarray*} \langle X(s)|X(t) \rangle &=& \sum_{u,v \in
\mathbb{N}_0} \overline{\delta_s(u)}M(u,v)\delta_t(v)\\ &=&
M(s,t). 
\end{eqnarray*}  \hfill $\Box$


\begin{remark} There are subtleties involved in these Hilbert space
  completions.  In the pre-Hilbert space (before the completion), we
  will typically be working with a space $\mathcal{D}$ of all finitely
  supported functions on $S$.  In other words, $\mathcal{D}$ is the
  space of all finite linear combinations from the orthonormal basis
  $\{\delta_s \}_{s \in S}$ for $\ell^2(S)$.  When the completion to a
  Hilbert space is done, $\mathcal{D}$ will be a dense linear subspace
  in the completion.  Here, ``dense'' refers to the norm being used.
  In Chapter \ref{Ch:Kato}, for example, the norm is a weighted
  $\ell^2$ norm.

  Caution: In these alternative Hilbert space completions, say
  $\mathcal{H}_Q$ above, explicit representations of the vectors in
  $\mathcal{H}_Q$ are often not transparent.  The useful geometric
  representations of limits of concretely given functions may be
  subtle and difficult.  Notions of boundary constructions reside in
  these completions.  In fact, this subtlety is quite typical when
  Hilbert space completions are used in mathematical physics problems.
  Such examples are seen in \cite{JoOl00, Jor00}, where symmetries
  result in separate positive definite quadratic forms, and hence two
  very different Hilbert space completions. \end{remark}


We are now ready to state our existence result for the real moment
problem, using the language of Kolmogorov and
Parthasarathy-Schmidt (the first proof above). Given a matrix $M$,
our goal is to find a measure $\mu$ such that $M$ is the moment
matrix for $\mu$. The previous theorem gives us a Hilbert space
and a map $(\mathcal{H},X)$, unique up to unitary equivalence,
which we can use under the right condition ($M$ a Hankel matrix)
to determine a solution to the moment problem. This solution will
not, in general, be a unique solution.

\begin{theorem}\label{Thm:OmegaHan} Let $M$ be an infinite matrix
satisfying the positive semidefinite condition (\ref{Eqn:PD}), and
let $M(0,0) = 1$. Let $\Omega_{Han}$ be the measurable subset of
$\Omega$ given by
\begin{equation}\label{Eqn:OmegaHan}
\Omega_{Han} = \{ \omega \in \Omega \;:\; \omega(k) =
[\omega(0)]^k \, \textrm{ for all } k \in \mathbb{N}_0\}.
\end{equation}
Then Kolmogorov's extension construction yields a probability
measure $P_{Han}$ on $\Omega_{Han}$ with $M$ as its moment matrix
if and only if $M$ satisfies the Hankel property from Definition
\ref{Defn:HankelN0d}.
\end{theorem}

\begin{proof} ($\Rightarrow$)  Suppose a measure $P_{Han}$ exists as stated.  Theorem \ref{Thm:ParKol} defines the map $X$ such that
$X(k)(\omega) = \omega(k) = [\omega(0)]^k$ for $\omega \in
\Omega_{Han}$.   For the covariance function we have
\begin{equation}
M(j,k) = \int_{\Omega_{Han}} X(j) X(k) \mathrm{d}P_{Han}.
\end{equation}
Using Equation (\ref{Eqn:OmegaHan}), we then find
\begin{eqnarray*}
M(j,k)
&=& \int_{\Omega_{Han}} [X(0)(\omega)]^j [X(0)(\omega)]^k\mathrm{d}P_{Han}(\omega)\\
&=& \int_{\mathbb{R}} x^jx^k \mathrm{d}(P_{Han} \circ X(0)^{-1})(x) \\
&=& \int_{\mathbb{R}} x^{j+k} \mathrm{d}\mu(x)
\end{eqnarray*}
where we define the measure by
\begin{equation}
\mu = P_{Han} \circ X(0)^{-1}.
\end{equation}
Since this formula shows that $M$ is a Hankel matrix, we have
proved one implication.

($\Leftarrow$) Conversely, if $M$ is a Hankel matrix,
then Kolmogorov consistency holds on the subset $\Omega_{Han}
\subset \Omega$ because $\Omega_{Han}$ is measurable.  By the
extension principle, we get a measure space
$(\Omega_{Han},P_{Han})$ such that the measure $\mu = P_{Han}
\circ X(0)^{-1}$ on $\mathbb{R}$ is a solution to the moment
problem for $M = M^{(\mu)}$.
\end{proof}

 There is an analogous result for the complex case, which we
describe briefly.  Let $M:\mathbb{N}_0 \times \mathbb{N}_0
\rightarrow \mathbb{C}$ be a function satisfying Property
PD$\mathbb{C}$, i.e.,
\begin{equation}
\sum_{i,j,k,\ell \in \mathbb{N}_0} \overline{c}_{i,j} M_{i+\ell,
j+k} c_{k,\ell} \geq 0
\end{equation}
for all doubly-indexed sequences $c = \{c_{ij}\}\in\mathcal{D}$.
Define an induced function $\widehat{M}:\mathbb{N}_0^2 \times
\mathbb{N}_0^2 \rightarrow \mathbb{C}$ by
\begin{equation}
\widehat{M}((i,j),(k,l)) = M(i+l,j+k).
\end{equation}
One can readily verify that $\widehat{M}$ satisfies the positive
definite condition given in Theorem \ref{Thm:ParKol}.

\begin{corollary}\label{Cor:Complex}
When Theorem \ref{Thm:ParKol} is applied to the function
$\widehat{M}$ above, we get a pair $(X,\mathcal{H})$ where the
Hilbert space may be taken to be $\mathcal{H} =
L^2(\mathbb{C},\mu)$, and the function can be given by $X(i,j) =
z^i \overline{z}^j$ for all $(i,j) \in \mathbb{N}_0^2$.
\end{corollary}

\begin{proof}  Given these choices of $\mathcal{H}$ and $X$, we
see that
\begin{equation*}
\langle X(i,j) | X(k,l) \rangle_{\mathcal{H}} = \int_{\mathbb{C}}
\overline{z^i \overline{z}^j} z^k \overline{z}^l\mathrm{d}\mu(z)
\end{equation*} by Theorem \ref{Thm:ParKol}.
Then,
\begin{equation*}
\begin{split}
& \int_{\mathbb{C}} \overline{z^{i+l}}z^{j+k} \mathrm{d}\mu(z)
=\langle z^{i+l} | z^{j+k} \rangle_{L^2(\mu)}\\
&= M(i+l,j+k) = \widehat{M}((i,j),(k,l)) .
\end{split}
\end{equation*}
This verifies our desired result and also shows that $\mu$ is a
solution to the complex moment problem for the matrix $M_{ij} =
M(i,j)$.
\end{proof}

\section{Examples}

\begin{example}\label{Ex:Laguerre-2} The measure $\mu = e^{-x}dx$.\end{example}
We showed in Example \ref{Ex:Laguerre} that the measure $e^{-x}dx$ on
the positive reals has moment matrix \[M^{(\mu)}_{i,j} = (i+j)!\] for
all $i,j \in \mathbb{N}_0$.  This measure does not have compact
support, but it is an example of a case where the solution to the
moment problem $M=M^{(\mu)}$ is unique.  We know this because the
Laguerre system of orthogonal polynomials is dense in $L^2(\mu,
\mathbb{R}^+)$.

Recall our earlier observation that this is also an example which
illustrates that infinite Hankel matrices cannot always be realized
directly by operators on $\ell^2$-sequence spaces.  However we will
show that an operator representation may be found after a certain
renormalization is introduced. This will be an example of a general
operator theoretic framework to be introduced in Chapter
\ref{Ch:Kato}.  \hfill $\Diamond$

We next describe a series of examples of measures having moments
involving the Catalan numbers.  We will revisit these examples in
Chapter \ref{Sec:Spectrum}, where we will be able to say more about
the operator properties of the moment matrices.

The $k^{\textrm{th}}$ Catalan number $C_k$ is defined 
\[C_k
:=\frac{1}{k+1}\binom{2k}{k} 
= \frac{(2k)!}{k!(k+1)!},
\]
where the first Catalan number, $C_0$, is $1$.  
We define $B_k$ to be
\[
B_k:=\binom{2k}{k}.
\]

The Catalan numbers satisfy the following relation:
\[C_{k+1} = \sum_{n=0}^k C_n C_{k-n}.\] There is an explicit formula
for the generating function associated with the Catalan numbers:
\begin{equation}\label{Eqn:CatGenFn}
G_{\textrm{Cat}}(x) 
= \sum_{k=0}^{\infty} C_k x^k
=\frac{1-\sqrt{1-4x}}{2x}.
\end{equation}
The radius of convergence for $G_{\textrm{Cat}}(x)$ is $\frac{1}{4}$.

The generating function $G_{\textrm{Bin}}(x)$ associated with the $B_k$'s has the same radius of convergence; in fact,
\begin{equation}
G_{\textrm{Cat}}(x) = \frac{1}{x}\int_0^x G_{\textrm{Bin}}(y)\mathrm{d}y,
\end{equation}
and
\begin{equation}
G_{\textrm{Bin}}(x) = \frac{2}{\sqrt{1-4x}}.
\end{equation}
The Hankel matrix $M^{\textrm{Cat}} = (C_{j+k})$ is positive definite because every principal submatrix has determinant $1$. 

\begin{example}\label{Ex:Wigner}Wigner's semicircle measure.\end{example}
The measure $\mathrm{d}\mu$ is given on $(-2, 2)$ by
\[ \mathrm{d}\mu(x) = \frac{\sqrt{4-x^2}}{2\pi}\mathrm{d}x\]
and is $0$ otherwise.
A simple calculation shows that
\[ m_{2k} = \int_{-2}^2 x^{2k}\mathrm{d}\mu(x) = C_k\]
and 
\[ m_{2k+1} = \int_{-2}^2 x^{2k+1}\mathrm{d}\mu(x) = 0.\]
\hfill$\Diamond$

\begin{example}\label{Ex:Secant}The secant measure.\end{example}
The measure $\mathrm{d}\mu$ is given on $(-2, 2)$ by
\[ \mathrm{d}\mu(x) = \frac{1}{\pi\sqrt{4-x^2}}\mathrm{d}x\]
and is $0$ otherwise.  In this case,
\[ m_{2k} = \int_{-2}^2 x^{2k}\mathrm{d}\mu(x) = B_k\]
and 
\[ m_{2k+1} = \int_{-2}^2 x^{2k+1}\mathrm{d}\mu(x) = 0.\]
\hfill$\Diamond$

\begin{example}\label{Ex:HalfSecant}The half-secant measure.\end{example}
The measure $\mathrm{d}\mu$ is given on $(0, 2)$ by
\[ \mathrm{d}\mu(x) = \frac{2}{\pi\sqrt{4-x^2}}\mathrm{d}x\]
and is $0$ otherwise.  In this example, both the even and odd moments are nonzero:
\[m_{2k} =\int_0^2 x^{2k} \mathrm{d}\mu(x) = B_k,\]
and
\[m_{2k+1} = \int_0^2 x^{2k+1} \mathrm{d}\mu(x) = \frac{4^{k+2}}{\pi(k+1)} \binom{2k+1}{k}^{-1} .\]
\hfill$\Diamond$

\section{Historical notes}\label{Subsec:history}

Our positive semidefinite functions $M$ in the form (\ref{Eqn:PD})
have a history in a variety of guises under the names \textit{positive
  semidefinite kernels}, \textit{positive hermitian matrices},
\textit{reproducing kernels}, or \textit{positive definite functions},
among others. In these settings, a positive semidefinite function is
just a function in two variables, or, equivalently, a function $M$
defined on $S \times S$ where $S$ is a set.

There is a broad literature covering many aspects of positive
semidefinite kernels, especially in the case when $S$ is a domain in a
continuous manifold; $S$ may even be a function space.  Function
spaces lead to the theory of reproducing kernels, which are also
called Bergman-Shiffer-Aronszajn kernels. See
\cite{Aro50},\cite{BeSc52} for a classical exposition and
\cite{BeRe84}, \cite{Dut04}, \cite{DuJo06} for a modern view.

If $S$ is a group or a semigroup, and if $M$ is a positive
semidefinite kernel, we will adopt additive notation ``$+$'' for the
operation in $S$ and restrict attention to the abelian case.  Of
special interest are the cases when the function $M$ has one of two
forms:
\begin{enumerate}[(i)]
\item $M(s,t) = F_1(s - t)$ for some function $F_1$ on $S$
\item $M(s,t) = F_2(s + t)$, $F_2$ a function on $S$.
\end{enumerate}
In the first case, we say that the function $F_1$ is \textit{positive
  semidefinite} on $S$. Equivalently, we say that $M$ is a
\textit{Toeplitz matrix}. There is a substantial theory of positive
semidefinite functions on groups (and semigroups), see e.g.,
\cite{HeRo79}, \cite{BeRe84}; positive semidefinite functions play a
central role in harmonic analysis.  

In the second case (ii), $M$ is often called a \textit{Hankel matrix}.  In the following chapters, we develop operator theoretic duality  techniques to study positive semidefinite kernels in the form of infinite Hankel matrices.
\begin{enumerate}
\item We define a (generaly unbounded) operator $F$ which maps
  sequence spaces $\ell^2$ or their weighted variants $\ell^2(w)$ into
  function spaces $L^2(\mu)$.  This crucial operator $F : \ell^2(w)
  \rightarrow L^2(\mu)$ is defined on the dense subspace $\mathcal{D}$
  of finite sequences in $\ell^2$ and sends a sequence into the
  generating function (or polynomial).
\item Given the Hilbert spaces $\ell^2(w)$ and $L^2(\mu)$, we
  introduce a duality with the use of adjoint operators.  If $F$ is a
  given operator from $\ell^2(w)$ to $L^2(\mu)$, its adjoint operator
  $F^*_w$ maps from $L^2(\mu)$ to $\ell^2(w)$. 
\end{enumerate}
As a byproduct of our operator analysis for the spaces $\ell^2(w)$ and
$L^2(\mu)$, we get operations on the two sides which are unitarily
equivalent.  The unitary equivalence is critical:  it allows us to
derive spectral properties of measures $\mu$ from matrix operations
with infinite matrices, and vice versa.

As a final historical note, we mention that because the operator $F$
associates elements of sequence spaces with generating functions (or
polynomials), we initially obtained some of our results with formal
power series in the spirit of Rota's umbral calculus \cite{RoRo78},
\cite{Rom84}.  We then found conditions (such as renormalizing
sequence spaces) which guaranteed not just formal convergence but
actual convergence.  In Rota's
umbral calculus, the ``umbra'' is a space of formal power series.  In
this Memoir, the ``umbra'' consists of the Hilbert spaces $\ell^2(w)$ and
$L^2(\mu)$ and the closed linear operators $F$ and $F^*_w$ between
them.  For more about the book \cite{Rom84} and the umbral calculus's
relation to other parts of mathematics, see the 2007 review by
Jorgensen \cite{JorAmazon}.

%% file: affine_08_09_11.tex
\chapter{A transformation of moment matrices: the affine case}\label{Sec:Exist}

In this chapter we use the maps from an iterated function system to describe a corresponding transformation on moment matrices.   In the case of an affine IFS, we show that the matrix transformation is a sum of triple products  of
infinite matrices.

Given a measurable endomorphism $\tau$ mapping a measure space
$(X,\mu)$ to itself, we make use of the measure transformation
\begin{equation}\label{Eqn:ExistMeasTrans}
\mu\mapsto\mu\circ\tau^{-1}.
\end{equation}
If
$(X,\mu)$ and $(X,\mu \circ \tau^{-1})$ have finite moments of all
orders, we can also study the transformation of moment matrices
 \begin{equation}\label{Eqn:ExistMxTrans}
M^{(\mu)}\mapsto M^{(\mu\circ\tau^{-1})}.
\end{equation}

We are interested in the circumstances under which this
transformation of moment matrices  (\ref{Eqn:ExistMxTrans}) can be
expressed explicitly as a well-defined intertwining of the form
\begin{equation}\label{Eqn:CovarianceFirst}
 M^{(\mu\circ\tau^{-1})} = A^*M^{(\mu)}A
\end{equation}
for some suitable infinite matrix $A$.  (We are using the notation
$A^*$ here to represent the conjugate transpose of the infinite
matrix $A$.)  We then ask whether there are conditions under which
$A$ might be an operator on the Hilbert space $\ell^2$.

We begin by stating the result from \cite{EST06} that the
matrix $A$ can be described explicitly in the case where $\tau$ is
an affine map of a single variable.  We give a proof of this
result in Section \ref{Sec:SingleVar} and then examine in
Section \ref{Subsec:FixedPointsHut} the connections of this
result to affine iterated function systems (IFSs) in
one dimension.  We also prove a
stronger uniqueness result for the fixed point of iterations of
the moment matrix transformation arising out of a Bernoulli IFS.
We examine in Section \ref{Subsec:Hankel} the types of
transformations that preserve infinite Hankel matrices, as a
precursor to Chapter \ref{Sec:ComputeA} in which we explore the
existence of matrices $A$ for more general measurable maps $\tau$.

\section{Affine maps}\label{Sec:SingleVar}

We consider $\mu$ a probability measure on $\mathbb{R}$ or
$\mathbb{C}$.  We will consider the case where $\tau$ is an affine
map, noting that a finite set $\{\tau_b\}$ of affine maps on
$\mathbb{R}$ or $\mathbb{C}$ can comprise an affine iterated
function system (IFS). Such IFSs are used in the study of infinite
convolution problems; see e.g., \cite{Erd39,JKS07a,JKS07b}. We
begin with some definitions.

Recall that the moment matrix $M^{(\mu)}$  of $\mu$ is the
infinite matrix defined by  $M^{(\mu)}_{i,j} = \int_{\mathbb{R}}
x^{i+j} d\mu(x)$ when $x$ is a real variable, and by
$M^{(\mu)}_{i,j} = \int_{\mathbb{C}}  \overline{z^i}z^j d\mu(z)$
when $z$ is a complex variable.  We naturally must assume that
moments of all orders exist. Certainly moments of all orders exist when $\mu$ has
compact support, but we will not always be restricted to compactly
supported measures.  Recall that the row and column
indexing of $M^{(\mu)}$ start at row $0$ and column $0$.

In the real case, we recall that every moment matrix $M^{(\mu)}$
is a Hankel matrix.   Our moment matrix transformation
corresponding to $\tau$ in this real setting must preserve the
Hankel property, then, since $M^{(\mu \circ \tau^{-1})}$ is also a
moment matrix.

The following result is stated in \cite{EST06}. Because one of our goals is to generalize the
lemma, we include a proof here for completeness.  This result
gives a matrix $A$ corresponding to an affine map $\tau$ on
$\mathbb{C}$, so that the moment transformation is an intertwining
by matrix multiplication. Note that the same matrix $A$ works in
the real case if we take $\tau$ to be an affine map of a single
real variable.

\begin{lemma}\label{Lemma:Amatrix}
\rm\cite[Proposition 1.1, p. 80]{EST06}  \it Suppose
$\tau:\mathbb{C}\rightarrow\mathbb{C}$ is an affine map (not
necessarily contractive) given by
\begin{equation*}
\tau(z) = cz + b,
\end{equation*}
where $c, b\in\mathbb{C}$.  Let 
$M^{(\mu\circ\tau^{-1})}$  be the moment matrix associated with the measure $\mu\circ\tau^{-1}$.  We have 
\begin{equation}
M^{(\mu\circ\tau^{-1})} = A^* M^{(\mu)} A,
\end{equation}
where $A = (a_{i,j})$ is the upper triangular matrix with
\begin{equation}\label{Eqn:UpTriA}
a_{i,j} =
\begin{cases}
\binom{j}{i}c^i b^{j-i} &  i \leq j\\
0 & \text{otherwise}
\end{cases}.
\end{equation}
\end{lemma}

\begin{proof} The
$(i,j)^{\mathrm{th}}$ entry of the moment matrix $M^{(\mu\circ\tau^{-1})}$
is given by
\begin{equation}
\begin{split}
M^{(\mu\circ\tau^{-1})}_{i,j}
& = \int_{\mathbb{C}} \overline{z^i} z^j \:d(\mu\circ\tau^{-1})(z)\\
& = \int_{\mathbb{C}} \overline{(c z +b)^i}(c z + b)^j \:d\mu(z).
\end{split}
\end{equation}
Even though $A$ and $M^{(\mu)}$ are infinite matrices, the sum
defining the $(i,j)^{\mathrm{th}}$ entry of the triple product $A^*
M^{(\mu)} A$ is finite because $A$ is upper triangular.  With this
observation, the matrix product is well defined and we can compute
the $(i,j)^{\mathrm{th}}$ entry of the product $A^*M^{(\mu)}A$ by
\begin{equation*}
\begin{split}
(A^* M^{(\mu)} A)_{i,j}
& = \sum_{k = 0}^i \sum_{\ell = 0}^j A^*_{i,k} M^{(\mu)}_{k,\ell} A_{\ell,j}\\
& = \sum_{k = 0}^i \sum_{\ell = 0}^j \binom{i}{k}\overline{c^kb^{i-k}} \Biggl(\int_{\mathbb{C}} \overline{z^k} z^{\ell}\:d\mu(z)\Biggr)
\binom{j}{\ell} c^{\ell}b^{j-\ell}\\
& = \sum_{k=0}^i \binom{i}{k}
\overline{c^k}\overline{b^{i-k}}\sum_{\ell = 0}^j
\binom{j}{\ell}c^{\ell} b^{j-\ell} \Biggl(\int_{\mathbb{C}} \overline{z^k}
z^{\ell}\:d\mu(z)\Biggr).
\end{split}
\end{equation*}
Taking advantage of the linearity of the integral, we compute the
sums in $k$ and $\ell$ to obtain
\begin{equation*}
\begin{split} (A^* M^{(\mu)} A)_{i,j} & = \int_{\mathbb{C}} \sum_{k=0}^i \binom{i}{k}
\overline{c^k}\overline{z^k}\overline{b^{i-k}} \sum_{\ell = 0}^j
\binom{j}{\ell}c^{\ell}z^{\ell} b^{j-\ell} \: d\mu(z) \\ & = \int_{\mathbb{C}}
\overline{(cz+b)^i} (cz+b)^j \:d\mu(z).
\end{split}
\end{equation*}
This gives us our desired result.\end{proof}

\section{IFSs and fixed points of the Hutchinson operator}\label{Subsec:FixedPointsHut}
An iterated function system has an associated compact set $X$ and a measure $\mu$ supported on $X$, where both $X$ and $\mu$ arise as unique solution to a fixed-point problem described in the following paragraph.   A central theme in this Memoir is that every IFS in the classical sense corresponds to a non-abelian system of operators and a version of Equations (\ref{Eqn:SetInvariance}) and (\ref{Eqn:TransformedMu}) in which the moment matrix $M = M^{(\mu)}$ will be a solution to an associated fixed-point problem for moment matrices.  (See Proposition \ref{Prop:MatrixFixedPt} and Sections \ref{Sec:GeneralA} and \ref{Subsec:KatoA}.)  One of the consequences is that we will be able to use the formula for $M$ in a recursive computation of the moments, which is not easy to do directly for even the simplest Cantor measures.

Let $I$ be a finite index set.  An \textit{iterated function
system} (IFS) is a finite collection $\{\tau_i\}_{i \in I}$ of
contractive transformations on $\mathbb{R^d}$ and a set of
probabilities $\{p_i\}_{i \in I}$. Using Banach's fixed point
theorem and the Hausdorff metric topology, it is proved in Hutchinson's paper
\cite{Hut81} that there is a unique compact subset $X$ of
$\mathbb{R}^d$, called the \textit{attractor} of the IFS, which
satisfies the equation
\begin{equation}\label{Eqn:SetInvariance} X = \bigcup_{i \in I}\tau_i(X).
\end{equation}
There is also a unique measure $\mu$ supported on $X$ which arises from
Banach's theorem.

\begin{theorem}[Hutchinson, \cite{Hut81}]
Given a contractive IFS $\{\tau_i\}_{i \in I}$ in $\mathbb{R}^d$
and probabilities $\{p_i\}$,  there is a unique Borel probability
measure $\mu$ on $\mathbb{R}^d$ satisfying
\begin{equation}\label{Eqn:TransformedMu} \mu = \sum_{i \in I} p_i
(\mu \circ\tau_i^{-1}). \end{equation} The measure $\mu$ is called
an \textit{equilibrium measure}. Moreover if $p_i
> 0$ for all $i \in I$, the support of $\mu$ is the unique compact
attractor $X$ for $\{\tau_i\}_{i \in I}$.
\end{theorem}

First, let us examine an IFS on the real line $\mathbb{R}$
consisting of contractive maps.  We will find that the moment
matrix for the equilibrium measure of the IFS also satisfies an
invariance property corresponding to Equations
(\ref{Eqn:SetInvariance}) and (\ref{Eqn:TransformedMu}).

 Let $I$ be a finite index set.  Let a contractive IFS
on $\mathbb{R}$ be given by the maps $\{\tau_i\}_{i \in I}$ and the
nonzero probability weights $\{p_i \}_{i \in I}$.  Suppose $\mu$ is
the invariant Hutchinson measure  associated with this IFS.   Then $\mu$ is a probability measure satisfying  (\ref{Eqn:TransformedMu}) which is supported on the attractor set $X$. 

We now observe, as also noted in \cite{EST06}, that the moment matrix for the equilibrium measure
$\mu$ of an IFS on the real line also satisfies an invariance
property $\mathcal{R}(M^{(\mu)}) = M^{(\mu)}$ under the
transformation
\begin{equation} \mathcal{R}: M^{(\nu)} \mapsto M^{(\sum_{i \in I} p_i (\nu \circ \tau_i^{-1}))}. \end{equation}
In other words, $M^{(\mu)}$ is a fixed point of the transformation
$\mathcal{R}$.   We now state a uniqueness result.

\begin{proposition}\label{Prop:MatrixFixedPt}  Given a
 contractive IFS $\{\tau_i\}_{i \in I}$ on
$\mathbb{R}$ and probability weights $\{p_i\}_{i \in I}$ with $p_i
\in (0,1)$ and $\sum_{i \in I}p_i = 1$, there exists a unique
moment matrix $M^{(\mu)}$ for a probability measure $\mu$ which
satisfies
\begin{equation}
\mathcal{R}(M^{(\mu)}) = M^{(\mu)}.
\end{equation}
  This unique solution  is exactly
the moment matrix for the Hutchinson equilibrium measure for the
IFS.  Furthermore, for any Borel regular probability measure $\nu$
supported on the attractor set  $X$ of the IFS,
$\mathcal{R}^n(M^{(\nu)})$ converges componentwise to $M^{(\mu)}$.
\end{proposition}

\begin{proof}
We begin by proving the convergence of the moments.  Let $S$ be
the transformation of measures supported on $X$ given by
\begin{equation} S(\nu) = \sum_{i \in I} p_i (\nu \circ \tau_i^{-1}).
\end{equation}
From \cite{Hut81}, we know that $S(\nu)$ maps probability measures
to probability measures, and that for any Borel regular
probability measure $\nu$ supported on $X$, $S^n\nu \rightarrow
\mu$ as $n \rightarrow \infty$, where convergence is in the metric
$\rho$ which Hutchinson calls the $L$-metric on the space of
measures \cite{Hut81}.

To define the $L$-metric, first recall the space of Lipschitz functions $\textrm{Lip}(X,\mathbb{R})$ on a metric space $X$ with metric $d$.  $\textrm{Lip}(X,\mathbb{R})$ consists of all functions $\phi:X \rightarrow \mathbb{R}$ such that there exists a constant $C=C_{\phi}< \infty$  with 
\begin{equation}\label{Eqn:Lip} |\phi(x)-\phi(y)| \leq C_{\phi} d(x,y) \; \textrm{ for all }\; x,y \in X.\end{equation}  We then define the  constant  \begin{equation}
L(\phi) = \inf \{C_{\phi} \,|\,  C_{\phi} \textrm{ satisfies (\ref{Eqn:Lip})} \},  \end{equation} and hence the space of functions \[ \mathrm{Lip}_1 = \{ \phi \in \mathrm{Lip}(X,\mathbb{R})\,|\, L(\phi) \leq 1 \}. \]  We can now define the $L$-metric:
\begin{equation}\label{def:Lmetric} \rho(\mu,\nu) = \sup \left\{\int_X
\phi \, \mathrm{d}\mu - \int_X \phi \, \mathrm{d}\nu \,:\, \phi
\in  \mathrm{Lip}_1(X,\mathbb{R}) \right\}.
\end{equation}
  When the
measures are defined on compact spaces, the $L$-metric topology is
equivalent to a weak topology, and therefore we have $$\int_X f
\mathrm{d}(S^n\nu) \rightarrow \int_X f \mathrm{d}\mu $$ for all
functions $f:X \rightarrow \mathbb{R}$ which are bounded on
bounded sets  \cite{Hut81}.

Let $f(x) = x^{i+j}$, which is bounded on bounded sets for all
choices of $i,j \in \mathbb{N}_0$.  Inserting $f$ into the
convergence above yields the componentwise convergence of the
moment matrix for $S^n\nu$ to the moment matrix for $\mu$.

Given that $\mu$ satisfies the invariance property $$ S(\mu) =
\mu,$$ we next wish to conclude that $\mathcal{R}(M^{(\mu)}) =
M^{(\mu)}$.    We defined the transformation $\mathcal{R}$ by
\[\mathcal{R}(M^{(\nu)}) = M^{(\sum_i p_i \nu \circ \tau_i^{-1})} =
M^{S\nu}.\] Since $S \mu = \mu$, it is clear that $S\mu$ and $\mu$ have the same
moment matrices. Therefore $\mathcal{R}(M^{(\mu)}) = M^{(\mu)}$.

It remains to be shown that the matrix $M^{(\mu)}$ is the unique
solution among moment matrices for probability measures on $X$ to
the equation $\mathcal{R}(M) = M$. We have from \cite{Hut81} that
the measure $\mu$ is the unique probability measure satisfying
$S\mu = \mu$. Suppose there exists another matrix $M^{(\nu)}$
which is a moment matrix for a probability measure $\nu$ and which
is a fixed point for $\mathcal{R}$.  Since
$\mathcal{R}(M^{(\nu)}) = M^{(S\nu)}$, the measures
$\nu$ and $S\nu$ have the same moments of all orders.  By
Stone-Weierstrass, since the measures have compact support and
agree on the moments (hence the polynomials), the
measures must be the same.  Since $S\nu = \nu$, we must have $\nu =
\mu$.
\end{proof}

\begin{corollary}\label{Cor:InfUnique}  Let $\{\tau_i\}_{i \in I}$ be an affine
contractive IFS on $\mathbb{R}$, such that $\tau_i(x) = c_ix+b_i$.
Let $\{p_i\}_{i \in I}$ be probability weights and let $A_i$ be
the matrix given in Lemma \ref{Lemma:Amatrix} which encodes
$\tau_i$ for each $i \in I$. The moment matrix transformation
\begin{equation}\label{Eqn:AffineTrans} M^{(\nu)} \mapsto \sum_{i \in I} p_i A_i^*M^{(\nu)}A_i \end{equation} has a
unique fixed point among moment matrices. Moreover, the fixed
point is the moment matrix $M^{(\mu)}$ for the Hutchinson
equilibrium measure of the IFS.
\end{corollary}

Given an affine IFS  and its associated moment matrix
transformation, we can state a uniqueness result for finite
matrices which corresponds to and extends Corollary
\ref{Cor:InfUnique}.  Given any infinite matrix $M$, denote by
$M_n$ the $(n+1) \times (n+1)$ matrix which is the upper left
truncation $[M_{ij}]_{i,j = 0, 1, \ldots, n}$ of $M$.  Let
$\{\tau_i\}_{i=0}^k$ be a contractive affine IFS on $\mathbb{R}$,
and let $\mu$ be the unique Hutchinson measure corresponding to
the IFS and probability weights $\{p_i\}_{i=0}^k$. Note that we
have $\tau_i(x) = c_ix+b_i$, where $c_i<1$ for each $i=0,1,\ldots,
k$ in the most general affine IFS.

We know from Corollary \ref{Cor:InfUnique} that  $M^{(\mu)}$ is
the unique infinite matrix which is fixed under the transformation
in Equation (\ref{Eqn:AffineTrans}).  We now give the result  that
the truncated form of the transformation in Equation
(\ref{Eqn:AffineTrans}) has the truncated moment matrix as a fixed
point, and moreover, that it is a unique fixed point among Hankel
matrices.  We remark here that some of the computations required for this
proof are done in \cite{EST06}, but they don't use them to state
this uniqueness result.

We begin by observing that, because the matrices $A_i$ which
encode affine maps are triangular (recall Lemma
\ref{Lemma:Amatrix}), a truncation of the matrix product $ A_i^*MA_i$ is exactly equal to the product of the truncated matrices  $(A_i)_n^*M_n(A_i)_n$.

\begin{lemma}\label{Lem:Truncation}  Given infinite matrices $A,M$ such that $A$ is upper triangular,
\begin{equation} (A^*MA)_n = A^*_nM_nA_n. \end{equation}
\end{lemma}

\begin{proof}  This follows immediately by the properties of block
matrices assuming that the associated block products are well defined.  In the matrices below, $B$ and $D$ are upper triangular:   $$ \left[ \begin{matrix} B^* & 0\\C^* & D^*
\end{matrix} \right] \left[ \begin{matrix} E&F\\G&H \end{matrix}
\right]\left[
\begin{matrix} B&C\\0&D \end{matrix} \right] = \left[
\begin{matrix}  B^*EB & * \\ * & * \end{matrix} \right]. $$

\end{proof}

Given the affine IFS $\{\tau_i\}_{i=0}^k, \{p_i\}_{i=0}^k$ mentioned above, let $A_i$ be the matrix given  by Lemma \ref{Lemma:Amatrix} which encodes each affine map $\tau_i(x) = c_ix+b_i$.   Let $\mu$ be the unique Hutchinson equilibrium measure for this IFS.  Using our truncation notation, define the matrix transformation $\mathcal{R}_n$ on $(n+1) \times (n+1)$ matrices $M$ by
\begin{equation}\label{Eq:finitetransform} \mathcal{R}_n(M) = \sum_{i=0}^k p_i (A_i^*)_n M (A_i)_n . \end{equation}   We then have the following result.

\begin{proposition}\label{Prop:finite}  For each $n \in \mathbb{N}$, $M = M^{(\mu)}_n$ is the unique Hankel matrix having $M_{0,0}=1$ which is a fixed point  for $\mathcal{R}_n$.
\end{proposition}

\begin{proof}
 We know from Corollary \ref{Cor:InfUnique} and Lemma \ref{Lem:Truncation} that for each $n$, $M^{(\mu)}_n$ is a fixed point of $\mathcal{R}_n$.  We need only to prove the uniqueness.

 For $n=1$, assume $M$ is a $2 \times 2$ Hankel matrix with $M_{0,0}=1$ such that $\mathcal{R}_1(M) = M$.  $M$ is of the form $$M = \left[ \begin{matrix} 1&x\\x&y \end{matrix} \right].$$  Then,
 \begin{eqnarray*} \mathcal{R}_1(M) &=& \sum_{i=0}^k p_i \left[ \begin{matrix} 1&0\\b_i&c_i \end{matrix} \right] \left[ \begin{matrix} 1&x\\x&y \end{matrix} \right] \left[ \begin{matrix} 1&b_i\\0&c_i \end{matrix} \right] \\ &=& \sum_{i=0}^k p_i \left[ \begin{matrix} 1&b_i + c_ix\\b_i +c_ix& b_i^2 + 2b_ic_ix+c_i^2 \end{matrix} \right] \\ &=& \left[ \begin{matrix} 1&x\sum p_ic_i + \sum p_ib_i \\x\sum p_ic_i + \sum p_ib_i& y\sum p_ic_i^2 + 2x\sum p_ib_ic_i + \sum p_ib_i^2 \end{matrix} \right] \\ &=& \left[ \begin{matrix} 1&x\\x&y \end{matrix} \right] = M.
 \end{eqnarray*}

 There is a unique solution for $x$ and $y$ in the equations
 $x = x\sum p_ic_i + \sum p_ib_i$ and $y= y\sum p_ic_i^2 + 2x\sum p_ib_ic_i + \sum p_ib_i^2$, thus $M = M^{(\mu)}_1$.

 Assume next that $M^{(\mu)}_{n-1}$ is the unique $n \times n$ Hankel matrix which is a fixed point of $\mathcal{R}_{n-1}$.  Take $M$ to be an $(n+1) \times (n+1)$ Hankel matrix with $M_{0,0}=1$ such that $\mathcal{R}_n(M) = M$.  By Lemma \ref{Lem:Truncation} and our hypothesis, $M$ is of the form
 $$ M = \left[ \begin{matrix} M^{(\mu)}_{n-1} &  B\\ B^{tr} & y \end{matrix} \right], $$  where $B$
  is  $n \times 1$.  Due to the Hankel structure of $M$, all but one of the entries in $B$ are known from $M^{(\mu)}_n$:
 $$  B = \left[ \begin{matrix} m_n \\ \vdots \\m_{2n-2} \\ x \end{matrix} \right].$$

 It is straightforward to verify that, as in the $n=1$ case, the terms in the matrix equation $S(M)=M$ yield two
 linear equations in $x$ and $y$ which have unique solutions.  Therefore $M = M^{(\mu)}_n$.
\end{proof}

The computations in Proposition \ref{Prop:finite} to solve for $x$
and $y$ give a recursive equation to compute the moments of an
equilibrium measure in terms of the previous moments.

\begin{corollary}[\cite{EST06}] Given the IFS as above with equilibrium measure $\mu$,  $$M^{(\mu)}_{m,n} = \frac{1}{1-\sum_{i=0}^k p_ic_i^2} \sum_{i=0}^k p_i
\sum_{j=0,\ell=0, (j,\ell) \neq (m,n)}^{m,n}
\binom{m}{j}\binom{n}{\ell} b_i^{m+n-j-\ell}c_i^{j+\ell}
M^{(\mu)}_{j,\ell} $$
\end{corollary}

\section{Preserving Hankel matrix structure}\label{Subsec:Hankel}
When we work with affine IFSs for a real variable $x$, the moment
matrix of a probability measure is a Hankel matrix whose
$(0,0)^{\mathrm{th}}$ entry is $1$.  We will examine here the sorts of
matrix transformations of the form $$M \mapsto A^*MA$$ which
preserve Hankel structure, so that in the next section we can look
for moment matrix transformations corresponding to more general
maps $\tau$. Let $\mathcal{H}^{(1)}$ denote the set of positive
definite Hankel matrices whose $(0,0)^{\mathrm{th}}$ entry is $1$.

 We have already shown that the
transformation $\mathcal{R}$ corresponding to an affine IFS given
by
\begin{equation}
\mathcal{R}(M) = \sum_{i\in I} p_i (A_i^{*} M A_i) = \sum_{b\in
B}p_b M^{(\mu\circ\tau_b^{-1})}
\end{equation}
maps a moment matrix $M^{(\mu)}$ to another moment matrix
$M^{(\sum_{i \in I}p_i \mu \circ \tau_i^{-1})}$, and therefore
maps $\mathcal{H}^{(1)}$ into itself.

\begin{definition}
We say that the matrix $A$ \textit{preserves $\mathcal{H}^{(1)}$} if
$A^{*}MA \in \mathcal{H}^{(1)}$ for all matrices
$M\in\mathcal{H}^{(1)}$.
\end{definition}

In the following lemma, we refer to inverses of infinite matrices
and to invertible matrices.  A careful definition of these
concepts is given in Definition \ref{def:inverse}.

\begin{lemma}\label{Lemma:DGPreserveH1}
The following matrices preserve $\mathcal{H}^{(1)}$:
\begin{equation}
D(\delta)=
\begin{bmatrix}
1 & 0        & 0            & 0             & \cdots \\
0 & \delta & 0            & 0             & \cdots \\
0 & 0        & \delta^2 & 0             & \cdots \\
0 & 0        & 0             & \delta^3 &  \\
\vdots & \vdots & \vdots      &                & \ddots \\
\end{bmatrix}
\quad
\text{ and }
\quad
G(\gamma) =
\begin{bmatrix}
1 & \gamma & \gamma^2  & \gamma^3    & \cdots\\
0 & 1             & 2\gamma    & 3\gamma^2 & \cdots\\
0 & 0             & 1                  & 3\gamma     & \cdots \\
0 & 0             & 0                  & 1                   &  \\
\vdots & \vdots & \vdots      &                & \ddots \\
\end{bmatrix}.
\end{equation}
The inverse of $D(\delta)$ is $D(\delta^{-1})$ ($\delta\neq 0$) and the inverse of $G(\gamma)$ is $G(-\gamma)$.
\end{lemma}

We are pleased to thank Christopher French for the proof of the
following proposition.  See also \cite{Fre07}, which implicity
uses Proposition \ref{Prop:French} throughout, and the papers
\cite{SpSt06} and \cite{Lay01}.

\begin{proposition}\label{Prop:French}
Suppose $A=(a_{i,j})$ is an infinite upper triangular, invertible
matrix which preserves $\mathcal{H}^{(1)}$.  Then either $A$ or
$-A$ is the product of matrices of the form $D(\delta)$ and
$G(\gamma)$, where $D$ and $G$ are defined in Lemma
\ref{Lemma:DGPreserveH1}.
\end{proposition}
\begin{proof} If $A^{*}MA \in\mathcal{H}^{(1)}$ for all $M\in
\mathcal{H}^{(1)}$, then $a_{0,0}^2 = 1$, so $a_{0,0} = \pm 1$. If
$a_{0,0} = -1$, replace $A$ with $-A$.

Denote $A$ by
\begin{equation*}
A = \begin{bmatrix}
  1        & a_{0,1} & a_{0,2} & \cdots \\
  0        & a_{1,1} & a_{1,2} & \cdots \\
  0        &      0     & a_{2,2} & \cdots \\
\vdots &  \vdots &              & \ddots\\
\end{bmatrix}.
\end{equation*}

Since $a_{1,1}\neq 0$, multiply $A$ on the right by
$D(1/a_{1,1})$. The upper $2\times 2$ principal submatrix of
$AD(\delta)$ is
\begin{equation*}
\begin{bmatrix}
1 & a_{0,1}/a_{1,1}\\
0 & 1\\
\end{bmatrix}.
\end{equation*}
Now set $\gamma = -a_{0,1}/a_{1,1}$; the matrix
$AD(\delta)G(\gamma)$ now has the $2\times 2$ identity matrix as
its upper left principal submatrix.

Now set $\tilde{A} = AD(1/a_{1,1})G(-a_{0,1}/a_{1,1})$; we know
that $\tilde{A}$ preserves $\mathcal{H}^{(1)}$ by Lemma
\ref{Lemma:DGPreserveH1}. We will show that $\tilde{A}$ is the
infinite identity matrix.   We start with the $k=2$ case (instead
of $k=1$) to give more intuition.

Let $\tilde{A}$ be denoted
\begin{equation}
\tilde{A}=
\begin{bmatrix}
1 & 0 & a_0& \cdots\\
0 & 1  & a_1& \cdots\\
0 & 0 & a_2 & \cdots \\
\vdots & \vdots & &\ddots\\
\end{bmatrix}
\end{equation}
Consider the upper left $3\times 3$ principal submatrix of the
matrix $\tilde{A}^{*}M\tilde{A}$:
\begin{equation}
\begin{bmatrix}
1 & m_1 & a_0+a_1m_1 +a_2m_2\\
m_1 & m_2 &              *\\
a_0+a_1m_1 +a_2m_2 & * & * \\
\end{bmatrix}.
\end{equation}
Since $\tilde{A}^{*}M\tilde{A}$ belongs to $\mathcal{H}^{(1)}$ for
all $M\in\mathcal{H}^{(1)}$, we know that the above matrix is
Hankel. We can treat $m_1$ and $m_2$ as independent variables, so
$a_0 = a_1 = 0$ and $a_2 = 1$. Therefore $\tilde{A}$ actually has
an upper $3\times 3$ principal submatrix which is the identity.

Continuing inductively, suppose the upper $k\times k$ principal
submatrix of $\tilde{A}$ is the $k\times k$ identity matrix and
$\tilde{A}$ has the form
\begin{equation}
\tilde{A} =
\begin{bmatrix}
1 & 0 & 0 & \cdots & 0 & a_0& * &\cdots\\
0 & 1 & 0 & \cdots & 0& a_1& *& \cdots \\
\vdots     & & & & &   \vdots &  & \\
0 & 0 & 0 & \cdots & 1& a_{k-1} & *& \cdots\\
0 & 0 & 0  & \cdots & 0 & a_{k}& *&\cdots\\
0 & 0 & 0 & \cdots & 0 & 0 & *  &\cdots \\
\vdots & \vdots & \vdots & & \vdots &\vdots & &\ddots
\end{bmatrix}.
\end{equation}  Multiplying $\tilde{A}^{*}M\tilde{A}$, we find two expressions for the $(k,0)^{\mathrm{th}}$ entry:
\begin{equation}
m_{k} = \sum_{i=0}^{k} a_i m_i.
\end{equation}
Since this equation holds for all choices of $m_1, \ldots m_k$, we must have that $a_0 = \cdots = a_{k-1} = 0$ and $a_k = 1$.

Therefore, $\tilde{A} = AD(1/a_{1,1})G(-a_{0,1}/a_{1,1}) = I$, so
we can write
\begin{equation}
A = G(a_{0,1}/a_{1,1})D(a_{1,1}).
\end{equation}
\end{proof}

%% file: measurable_08_09_11.tex
\chapter{Moment matrix transformation:  measurable
maps}\label{Sec:ComputeA}   In the previous chapter, we focused on the case of affine transformations comprising IFSs, but there is a great deal of interest in concrete applications to nonaffine examples; e.g. real and complex Julia sets.  We turn our attention to nonaffine measurable maps in this chapter.   We begin with an
arbitrary measurable map $\tau$ on a measure space $(X,\mu)$, where $X$ is a
subset of $\mathbb{R}^d$ for $d \geq 1$ or $\mathbb{C}$, and the moments of all orders with respect to $\mu$ and $\mu \circ \tau^{-1}$ are finite.  We ask
whether the transformation of moment matrices $M^{(\mu)} \mapsto
M^{(\mu \circ \tau^{-1})}$ can be expressed as a matrix triple
product $M^{(\mu \circ \tau^{-1})} = A^*M^{(\mu)}A$, for $A$ an
infinite matrix. In the previous section, we stated the
appropriate $A$ when $\tau$ is an affine map.  We now seek to find
an intertwining matrix $A$ for more general maps $\tau$.  

We find that $A$, when it can be
written down, need not be a triangular matrix.  We therefore also
will need to examine the hypotheses under which the formal matrix
products we write down are actually well defined.  


\section{Encoding matrix $A$ for $\tau$}\label{Sec:GeneralA}
The following analysis will be restricted to the real one-dimensional case.
We let $X \subset \mathbb{R}$ be a Borel subset and let $\mu$ be a Borel measure with moments of all orders on $X$.  Let $\tau$ be a measurable map on $X$, so that $\mu \circ \tau^{-1}$ also is a measure on $X$ with moments of all orders.     We seek to find an infinite matrix $A$ which enacts the moment matrix transformation from $M^{(\mu)}$ to $M^{(\mu \circ \tau^{-1})}$; more precisely, such that
\begin{equation}\label{eqn:Atau} M^{(\mu \circ \tau^{-1})} = A^* M^{(\mu)} A.
\end{equation}

In the space $L^2(\mu)$, let $\mathcal{P}$ be the closed linear
span of the monomials.  (The monomials are $L^2$ functions since
all moments are finite for $\mu$.)   The following results hinge
on a careful description of the Gram-Schmidt process on the
monomials, since the monomials could possibly have linear
dependence relations among them in $L^2(\mu)$ (See Example \ref{Ex:dependent} below).  Let $\{v_j\}_{j
\in \mathbb{N}_0}$ be the set of monomial functions, but if $x^j$
is dependent on $1,x,x^2, \ldots, x^{j-1}$, we remove $x^j$ from
the collection.  In other words, $v_j$ might not be $x^j$, but it
is a monomial function $x^k$ for some $k \geq j$, and there are no
finite dependence relations among the set
$\{v_j\}_{j\in \mathbb{N}_0}$.    

\begin{example}\label{Ex:dependent} Bernoulli IFSs which yield measures with linearly dependent monomials. \end{example}
Consider the affine IFS on $\mathbb{R}$ of  the form \[ \tau_0(x) = \lambda x,  \quad
\tau_1(x) = \lambda(x + 2).\]  The properties of the Hutchinson equilibrium measure depend on the choice of the parameter $\lambda$.  If $\lambda < \frac12$, the measure is supported on a generalized Cantor set, while when $\lambda \geq \frac12$ the measure is supported on an interval in the real line.

We can transform these into measures supported on subsets of the unit circle $\mathbb{T} \subset \mathbb{C}$ via the map $x \mapsto e^{2 \pi i x}$.  By \cite{JoPe98}, when $\lambda = \frac14$, there is an orthonormal basis of monomials \[ \{z^k\,:\, k= 0,1, 4, 5, 16, 17, 20, 21, \ldots\} \] for the corresponding measure contained in $ \mathbb{T}$.   The collection of all the monomials, however, does not have finite linear dependence relations, as described in Remark \ref{Rem:Cantor}.  On the other hand, if $\lambda = \frac13$, the corresponding circle measure does have finite linear dependence relations among the monomials.  These cases also arise when $\lambda > \frac12$ \cite{JKS07b}.
\hfill $\Diamond$
 
 \begin{remark}\label{Rem:Cantor}
A central theme in our study of moments in Chapters \ref{Sec:Exist} and \ref{Sec:ComputeA} is how the
study of moments relates to the self-similarity which characterizes
equilibrium measures for IFSs.  Because we use operator-theoretic methods
to study moment matrices, we encounter spectral information along the
way.  Historically (see \cite{Akh65}) the approach to spectral theory of
moments in $\mathbb{R}$ went as follows:
\begin{enumerate}
\item Start with the monomials $\{x^k\}_{k \in \mathbb{N}_0}$ viewed as a dense subset of (a subspace of)
$L^2(\mu)$.
\item Apply Gram-Schmidt to obtain the associated orthonormal polynomial
basis $\{p_k\}_{k \in \mathbb{N}_0}$.
\item The tri-diagonal matrix $J$ representing multiplication by $x$ in the ONB
$\{p_k\}_{k \in \mathbb{N}_0}$ (see Chapter \ref{Ch:Extensions}) gives rise to spectral information.
\end{enumerate}
However, the classical approach does not take into account the
self-similar properties that $\mu$ may have.  Following \cite{JoPe98} and
\cite{Jor06}, we can seek instead to encode the IFS structure directly into
the analysis of moments.  A case in point is the Cantor system mentioned above
with scaling by $\frac{1}{4}$ and two subdivisions (a measure $\mu$ with scaling
  dimension $\frac12$).  It was shown in \cite{JoPe98} that it is better to realize
$\mu$ as a complex measure which is supported on the circle in the
complex plane.  The monomials we are then led to study are the complex
monomials $\{z^k \}_{k \in \mathbb{N}_0}$ in $L^2(\mu)$.  

This set of all the monomials has no finite linear dependence relations.  To see this, we assume a polynomial $p(z) = \sum_{j=0}^n a_jz^j$ supported on the circle is zero in $L^2(\mu)$.  Then for almost every point $z$ in the support of $\mu$, the polynomial is zero.  Since this measure has no atoms (see \cite{DuJo06d},\cite{JKS07c}), the complement of a set of measure zero must be infinite.  We conclude that since $p$ has infinitely many zeros, it is the identically zero polynomial and hence $a_0=a_1=\cdots=a_n=0$.    

The standard application of Gram-Schmidt on the monomials would produce an ONB of polynomials, but it would be different from an ONB of monomials such as $\{z^k | k =0, 1, 4, 5, 16, 17, 20, 21, ...\}$.  Moreover, this approach misses the IFS scaling property of $\mu$.    \end{remark}

Returning to our quest to encode the map $\tau$ via a matrix $A$, we perform the Gram-Schmidt process on
the finitely linearly independent monomials $\{v_j\}_{j\in\mathbb{N}_0}$  to  construct an orthonormal
basis  $\{p_k\}_{k \in \mathbb{N}_0}$ of polynomials for
$\mathcal{P}$.   Part of the definition of the Gram-Schmidt
process gives that each polynomial $p_k$ is in the span of
$\{v_0,v_1, \ldots v_k\}$, and hence is orthogonal to $sp\{v_0,
\ldots, v_{k-1}\} = sp\{p_0, \ldots p_{k-1}\}$.  We can write the lower triangular matrix $G$ which enacts Gram-Schmidt as follows:
\begin{equation}\label{Eqn:GramMatrix} \sum_{i=0}^k G_{k,i}v_i = p_k.\end{equation} Also, assume that $\mu$ is a
probability measure, so $p_0$ is the constant function $1$, i.e.
$p_0(x)\equiv 1$.

In the cases where some of the monomials have been left out of the
sequence $\{v_j\}$, we define a moment matrix $N^{(\mu)}$ for
$\mu$ to be \begin{equation} N^{(\mu)}_{j,k} = \langle v_j|v_k
\rangle_{L^2(\mu)}. \end{equation} This adjusted moment matrix $N^{(\mu)}$
will be symmetric in the real cases but will not have the Hankel
property.   The moments contained in $N^{(\mu)}$ are total in
$\mathcal{P}$.

We place the condition on the map $\tau$  that the powers
$\tau^j$ are in  the space $\mathcal{P}$ for all $j \in
\mathbb{N}_0$.  From here, we define the following transformations
on $\mathcal{P}$:

\begin{eqnarray}  Rp_k  &=& v_k \label{Eqn:R} \\ Tp_k &=& v_k \circ \tau. \label{Eqn:T} \end{eqnarray}

$R$ and $T$ are well-defined operators on $\mathcal{P}$ since they are defined on an orthonormal basis.  They might be unbounded, but their domains do contain the ONB elements $\{p_k\}_{k \in \mathbb{N}_0}$ by definition.  Therefore, we can express $R$ and $T$ in matrix form with respect to $\{p_k\}_{k \in \mathbb{N}_0}$.  With a slight abuse of notation, we will also refer to these matrices as $R$ and $T$ respectively:

\begin{eqnarray}  R_{j,k} &=&  \langle p_j | Rp_k \rangle_{L^2(\mu)} = \langle p_j  | v_k \rangle_{L^2(\mu)} \label{eqn:grammatrix}\\ T_{j,k} &=& \langle p_j|Tp_k \rangle_{L^2(\mu)} =  \langle p_j|v_k \circ \tau \rangle_{L^2(\mu)}. \label{eqn:T}\end{eqnarray}
Observe that  the matrix for $R$ is upper triangular, since each $p_j$ is orthogonal to $sp\{v_0, \ldots v_{j-1}\} = sp\{p_0, \ldots p_{j-1}\}$.

\begin{lemma}\label{Lem:GramInverse} The matrix $R$ is invertible in the
sense of Definition \ref{def:inverse}, and the matrix of $R^{-1}$
is exactly the transpose of the matrix which enacts the
Gram-Schmidt process on the monomials, i.e.  \[ \sum_{i=0}^k
R^{-1}_{i,k}v_i = p_k. \]
\end{lemma}

\begin{proof}
We compute the matrix product $RG^{tr}$.  Note that the triangular structure of $G^{tr}$ gives a finite sum, so the product is well-defined:

\begin{eqnarray*}  (RG^{tr})_{i,j} &=& \sum_{k=0}^{\infty} R_{i,k}G_{j,k}\\
&=& \sum_{k=0}^{\infty} \langle p_j|v_k \rangle_{L^2(\mu)}G_{j,k} \\
&=& \Bigr\langle p_j|\sum_{k=0}^j G_{j,k}v_k\Bigr\rangle_{L^2(\mu)}\\&=& \langle p_i|p_j \rangle\\ &=& \delta_{i,j}.  \end{eqnarray*}
We need to show that this composition has a dense domain.  One can readily verify that when a monomial $v_j$ is expressed as a column vector with respect to $\{p_k\}_{k \in \mathbb{N}_0}$, that column vector only has nonzero entries $\langle p_k|v_j \rangle$ for $k \leq j$.  Moreover, by the upper triangular structure of $G^{tr}$, the matrix-vector product $G^{tr}v_j$ must be in $\mathcal{D}$.    These are in the domain of the matrix for $R$, so we can now conclude that the monomials
$v_k$ are all in the domain of the operator $RG^{tr}$.  Therefore, $RG^{tr}$ has dense domain and
on that domain, $RG^{tr}$ is the identity.  By Lemma \ref{lem:inverse},
$G^{tr}$ and $R$ are  inverses of each other, and we write $G^{tr} =
R^{-1}$.
\end{proof}

Next, we will discuss the adjoints $R^*$ and $T^*$.  These can be
written down as matrices, but it is not always true that their
domains are dense in $\mathcal{P}$.  We will show that there
exists a \textit{renormalization} of $L^2(\mu)$ such that both
$R^*$ and $T^*$ have dense domains in $\mathcal{P}$.  This is
equivalent (see \cite[Proposition 1.6, Chapter 10]{Con90}) to
saying $R$ and $T$ are closable operators.

Given a set of nonzero weights $w=\{w_i\}_{i \in \mathbb{N}_0}$, define the space $\mathcal{P}_w$ to be the set of all measurable functions which are in the span of the monomials with respect to the weighted norm \[
\|f\|_{\mathcal{P}_w}^2 = \sum_{i=0}^{\infty} w_i |\langle f|p_i \rangle_{L^2(\mu)}|^2.\]
To be precise, we see that $\int f p_i \mathrm{d}\mu$ must be finite for all $i \in \mathbb{N}_0$, but depending on the weights, observe that $\mathcal{P}_w$ could include functions which are not in $L^2(\mu)$.
The inner product, then, on $\mathcal{P}_w$ is given by \[ \langle f|g \rangle_{\mathcal{P}_w} = \sum_k w_k \langle f|p_k\rangle_{L^2(\mu)} \langle p_k|g\rangle_{L^2(\mu)}.\]

If we consider the operator $R$ as defined above, but as a map from $\mathcal{P}_w$ to $\mathcal{P}$, then the adjoint $R_w^*: \mathcal{P} \rightarrow \mathcal{P}_w$ is given by 
\begin{equation}\label{Eqn:DefnRStar}
R_w^*g = \sum_k \frac{1}{w_k} \langle v_k| g \rangle_{L^2(\mu)} p_k
\end{equation}
on every $g$ in the domain of $R^*_w$ (i.e. every $g$ such that $R_w^*g \in \mathcal{P}_w$.)  Observe that the weights change the adjoint, so we denote it by $R^*_w$.  To verify Equation (\ref{Eqn:DefnRStar}), we compute for $f \in \mathcal{P}_w \cap\, \mathrm{Dom}(R)$ and $g \in \mathcal{P}$; 
 \[Rf = \sum_{k=0}^{\infty} \langle p_k|f \rangle_{L^2(\mu)} v_k,\]
which gives

 \[ \langle g|Rf \rangle_{L^2(\mu)} = \Bigr\langle g|\sum_{k=0}^{\infty} \langle p_k|f \rangle_{L^2(\mu)}  v_k\Bigr\rangle_{L^2(\mu)} = \sum_{k=0}^{\infty} \langle p_k|f \rangle_{L^2(\mu)}  \langle g|v_k \rangle_{L^2(\mu)} .\] 
 
 We then compute 

\begin{eqnarray*}   \Bigr\langle \sum_k \frac{1}{w_k} \langle v_k|g \rangle_{L^2(\mu)} p_k \Bigr| f \Bigr\rangle_{\mathcal{P}_w} &=&
\sum_j w_j \Bigr\langle  \sum_k \frac{1}{w_k} \langle v_k|g \rangle_{L^2(\mu)} p_k \Bigr| p_j \Bigr \rangle_{L^2(\mu)} \langle p_j|f \rangle_{L^2(\mu)} \\ &=& \sum_k \langle g|v_k \rangle_{L^2(\mu)} \langle p_k|f\rangle_{L^2(\mu)}.
\end{eqnarray*}
The equality of these two expressions verifies our definition of $R^*_w$,  at least for the dense set of finite linear combinations of the orthogonal polynomials, hence holds for all $f$ in the domain of $R$.  


In the next lemma, we produce a weighted norm on the space
$\mathcal{P}$ using weights $\{w_k\}$ which guarantee that $R$
will be closable in the weighted space $\mathcal{P}_w$.

\begin{lemma}\label{Lemma:RClosable}
Given the measure space $(X, \mu)$ as above having finite moments of all orders and given the operator $R$ defined in Equation (\ref{Eqn:R}), there exist weights $\{w_i\}_{i \in \mathbb{N}_0}$ such that $R: \mathcal{P}_w \rightarrow
\mathcal{P}$  is closable.
\end{lemma}
\begin{proof} 
Observe that  $R_w^*g\in \mathcal{P}_w$ precisely when
\begin{equation}\label{Eqn:RStar2}
\sum_k w_k | \langle p_k| R_w^*g \rangle_{L^2(\mu)}|^2 < \infty.
\end{equation}
Substituting in the definition of $R^*_w$ shows that Equation (\ref{Eqn:RStar2}) is true if and only if
\begin{equation}
\sum_k \frac{1}{w_k} \Big|\langle v_k|g\rangle_{L^2(\mu)}\Big|^2  < \infty.
\end{equation}
Since $\mu$ has finite $0^{\mathrm{th}}$ moment, i.e. is a finite measure, we have $L^{\infty}(\mu) \subseteq L^2(\mu)$.  Set $\mathcal{M}_k:=\int |v_k| d\mu(x)$, and suppose $f\in
L^{\infty}(\mu)$.  Thenf
\begin{equation}
\Big| \langle v_k|f \rangle_{L^2(\mu)}\Big| \leq \int_{\mathbb{R}} |v_k f |\mathrm{d}\mu \leq
\|f\|_{L^{\infty}(\mu)} \mathcal{M}_k.
\end{equation}
Therefore, if $\{\mathcal{M}_k^2/w_k\}\in \ell^1$, then
\begin{equation}
L^{\infty}(\mu)\subset \textrm{dom}(R^*).
\end{equation}
We thus explicitly define a choice of weights $w = \{w_k\}_{k \in \mathbb{N}_0}$ by
\begin{equation}\label{Eqn:RWeights}
w_k:=(1 + k^2)\mathcal{M}_k^2.
\end{equation}
Since $L^{\infty}(\mu)$ is dense in $L^2(\mu)$, we now know that
$\textrm{dom}(R_w^*)$ is dense in $\mathcal{P}$, where $w = \{w_k\}_{k \in \mathbb{N}_0}$ is
defined in (\ref{Eqn:RWeights}).  Therefore $R:\mathcal{P}_w
\rightarrow \mathcal{P}$ is closable. 
\end{proof}

The space $\mathcal{P}_w$ has the weighted orthogonal polynomials $\left\{\frac{p_k}{\sqrt{w_k}}\right\}_{k \in \mathbb{N}_0}$ as an orthonormal basis.  We will denote these vectors $\{p_k^w\}_{k \in \mathbb{N}_0}$.  We next observe the following property of the matrix $R_w^*R$ written in terms of the orthonormal basis $\{p^w_k\}_{k \in \mathbb{N}_0}$:

\begin{eqnarray}\label{Eqn:RStarR}  R_w^*R_{i,j}&=& \langle p_i^w | R^*_wRp_j^w \rangle_{\mathcal{P}_w} \\ \nonumber&=& \langle Rp_i^w|Rp_j^w \rangle_{L^2(\mu)} \\ \nonumber&=& \frac{\langle v_i|v_j \rangle_{L^2(\mu)}}{\sqrt{w_iw_j}} \\\nonumber &=& \frac{1}{\sqrt{w_iw_j}}N^{(\mu)}_{i,j}\end{eqnarray}.

So $R^*_wR$ is a self-adjoint operator and is exactly a weighted version of the moment matrix.

Next, we recall our definition (Equation (\ref{Eqn:T})) for the
operator $T$ which maps $p_i$ to $v_i \circ \tau$ for each $i \in
\mathbb{N}_0$.   By our hypothesis on $\tau$, we know $T$ is an
operator on $\mathcal{P}$.  We wish $T$ to also be closable, so
that $T^*$ is densely defined.  As we discovered for $R$, this may
require weights.  Since the adjoint depends on the weights, we will denote it $T^*_w$.

A function $f$ is in the domain of $T_w^*$ with respect to weights
$w = \{w_k\}_{k \in \mathbb{N}_0}$ if $T_w^*f$ is in  $\mathcal{P}_w$.  The same
computations used for $R^*_w$, replacing each $v_k$ with $v_k
\circ \tau$, show
\begin{equation}\label{Eqn:DefTStar} T_w^*g = \sum_{k=0}^{\infty} \frac{1}{w_k} \langle v_k \circ \tau|g \rangle_{L^2(\mu)}. \end{equation}   We also find that $g \in \mathrm{dom}(T_w^*)$ if and only if \begin{equation}\label{Eqn:DomTStar} \sum_{k=0}^{\infty} \frac{1}{w_k} \Big|\langle v_k \circ \tau|g \rangle_{L^2(\mu)}\Big|^2 < \infty. \end{equation}

The condition on $\tau$ under which there are weights such that $T^*_w$ is densely defined on $\mathcal{P}_w$ is given below.

\begin{lemma}\label{Lem:TClosable} Let $\mathcal{M}'_k  = \int_X |v_k \circ \tau |d\mu = \int |v_k| d(\mu \circ \tau^{-1})$.
If $\mathcal{M}'_k < \infty$ for all $k \in \mathbb{N}_0$,  we
define the weights $w=\{w_k\}_{k \in \mathbb{N}_0}$ by \[ w_k = (\mathcal{M}'_k)^2
(1+k^2),
\] for which $T^*_w$ is densely defined from $\mathcal{P}$ to
$\mathcal{P}_w$.
\end{lemma}
\begin{proof}
Repeating the computation in Lemma \ref{Lemma:RClosable}, we find
that $f \in L^{\infty}(\mu)$ is in the domain of $T_w^*$ if and only
if
\[ \sum_k \frac{1}{w_k} (\mathcal{M}')^2 \|f\|_{\infty} < \infty, \] which holds for the given weights.
\end{proof}

This then gives us an expression for the matrix of the self-adjoint operator $T_w^*T$ with respect to our orthonormal basis of polynomials:

\begin{eqnarray}\label{Eqn:TStarT} (T_w^*T)_{i,j} &=& \langle p^w_i|T^*_wT p^w_j \rangle_{\mathcal{P}_w} \nonumber\\ &=& \langle Tp^w_i | Tp^w_j \rangle_{L^2(\mu)} \nonumber\\&=& \frac{\langle v_i \circ \tau|v_j \circ \tau \rangle_{L^2(\mu)}}{\sqrt{w_iw_j}} \nonumber \\ &=& \frac{ 1}{\sqrt{w_iw_j}} N^{\mu \circ \tau^{-1}}_{i,j}. \end{eqnarray}

We next define a matrix
$A$ that will give coefficients of the functions $v_k \circ \tau$
expanded in terms of the monomials $\{v_j\}_{j \in \mathbb{N}_0}$:

 \begin{equation}\label{Eqn:DefA}
v_k \circ \tau = \sum_{j=0}^{\infty} A_{jk}v_j.
\end{equation}
The entries of $A$ exist since we assumed that each $v_k \circ
\tau$ is an element of  $\mathcal{P}$ and therefore has an $L^2$-convergent
expansion in the monomials.  We may or may not, however, be able to compute
entries of the matrix $A$ directly from this definition.  

\begin{example}\label{Ex:Polynomial} A nonaffine map:  $\tau(x) = x^2+b$ \end{example}   This is perhaps the simplest example on a nonaffine transformation.  The matrix $A$ which encodes $\tau(x)=x^2+b$ can be computed from the powers of $\tau$ to satisfy Equation (\ref{Eqn:DefA}): 

\[A = \left[ \begin{matrix} 1&b&b^2 &b^3& \cdots\\ 0&0& 0 &0& \cdots\\
0&1&2b&3b^2 &\cdots\\ 0&0&0 &0& \cdots \\  0&0&1&3b&\cdots \\ \vdots & \vdots & \vdots  & \vdots & \ddots  \end{matrix} \right].  \]
\hfill $\Diamond$

We wish to think of $A$ as the matrix representation for an operator on the weighted space $\mathcal{P}_w$.  With an abuse of notation, we will also refer to this operator as $A$.  In general, it is not even certain that the operator $A$ has dense domain.  Throughout the remainder of this section, however, we will restrict our attention to the cases in which the domain of $A$ contains the monomials, and hence the orthogonal polynomials.  

\begin{proposition}\label{Prop:TRA}  Given $A$ as in Equation (\ref{Eqn:DefA}) and such that $A$ is the matrix representation of an operator (also denoted $A$) on $L^2(\mu)$ with respect to the ONB $\{p_k\}_{k\in \mathbb{N}_0}$, if the operator composition $RA$ is densely defined, then \[RA=T\] on the domain of $RA$, i.e. $RA$ is a restriction of $T$.
\end{proposition}

\begin{proof} Suppose $f \in \mathrm{dom}(R)$.  Then by Parseval we have
\begin{eqnarray*}  Rf &=& \sum_{j=0}^{\infty} \langle p_j|f
\rangle_{L^2(\mu)} Rp_j\\ &=& \sum_{j=0}^{\infty} \langle p_j|f \rangle_{L^2(\mu)}
v_j, \end{eqnarray*} with convergence in the $L^2$ sense.

We have assumed that, as an operator, $A$ has the monomial functions $\{v_k\}_{k \in \mathbb{N}_0}$ and the
 polynomials $\{p_k\}_{k \in \mathbb{N}_0}$ in its domain.  We have also assumed that the product $RA$ is well-defined on a dense subset of $L^2(\mu)$.  We can then compute 
\begin{eqnarray*} Tp_k &=& v_k \circ \tau \\ &=&  \sum_{j=0}^{\infty} A_{j,k}v_j \\&=&
\sum_{j=0}^{\infty} \langle p_j|Ap_k \rangle_{L^2(\mu)} v_j \\&=&
RAp_k.
  \end{eqnarray*}
The last line above holds for every $Ap_k$ that is in the domain
of $R$.
\end{proof}

\begin{theorem}\label{Thm:AMatrix} Given $(X,\mu)$ a Borel measure space with
$X \subset \mathbb{R}$ and given $\tau$ a measurable map from $X$
to itself, let $N^{(\mu)}$ be the adjusted moment matrix for $\mu$
and $N^{(\mu \circ \tau^{-1})}$ the corresponding moment matrix
for $\mu \circ \tau^{-1}$.  Then, if the matrix $A$ from Equation (\ref{Eqn:DefA}) satisfies the hypotheses in Proposition \ref{Prop:TRA}, \begin{equation}\label{Eqn:Nmu} A^* N^{(\mu)}A =
N^{(\mu \circ \tau^{-1})} .\end{equation}
\end{theorem}

\begin{proof}  We have $RA =T$.  Find weights $w$ such that both $T^*_w$ and $R^*_w$ are densely defined on $L^2(\mu)$.  Let $N^{(\mu)}_w$ be the weighted moment matrix with entries $\frac{1}{\sqrt{w_iw_j}}\langle v_i|v_j \rangle_{L^2(\mu)} $  Using Equations (\ref{Eqn:RStarR}) and (\ref{Eqn:TStarT}), we find that
\begin{eqnarray*}  N_w^{(\mu \circ \tau^{-1})} &=& T^*_wT \\&=& (RA)^*(RA) \\ &=& A^*R^*_wRA \\ &=& A^*N_w^{(\mu)}A.
\end{eqnarray*}
\end{proof}


\section{Approximation of $A$ with finite matrices}\label{Subsec:Finite}
 In this section, we explore whether we can perform computations with
finite matrices which yield a finite approximation to the infinite
matrix $A$ defined in Equation (\ref{Eqn:DefA}), and thereby achieve an
approximation of the moments for $\mu \circ \tau^{-1}$ from the
moments for $\mu$.

Let $\mu$ be a Borel measure on a set $X \subset \mathbb{R}$ such
that the moments $M^{(\mu)}_{i,j} = \int_X x^{i+j} d\mu(x)$ are
finite for all orders. Let $\tau:X \rightarrow X$ be a measurable
endomorphism such that the moments with respect to $\mu \circ
\tau^{-1}$ are also finite for all orders, and the powers of
$\tau$ are in the closed span $\mathcal{P}$ of the monomials in
$L^2(\mu)$. Let $T$ be the infinite matrix introduced in
Section \ref{Sec:GeneralA} with entries $T_{ij} = \langle p_i|v_j
\circ \tau \rangle_{L^2(\mu)}$, where $\{p_i\}_{i\in \mathbb{N}_0}$ are the
orthonormal polynomials in $L^2(\mu)$ given by performing the
Gram-Schmidt method on the monomials $\{v_j\}_{j \in \mathbb{N}_0}$.  Let $R$ be the
transformation taking $p_i$ to $v_i$, so the matrix entries are
$R_{i,j} = \langle p_i|v_j \rangle_{L^2(\mu)}$.

 Fix $n$. Let $R_n$ and $T_n$
be (as in Section \ref{Subsec:FixedPointsHut}) the $(n+1) \times
(n+1)$ truncations of the $R$ and $T$ matrices, respectively.  Let
$\mathcal{H}_n$ be the closed linear span of the monomials
$\{v_0,v_1, \ldots, v_n\}$ (which is also the closed linear span
of $\{p_i\}_{i=0}^n$) and let $P_n$ be the orthogonal projection
onto $\mathcal{H}_n$ from $\mathcal{P}$. Note that in the Dirac
notation mentioned in Chapter \ref{Sec:Notation},
\begin{equation}\label{Eqn:Pn} P_n = \sum_{i=0}^n |p_i\rangle
\langle p_i | .\end{equation}

We have proved in Proposition \ref{Prop:TRA}  that the matrix $A =
R^{-1}T$, in the cases where this
product of infinite matrices is well defined, gives the coefficients of $v_j \circ \tau$ in terms of the monomials $\{v_i\}_{i \in \mathbb{N}_0}$. We now wish to show
that the finite matrix product $R^{-1}_nT_n$ provides an
approximation of $A$, in the sense that it yields coefficients for
the projection of the powers of $\tau$ onto $\mathcal{H}_n$,
expanded in terms of the monomials.

\begin{lemma}\label{Lemma:finite} For fixed $n$, \begin{equation} \sum_{j=0}^n
(R^{-1}_nT_n)_{j,k}v_j = P_n(v_k \circ \tau). \end{equation}
Consequently,
\begin{equation} \lim_{n \rightarrow \infty} \sum_{j=0}^n
(R_n^{-1}T_n)_{j,k}v_j = v_k \circ \tau, \end{equation} where
convergence is in $L^2(\mu)$.
\end{lemma}

\begin{proof} Given our fixed $n$, let $k \leq n$.  In the
computation below, recall that for $j,k \leq n$, we have
$(R^{-1}_n)_{j,k} = R^{-1}_{j,k}$ and $(T_n)_{j,k} = T_{j,k}$.
Also, since the orthogonal polynomials $p_0, \ldots, p_j$ are in
the span of the monomials $v_0, \ldots v_j$, we know that the
truncated matrix $R^{-1}_n$ maps $v_j$ to $p_j$ for each $j \leq
n$.

 \begin{eqnarray*} \sum_{j=0}^n
(R_n^{-1}T_n)_{j,k}v_j &=& \sum_{j=0}^n \sum_{\ell=0}^n
(R_n)^{-1}_{j,\ell}(T_n)_{\ell,k}v_j \\ &=& \sum_{\ell=0}^n
\langle p_{\ell}|v_k \circ \tau \rangle_{L^2(\mu)} \sum_{j=0}^n
(R^{-1}_n)_{j,\ell} v_j  \quad \text{where both \: }\ell,j \leq n \\
&=& \sum_{\ell=0}^n \langle p_{\ell}|v_k \circ \tau
\rangle_{L^2(\mu)}p_{\ell}
\\ &=& P_n
v_k \circ \tau \quad \text{by Equation (\ref{Eqn:Pn})}.
\end{eqnarray*}

\end{proof}

\begin{remark} It is important to realize here that even for $j,k \leq
n$,
\begin{equation} (R^{-1}_nT_n)_{j,k} \neq (R^{-1}T)_{j,k}.
\end{equation}
We do, however, know that because the product $R^{-1}T$ is well
defined, if we fix $j,k$ and let $n \rightarrow \infty$, we do
have $(R^{-1}_nT_n)_{j,k} \rightarrow (R^{-1}T)_{j,k}$.
\end{remark}

If we were able to conclude from the truncation result in Lemma
\ref{Lemma:finite} that  $ \lim_{n \rightarrow \infty}
\sum_{j=0}^n (R^{-1}T)_{j,k}v_j = v_k \circ \tau$, where we take
the limit in $L^2(\mu)$, we would have an alternate proof of
Proposition \ref{Prop:TRA}. It other words, we would be able to
write
\begin{equation}  \sum_{j=0}^{\infty} (R^{-1}T)_{j,k}v_j =
v_k \circ \tau. \end{equation}   This convergence, however, would
require hypotheses about how the sequences $(R^{-1}_nT_n)_{j,k}$
converge to $(R^{-1}T)_{j,k}$ as $n \rightarrow \infty$.  To
illustrate this point, let us fix $k$ and let  $a_j^{(n)} =
(R^{-1}_n T_n)_{j,k}$ and $a_j = (R^{-1}T)_{j,k}$, so for each $j$,  $a_j^{(n)} \rightarrow a_j$ as $n \rightarrow
\infty$.
\begin{eqnarray*}  \Bigr\|\sum_{j=0}^n a_jv_j - v_k \circ \tau\Bigr\|_{L^2(\mu)} &=& \Bigr\| \sum_{j=0}^n a_jv_j - \sum_{j=0}^n a^{(n)}_jv_j
+ \sum_{j=0}^n a^{(n)}_jv_j  - v_k \circ \tau\Bigr\|_{L^2(\mu)} \nonumber \\
&\leq& \Bigr\|\sum_{j=0}^n (a_j-a_j^{(n)})v_j\Bigr\|_{L^2(\mu)} +
\Bigr\|\sum_{j=0}^n a^{(n)}_jv_j  - v_k \circ \tau\Bigr\|_{L^2(\mu)}
\end{eqnarray*}

Using Lemma \ref{Lemma:finite}, the second term of the last line
above can certainly be made arbitrarily small for large enough
$n$. The first term, however, may not have that property.

For each $k \in \mathbb{N}_0$, the truncated matrix products
$R^{-1}_nT_n$ produce asymptotic expansions for $v_k \circ \tau$
by giving expansions of $P_n(v_k \circ \tau)$ in terms of monomials
$v_j$, even if there is no \textit{a priori} known expansion of
$v_k \circ \tau$ in the monomials.  We assume that such an
expansion exists, that is,
$$v_k \circ \tau = \sum_{j=0}^{\infty} a_j v_j. $$ We may, however, have no
way of computing the actual coefficients.  Lemma
\ref{Lemma:finite} gives approximations to these coefficients
which get better as $n$ increases.
The next question is how these finite approximations to the matrix
$A$ interact with the moment matrix transformation given in Equation (\ref{Eqn:Nmu}) from
Theorem \ref{Thm:AMatrix}.  In the case described in Subsection
\ref{Sec:SingleVar}, where $\tau$ is an affine map on $\mathbb{R}$
or $\mathbb{C}$, the matrix $A$ is upper triangular. With this
added structure, it is readily demonstrated that the truncated
matrix product yields exactly the truncation of the infinite
matrix product, i.e.
$$A^*_nM^{(\mu)}_nA_n = [A^*M^{(\mu)}A]_n = M^{(\mu \circ \tau^{-1})}.$$  We also note that in this special
case, we have Equation (\ref{Eqn:UpTriA})  from \cite{EST06} giving
a concrete expression of the entries of $A$. We then can compute
the exact moments with respect to $\mu \circ \tau^{-1}$ using
finite matrix computations.

In the more general case, however, we may not have a construction
giving us the entries in $A$, so we may need to use the finite
matrix product $\widetilde{A}_n = R_n^{-1}T_n$.  In this case, we
find the triple product gives entries which are the inner products
of projections of powers of the measurable map $\tau$.

Given the map $\tau$, the entries in the adjusted moment matrix for $\mu
\circ \tau^{-1}$ are given by the inner product  $$ (N^{(\mu \circ
\tau^{-1})})_{i,j} = \langle v_i \circ \tau | v_j \circ \tau
\rangle_{L^2(\mu)}. $$

\begin{corollary} Let $\widetilde{A}_n = R_n^{-1}T_n$.  Then
\begin{equation}  (\widetilde{A}_n^*N^{(\mu)}\widetilde{A}_n)_{i,j} =
\langle P_n(v_i \circ \tau) | P_n (v_j \circ \tau) \rangle_{L^2(\mu)} = \langle
v_i \circ \tau | P_n (v_j \circ \tau) \rangle_{L^2(\mu)}.
\end{equation}
\end{corollary}

\begin{proof} This is a consequence of Lemma \ref{Lemma:finite}.

\begin{eqnarray*} (\widetilde{A}_n^*N^{(\mu)}\widetilde{A}_n)_{i,j}
&=& \sum_{k=0}^n \sum_{l=0}^n
(\widetilde{A}_n^*)_{i,k}(N^{(\mu)})_{k,l}(\widetilde{A}_n)_{l,j} \\
&=& \sum_{k=0}^n
\sum_{l=0}^n(\widetilde{\overline{A}}_n)_{k,i} \langle v_k|v_l
\rangle_{L^2(\mu)} (\widetilde{A}_n)_{l,j}\\ &=& \left\langle
\sum_{k=0}^n (\widetilde{A}_n)_{k,i}v_k \Bigr| \sum_{l=0}^n
(\widetilde{A}_n)_{l,j}v_l \right\rangle_{L^2(\mu)} \\&=& \langle P_n(v_i
\circ \tau) | P_n(v_j \circ \tau) \rangle_{L^2(\mu)}\\ &=& \langle v_i \circ
\tau | P_n (v_j \circ \tau) \rangle_{L^2(\mu)}.
\end{eqnarray*}
In the last line, we use the property of projections that $P_n =
P_n^* = P_n^2$.
\end{proof}

%% file: kato_08_09_11.tex
\chapter{The Kato-Friedrichs operator}\label{Ch:Kato}

The moment matrix $M^{(\mu)}$ is an infinite matrix, and while $M^{(\mu)}$ may not be a well defined operator on $\ell^2$, we may be able to view $M^{(\mu)}$ as an operator in some other sequence space.  We use the techniques of Kato and Friedrichs to turn $M^{(\mu)}$ into a self-adjoint densely defined operator on a weighted $\ell^2$ space.  Our main tool will be a quadratic form $Q_M$ which is defined from the moments of $\mu$.   We obtain the weighted $\ell^2$ space by finding a space in which the quadratic form $Q_M$ is closable.  In the course of showing that $Q_M$ is closable, we introduce two key operators:  $F$ and its adjoint $F^*$, which we will continue to study in Chapters \ref{Sec:IntOperators} and \ref{Sec:Spectrum}.  

\section{The quadratic form $Q_M$}\label{Subsec:QClosable}
Let $\mathcal{D}$ be the set of all finitely supported sequences indexed by $\mathbb{N}_0$.  
The familiar expression in Equation (\ref{Eqn:MatrixMult}) for the infinite matrix-vector product  $Mc$  makes
sense for every $c \in \mathcal{D}$, but $Mc$ may not be well defined for
every sequence $c$.  Even when it is well defined, $Mc$ may not
be in the same Hilbert space as $c$.   We will address this technicality by changing the domain of $M$.

Let $M$ be the moment matrix of a positive Borel measure with finite moments, where $M$ satisfies the
positive semidefinite condition in Definition \ref{Defn:PD} (real
case) or the PD$\mathbb{C}$ condition in Definition \ref{Defn:PDC}
(complex case).   The quadratic form $Q_M$ is defined on infinite sequences with only finitely many nonzero
components.  Specifically, given $c\in\mathcal{D}$ we have
\begin{equation}\label{Def:QM}
Q_M(c) = \sum_i \sum_j \overline{c_i} M_{i,j} c_j.
\end{equation}

Even though every operator on a Hilbert space determines a quadratic form, the converse is not necessarily true.  Here, we look for conditions on the quadratic form $Q_M$ which guarantee that the quadratic form gives rise to an operator on a Hilbert space $\mathcal{H}$.  The work of Friedrichs and Kato provides the correct conditions.  Friedrichs (see \cite[Section 6.2.3] {Kat80}) proved that semibounded symmetric operators in a Hilbert space
$\mathcal{H}$ have self-adjoint extensions in $\mathcal{H}$.  Later,
Kato \cite[Section 6.2.1-Theorem 2.1; Section 6.2.6-Theorem 2.23]{Kat80} extended
Friedrichs's theorem to quadratic forms.

\begin{theorem}[Kato]\label{Thm:Kato} Let $Q$ be a densely defined, closed, positive quadratic form.  Then there exists a unique  self-adjoint operator $H$ on a Hilbert space $\mathcal{H}$ such that the Hilbert space completion $\mathcal{H}_Q$ of $Q$ is equal to the completion of $H$, and for all $c$ in the domain of $Q$,  \[ Q(c) = \|H^{1/2}c\|^2_{\mathcal{H}} = \|c\|^2_{\mathcal{H}_Q}. \]  In particular, the domain of $Q$ is equal to the domain of $H^{1/2}$.  
\end{theorem}

Our quadratic form $Q_M$ arising from a moment matrix may not be closed, but we can show that $Q_M$ is
\textit{closable} and then apply Kato's theorem to the closure.  We change our notation briefly here in order to work with sequences of elements in $\mathcal{D}$.  A sequence in $\mathcal{D}$ will be denoted
$\{c_n\}_{n\in\mathbb{N}_0}$, and the $i^{\textrm{th}}$ component of $\{c_n\}$ is $c_n(i)$.  (We retain the
notation $c = \{c(i)\}_{i\in\mathbb{N}_0}$ for elements of
$\mathcal{D}$ and its subsequent completion throughout the next two
sections.) 

\begin{definition}[\cite{Kat80}]\label{Defn:QClosable}
Let $\mathcal{H}$ be a Hilbert space in which  $\mathcal{D}$ is dense. A quadratic form $Q$ defined on $\mathcal{D}$ is \textit{closable} if and only if
\[ c_n \rightarrow 0 \text{ in } \mathcal{H}\]
and
\[ Q(c_n - c_m) \rightarrow 0 \text{ as } m,n\rightarrow \infty\]
imply
\[ Q(c_n)\rightarrow 0 \text{ as }n\rightarrow \infty.\]
\end{definition}

We next define the operator $F$ to
map finite sequences $c\in \mathcal{D}$ to  polynomials (a.k.a generating
functions) in $L^2(\mu)$.  Given $c\in\mathcal{D}$, we define
\begin{equation}\label{Eqn:DefnF}
Fc(x) = f_c(x)  = \sum_{i\in \mathbb{N}_0} c_i x^i .
\end{equation}

\begin{example}A quadratic form which is not closable in $\ell^2$.\end{example}
In order to construct a quadratic form which is not closable in $\ell^2$, we need a sequence $\{c_n\}\subset\ell^2$ and a measure $\mu$ such that
\begin{enumerate}[(1)]
\item $\sum_{i\in\mathbb{N}_0} |c_n(i)|^2 \rightarrow 0$ as $n\rightarrow\infty$
\item $\int_{\mathbb{R}}|f_{c_m}(x) - f_{c_n}(x)|^2 \,\mathrm{d}\mu(x) \rightarrow 0$ as $m,n\rightarrow \infty$
\item $\int_{\mathbb{R}}|f_{c_n}(x)|^2 \,\mathrm{d}\mu(x) \not\rightarrow 0$ as $n\rightarrow\infty.$
\end{enumerate}
Let $\mu = \delta_b$ for $b > 1$, and let $c_n(i) =
\delta(n,i)b^{-i}$ (the Kronecker delta).  Since $b > 1$,
$\|c_n\|^2_{\ell^2} = b^{-2n} \rightarrow 0$ as
$n\rightarrow\infty$, so (1) is satisfied.  Since
\[
f_{c_n}(x) = b^{-n} x^n \textrm{ for each }n,
\]
and the integral in (2) simply evaluates $f_{c_n}$ and $f_{c_m}$ at $x = b$, we have
\[
\int_{\mathbb{R}}|f_{c_m}(x) - f_{c_n}(x)|^2 \,\mathrm{d}\mu(x) = |1 - 1|^2 = 0 \textrm{ for all }m, n.
\]
However, the integral in (3) is $1$ for all $n$.  Therefore the quadratic form associated with $\delta_b$, $b > 1$, is not closable in $\ell^2$.
\hfill$\Diamond$

\section{The closability of $Q_M$}

In order to show that the quadratic form $Q_M$ is closable for a particular infinite Hankel matrix $M$,
we show that the operator $F$ defined in Equation (\ref{Eqn:DefnF}) is closable with respect
to a weighted space  $\ell^2(w)$.  Recall
Definition \ref{Def:FClosed} regarding closed and closable
operators.  The  \textit{weighted} $\ell^2$ spaces are
defined by a sequence of weights $w= \{w_i\}_{i\in\mathbb{N}_0}$:
\[
\ell^2(w):=\Bigl\{ \{c(i)\} \Big| \sum_{i}w_i |c(i)|^2 <
\infty\Bigr\}.
\]
The norm on $\ell^2(w)$ is
\[
\|c\|^2_{\ell^2(w)} = \sum_{i\in\mathbb{N}_0} w_i |c(i)|^2,
\]
and the inner product is given by 
\[ \langle c|d \rangle_{\ell^2(w)} = \sum_{i \in \mathbb{N}_0} w_i \overline{c(i)}d(i). \]
We will generally take the weights to all be strictly greater than zero.

In the next lemma, we use an argument mirroring the proof of Lemma \ref{Lemma:RClosable} to produce a set
of weights $w$ which guarantee that $F$ will be closable in the weighted space $\ell^2(w)$, and then
in Theorem \ref{Thm:QClosable} we show that the quadratic form
$Q_M$ is closable in $\ell^2(w)$.   

 Consider $F$ as a map with domain in $\ell^2(w)$ for some weights $w=\{w_i\}_{i \in \mathbb{N}_0}$. We then see that the adjoint $F^*$ of $F$  is given by 
\begin{equation}\label{Eqn:DefnFStar}
(F_w^*f)_k = \frac{1}{w_k}\int_{\mathbb{R}} x^k f(x) \,\mathrm{d}\mu(x).
\end{equation}

To verify this, we must show that $F_w^*$ satisfies
\begin{equation}
\langle Fc|f\rangle_{L^2(\mu)} = \langle
c|F_w^*f\rangle_{\ell^2(w)}
\end{equation} for every $f$ such that $F_w^*f \in \ell^2(w)$.  

For each $c\in\mathcal{D}$, we have
\begin{equation}
\begin{split}
\langle Fc|f\rangle_{L^2(\mu)}
& = \int_{\mathbb{R}} \overline{\sum_{k \in \mathbb{N}_0} c(k)x^k} f(x) \,\mathrm{d}\mu(x)\\
& = \sum_{k \in \mathbb{N}_0} \overline{c(k)} \int_{\mathbb{R}} x^k f(x) \,\mathrm{d}\mu(x).\\
\end{split}
\end{equation}
We now multiply and divide by $w_k$:
\begin{equation}
\begin{split}
\langle Fc|f\rangle_{L^2(\mu)}
& = \sum_{k\in \mathbb{N}_0} w_k \overline{c(k)} \underbrace{\frac{1}{w_k} \int_{\mathbb{R}} x^k f(x) \,\mathrm{d}\mu(x)}_{(\ref{Eqn:DefnFStar})}\\
& = \langle c | F_w^*f\rangle_{\ell^2(w)},
\end{split}
\end{equation}
where we have shown the right had side of Equation (\ref{Eqn:DefnFStar}) satisfies the definition of the adjoint $F^*_w$.

\begin{lemma}\label{Lemma:FClosable}
Suppose $\int_{\mathbb{R}} |x|^k \,\mathrm{d}\mu(x) < \infty$ for each $k\in\mathbb{N}_0$.
There exist weights $\{w_i\}_{i \in \mathbb{N}_0}$ such that $F:\ell^2(w)\rightarrow
L^2(\mu)$ (\ref{Eqn:DefnF}) is closable.
\end{lemma}
\begin{proof}  The operator $F$ is closable if and only if
$\textrm{dom}(F^*)$ is dense in $L^2(\mu)$.  (See, for example, \cite[Proposition 1.6, Chapter 10]{Con90}.)

Now, $F_w^*f\in \ell^2(w)$ precisely when
\begin{equation}\label{Eqn:FStar2}
\sum_{k\in \mathbb{N}_0} w_k | (F_w^*f)_k|^2 < \infty,
\end{equation}
and (\ref{Eqn:FStar2}) is true if and only if
\begin{equation}\label{Eqn:Mkcriterion}
\sum_k \frac{1}{w_k} \Big| \int_{\mathbb{R}} x^k f(x) \,\mathrm{d}\mu(x)
\Big|^2 < \infty.
\end{equation}
This gives us one sufficient criterion for $F$ to be closable.  We see from Equation (\ref{Eqn:Mkcriterion}) that the monomial functions $\{v_j\}_{j \in \mathbb{N}_0}$ are in the domain of $F^*_w$ if and only if \begin{equation}\label{Eqn:FCriterion}
\sum_k \frac{1}{w_k} \Big| \int_{\mathbb{R}} x^{j+k} \,\mathrm{d}\mu(x)
\Big|^2 = \sum_k \frac{1}{w_k} |M_{j,k}|^2< \infty \end{equation} for all $j \in \mathbb{N}_0$.  

Rather than looking at this criterion (\ref{Eqn:FCriterion}), however, we will find weights that put the essentially bounded functions in the domain of $F^*_w$.  By hypothesis, $\mu$ is a finite measure so $L^{\infty}(\mu) \subseteq L^2(\mu)$.  Set $\mathcal{M}_k:=\int_{\mathbb{R}}|x|^k\,\mathrm{d}\mu(x)$, and suppose $f\in
L^{\infty}(\mu)$.  Then
\begin{equation}
\Big| \int_{\mathbb{R}} x^k f(x) \,\mathrm{d}\mu(x)\Big| \leq
\|f\|_{L^{\infty}(\mu)} \mathcal{M}_k.
\end{equation}

Therefore, if $\{\mathcal{M}_k^2/w_k\}\in \ell^1$, then
\begin{equation}
L^{\infty}(\mu)\subset \textrm{dom}(F_w^*),
\end{equation}
and we explicitly define a choice of weights $w_k$:
\begin{equation}\label{Eqn:FWeights}
w_k:=(1 + k^2)\mathcal{M}^2_k.
\end{equation}
Since $L^{\infty}(\mu)$ is dense in $L^2(\mu)$, we now know that
$\textrm{dom}(F_w^*)$ is dense in $\ell^2(w)$ where $w = \{w_k\}_{k \in \mathbb{N}_0}$ is
defined in (\ref{Eqn:FWeights}).  Therefore $F:\ell^2(w)
\rightarrow L^2(\mu)$ is closable.\end{proof}

\begin{theorem}\label{Thm:QClosable}
Given $\mu$ a Borel measure on $\mathbb{R}$ with finite moments of all orders and  $M = M^{(\mu)}$, there are weights
$\{w(i)\}_{i\in\mathbb{N}_0}$ such that $(Q_M, \mathcal{D},
\ell^2(w))$ is closable.
\end{theorem}
\begin{proof}Choose weights $w = \{w_i\}_{i\in\mathbb{N}_0}$ as
in Equation (\ref{Eqn:FWeights}) so that $F$ is closable.  Suppose
$c_n\underset{\ell^2(w)}{\longrightarrow} 0$ and $Q_M(c_n -
c_m)\rightarrow 0$ as $m,n\rightarrow \infty$ as in Definition \ref{Defn:QClosable}.  Our goal is to show that $Q_M(c_n)\rightarrow 0$.

Denote the polynomial $Fc_n$ by $f_n$ for each $n \in \mathbb{N}_0$.  Because $Q_M$ on $\mathcal{D}$ and $\widetilde{Q}_M$ defined on the
polynomials satisfy
\[\widetilde{Q}_M(F(c)) = Q_M(c),\]
the condition $Q_M(c_n - c_m)\rightarrow 0$ implies that $\{f_n\}_{n \in\mathbb{N}_0}$
is a Cauchy sequence in $L^2(\mu)$. Therefore, there exists $g \in
L^2(\mu)$ such that $f_{n}\underset{L^2(\mu)}{\longrightarrow} g$.

We will now use the condition that $c_n \rightarrow 0$ in
$\ell^2(w)$. By Lemma \ref{Lemma:FClosable},
$F:\ell^2(w)\rightarrow L^2(\mu)$ is closable, so its closure
$\overline{F}$ has a closed graph.  Since $c_n\rightarrow 0$,
$F(c_n) =\overline{F}(c_n)\rightarrow g$, and the graph of
$\overline{F}$ is closed, $g$ must be  $0\in L^2(\mu)$.
\end{proof}

As a result of Theorem \ref{Thm:QClosable}, we can apply Theorem
\ref{Thm:Kato} (Kato's Theorem) to the closure of the quadratic form $Q_M$.  We denote by $\mathcal{H}_Q$ the Hilbert space completion with respect to the quadratic form $Q_M$.  By Kato's theorem, there exists a self-adjoint operator $H$ on $\ell^2(w)$ such that the Hilbert space completion is also $\mathcal{H}_Q$.  This gives \begin{equation}\label{Eqn:KatoProperties} \|c\|^2_{\mathcal{H}_Q} = Q_M(c) = \|H^{1/2}c\|^2_{\ell^2(w)} \end{equation} for all $c \in \mathcal{H}_Q$.  Moreover, the domain of $H^{1/2}$ is equal to the domain of $Q_M$.  It should be noted that in general, the domain of $H$ is a subset of the domain of $H^{1/2}$ and may or may not contain $\mathcal{D}$.  

\section{A factorization of the Kato-Friedrichs operator}

Given $F$ from $\ell^2(w)$ to
$L^2(\mu)$, suppose the weights $w$ have been selected such that
$F$ is closable, i.e. $F^*_w$ is densely defined in $L^2(\mu)$.

Let $\widetilde{Q}_M$ be the quadratic form given by the $L^2(\mu)$-norm:  
\begin{equation}
\widetilde{Q}_M(f)=\int_{\mathbb{R}}|f|^2 \,\mathrm{d}\mu.
\end{equation}
Then for every $c \in \mathcal{D}$, we have \begin{eqnarray}
\widetilde{Q}_M(Fc) &=& \int_{\mathbb{R}}|Fc|^2\,\mathrm{d}\mu \nonumber\\ &=& \int_{\mathbb{R}}\sum_i \sum_j \overline{c}_ic_j x^{i+j} \,\mathrm{d}\mu \nonumber\\ &=&
\sum_i \sum_j \overline{c_i} M_{i,j}c_j = Q_M(c). \end{eqnarray}

\begin{proposition}\label{Lem:FStarFisH} $F^*_wF$ is a self-adjoint densely defined operator on $\ell^2(w)$, and $F^*_wF$ is the Kato operator $H$.
\end{proposition}
\begin{proof}
The domain of $F$ contains $\mathcal{D}$, and the domain of $F^*_w$
contains the polynomials.  Since $F$ maps finite sequences $c \in
\mathcal{D}$ to polynomials $Fc$, we see that $\mathcal{D}$ is
contained in the domain of $F^*_wF$.  In order to prove that
$F^*_wF$ is the Kato operator corresponding to the quadratic form
$Q_M$, we must show that it satisfies \[ Q_M(c) =
\|(F^*_wF)^{1/2}c\|_{\ell^2(w)} \] for all $c \in \mathcal{D}$.

 First, note that for any $c \in \mathcal{D}$,
\begin{eqnarray*}  \langle c |F_w^*Fc \rangle_{\ell^2(w)} &=& \big\langle (F^*_wF)^{1/2}c \big|(F^*_wF)^{1/2}c \big\rangle_{\ell^2(w)} \\
&=& \|(F_w^*F)^{1/2} c\|_{\ell^2(w)}^2. \end{eqnarray*} We also note that
\begin{eqnarray*} \langle c|F_w^*Fc \rangle_{\ell^2(w)} &=& \langle Fc | Fc \rangle_{\ell^2(w)}\\ &=&
\|Fc\|^2_2\\&=& \widetilde{Q}_M(Fc)= Q_M(c). \end{eqnarray*}

Therefore, for $c \in \mathcal{D}$,
\[ Q_M(c) = \langle c|F^*_wFc \rangle_{\ell^2(w)} = \|(F^*_wF)^{1/2}c\|^2_{\ell^2(w)}.\]
By uniqueness of the Kato operator for the closable quadratic form
$Q_M$, we have $F^*_wF = H$.
\end{proof}
This proposition makes the usually mysterious Kato operator much more concrete.  In particular, as a result of this proposition, the domain of $H$ includes $\mathcal{D}$.  
  
\section{Kato connection to $A$ matrix}\label{Subsec:KatoA}

Given a moment matrix $M^{(\mu)}$ and a measurable map $\tau$, we can use Kato theory as another way to describe a matrix $A$ such that, under appropriate convergence criteria, \[M^{(\mu \circ \tau)} = A^*M^{(\mu)}A.\]
This is a different approach to that used in Section \ref{Sec:GeneralA}.  We prove that under certain hypotheses on the measure $\mu$  and the map $\tau$, we can find an operator $\widetilde{A}$ which is an isometry between the Kato completion Hilbert spaces for the quadratic forms $Q_M$ and $Q_M'$, where $M = M^{(\mu)}$ and $M' = M^{(\mu \circ \tau^{-1})}$.  We demonstrate that this isometry $\widetilde{A}$ will intertwine the Kato operators for $Q_M$ and $Q_{M'}$, and therefore the corresponding matrix representation $A$ will intertwine the moment matrices $M$ and $M'$ as required.

Let $(X,\mu)$ be a measure space with $X \subseteq \mathbb{R}$.  For each $k \in \mathbb{N}_0$, let $v_k$ be the monomial function on $X$: $v_k(x) = x^k$. The polynomials may not be dense in $L^2(\mu)$, so denote by $\mathcal{P}$ the closed span of the polynomials in $L^2(\mu)$.  We take $\mu$ to be a measure such that $\{v_k\}_{k\in \mathbb{N}_0} \subseteq L^1(\mu) \cap \mathcal{P}$. Let $M = M^{(\mu)}$ be the moment matrix for $\mu$.  By Section \ref{Subsec:QClosable}, there exist weights $w=\{w_i\}_{i\in\mathbb{N}_0}$ such that the map $F$ is closable on the weighted space $\ell^2(w)$.  Thus the quadratic form $Q_M$ is closable, and we can define the Kato operator $H$ on $\ell^2(w)$.   Note that the domains of $Q_M, F$, and $F_w^*$ all contain the set of finite sequences $\mathcal{D}$.

Let $\tau$ be a measurable endomorphism on $X$ such that the functions $v_k \circ \tau = \tau^k$ are all in $L^1(\mu) \cap \mathcal{P}$.  Let $M' = M^{(\mu \circ \tau^{-1})}$ be the moment matrix for the measure $\mu \circ \tau^{-1}$.   Define the map $F':\mathcal{D} \rightarrow L^2(\mu \circ \tau^{-1})$ to be the analog of $F$, i.e. the map which takes $c \in \mathcal{D}$ to the polynomial  $f_c = \sum_i c_iv_i \in L^2(\mu \circ \tau^{-1})$.   Using the same reasoning as that used in Section \ref{Subsec:QClosable}, there exist weights such that $F'$ is closable, hence the quadratic form $Q_{M'}$ is closable, and we have $H'$, the Kato operator on the corresponding weighted $\ell^2$ space.

Let $w=\{w_k\}_{k \in \mathbb{N}_0}$ be weights that allow for both $F$ and $F'$ to be closable on $\ell^2(w)$, hence the Kato operators $H$ and $H'$ are both densely defined self-adjoint operators on $\ell^2(w)$. Denote the Kato completion Hilbert spaces for the quadratic forms $Q_M$ and $Q_{M'}$ by $\mathcal{H}_Q$ and $\mathcal{H}_{Q'}$, respectively.

Notice that $\phi \in L^2(\mu \circ \tau^{-1})$ if and only if $\int_{\mathbb{R}}|\phi \circ \tau|^2 \mathrm{d}\mu < \infty$, which holds if and only if $\phi \circ \tau \in L^2(\mu)$.  In fact, there is a natural isometry $\alpha$ between $L^2(\mu \circ \tau^{-1})$ and $L^2(\mu)$ given by \[ \alpha(\phi) = \phi \circ \tau.\]

\begin{remark}
The map $F$ which we have defined above actually takes on several realizations in this section.  Simply put, it is the map that takes finite sequences to the polynomial with these coefficients.  But we can think of $F$ as operating on the Kato spaces $\mathcal{H}_Q$, $\mathcal{H}_{Q'}$ or on the weighted space $\ell^2(w)$.  We can also take the codomain of $F$ to be $L^2(\mu)$ or $L^2(\mu \circ \tau^{-1})$.  In order to clarify these different realizations, we will denote the operator $F$ as follows:

\begin{eqnarray*}  F_Q&:& \mathcal{H}_Q \rightarrow L^2(\mu) \\ F_{Q'}&:& \mathcal{H}_{Q'} \rightarrow L^2(\mu \circ \tau^{-1}) \\ F &:& \ell^2(w) \rightarrow L^2(\mu) \\ F' &:& \ell^2(w) \rightarrow L^2(\mu \circ \tau^{-1}). \end{eqnarray*} Moreover, we also observe that the operator $T$ (defined below) which maps standard basis elements $e_i$ to $ \tau^i$ can be expressed in terms of the $F'$ and the isometry $\alpha$.  Given $c \in \mathcal{D}$,
\begin{eqnarray*}  Tc &=& \Bigr(\sum c_i \tau^i \Bigr) \in L^2(\mu)\\ &=& \sum c_i \alpha(v_i ) \in L^2(\mu)\\ &=& \alpha F'c.
\end{eqnarray*}
We will show below that the operators $F_Q$ and $F_{Q'}$ are isometries, and we have already proved that for appropriately chosen weights, the operators $F$ and $F'$ are closable on $\ell^2(w)$.
\end{remark}

\begin{lemma}\label{Lem:KatoIsometry}
 The Kato space $\mathcal{H}_Q$ is isometric to $\mathcal{P} \subseteq L^2(\mu)$, and $\mathcal{H}_{Q'}$ is isometric to a subspace of $L^2(\mu \circ \tau^{-1})$.
\end{lemma}

\begin{proof}  The norm on the Hilbert space $\mathcal{H}_Q$ is given by \[ \|c\|^2_{\mathcal{H}_Q} = Q_M(c) \] when $c \in \mathcal{D}$,  and by the definition of the Kato operator, we also have \[\|c\|_{\mathcal{H}_Q} = \|H^{1/2} c \|_{\ell^2(w)}.\]  Let $F_Q$  be the map taking $c \in \mathcal{D}$ to $f_c = \sum_i c_iv_i$ in $\mathcal{P} \subseteq L^2(\mu)$.  Then \begin{eqnarray*}  \|F_Qc\|_2^2 &=& \|f_c\|^2_2\\ &=& \int_{\mathbb{R}}\Bigr|\sum_ic_iv_i \Bigr|^2 \mathrm{d}\mu\\ &=& \sum_i \sum_j \overline{c_i}c_j \int_{\mathbb{R}}v_iv_k \mathrm{d}\mu\\& = &\sum_i\sum_j \overline{c_i}c_j M^{(\mu)}_{i,j} \\& =& Q_M(c) = \|c\|^2_{\mathcal{H}_Q}.  \end{eqnarray*}

Thus $F_Q$ is an isometry from $\mathcal{H}_Q$ to $\mathcal{P}$.  An identical argument proves that the corresponding $F_{Q'}$ map from $\mathcal{H}_{Q'}$ to $L^2(\mu \circ \tau^{-1})$ is also an isometry onto the subspace $\mathcal{P}'$ of $L^2(\mu \circ \tau^{-1})$ spanned by the polynomials.
\end{proof}

\begin{proposition}\label{Prop:DefA} The Kato Hilbert spaces $\mathcal{H}_Q$ and $\mathcal{H}_{Q'}$ are isometric.
\end{proposition}
\begin{proof} We use the notation $\mathcal{H} \cong \mathcal{G}$ if there exists an isometry from the Hilbert space $\mathcal{H}$ onto the Hilbert space $\mathcal{G}$.  Using Lemma \ref{Lem:KatoIsometry} and the map $\alpha$, we have the following isometries:
\[ \mathcal{H}_{Q'} \cong \mathcal{P}' \cong \mathcal{P} \cong \mathcal{H}_{Q}. \]
\end{proof}

If we denote the isometry in Proposition \ref{Prop:DefA} by $\widetilde{A}$, then
\[ \widetilde{A} = F_Q^{-1}\alpha F_{Q'}.\] The isometry property $\|c\|_{Q'}
= \|\widetilde{A}c\|_{Q}$ implies the quadratic forms satisfy $Q_M(\widetilde{A}c)  =
Q_{M'}(c)$ for all $c \in \mathcal{D}$.

Given $\{e_i\}_{i \in\mathbb{N}_0}$ the standard orthonormal basis for
$\ell^2$, we define the standard orthonormal basis
$\{e_i^w\}_{i\in\mathbb{N}_0} = \{ \frac{e_i}{\sqrt{w_i}}
\}_{i\in\mathbb{N}_0}$ for $\ell^2(w)$.  Given $F$ from $\ell^2(w)$ to
$L^2(\mu)$, suppose the weights $w$ have been selected such that
$F$ is closable, i.e. $F^*_w$ is densely defined in $L^2(\mu)$.  Recall from Lemma \ref{Lem:FStarFisH} that $F^*_wF$ is equal to the Kato operator $H$ on $\ell^2(w)$.

The operator $\widetilde{A}$ has domain containing $\mathcal{D} \subseteq
\mathcal{H}_Q$, but we will need to consider the analog of $\widetilde{A}$ as
an operator on $\ell^2(w)$.  In order to do this, we define the
following maps:
\begin{eqnarray*}  j_w &:& \ell^2(w) \rightarrow \mathcal{H}_Q\\
j_w'&:& \ell^2(w) \rightarrow \mathcal{H}_{Q'}.\end{eqnarray*}
Each of $j_w,j_w'$ acts like the identity map on $\mathcal{D}$,
mapping $e_i^w$ in $\ell^2(w)$ to $e_i^w$ in the Kato spaces
$\mathcal{H}_Q$, $\mathcal{H}_{Q'}$, respectively.   Then, we see \[ F =
F_Qj_w \qquad \txt{and} \qquad F' = F_{Q'}j_w'.\]

We then wish to define the operator on $\ell^2(w)$:  \[ A
= j_w'\widetilde{A}j_w^{-1}.\]

\begin{remark} Defining $A$ in this fashion requires the composition of operators to be defined on $\mathcal{D}$ and also requires $j_w$ be invertible.  We will assume these properties through the remainder of this section, but we note that they may not hold in general.  Thus, this approach requires convergence hypotheses to define $A$, just as were needed in Proposition \ref{Prop:TRA}.   \end{remark}


Given the maps $\alpha$ and $F'$ defined above, we define $T$ to
be the composition \[ T = \alpha F': \ell^2(w) \rightarrow
L^2(\mu),\]  where $T$ maps basis element $e_i^w$ to the function
$\frac{1}{\sqrt{w_i}}\tau^i =\frac{1}{\sqrt{w_i}} v_i \circ \tau$.

We now wish to write $T$ with respect to the isometry $\widetilde{A}$ and the
map $F=F_Qj_w$.  We will need to make use of the operator
$A$, as shown in the following diagram.
\[
\begin{CD}
L^2(\mu \circ \tau^{-1}) @>\alpha>>  L^2(\mu)\\
@AF_{Q'}AA    @AF_QAA\\
\mathcal{H}_{Q'} @>\widetilde{A}>>  \mathcal{H}_Q\\ 
@Aj_w'AA   @Aj_wAA\\
\ell^2(w) @>A>> \ell^2(w)
\end{CD}
\]

\begin{lemma}  Given the operator $T$ defined above, \[ T = Fj_w^{-1}\widetilde{A}j_w' = FA.\]    \end{lemma}
\begin{proof} Recall from the definition of $\widetilde{A}$ that $\alpha F_Q' = F_Q\widetilde{A}$.  We expand $T$ using the diagram above.
\begin{eqnarray*} T &=& \alpha F_w'\\ &=& \alpha F_Q'j_w'\\&=& F_Q\widetilde{A}J_w'\\&=& Fj_w^{-1}\widetilde{A}j_w' = FA. \end{eqnarray*}
\end{proof}

\begin{lemma} There exist weights such that $T$ is a closable operator on the weighted space $\ell^2(w)$.
\end{lemma}

\begin{proof} We have proven that there exist weights $w= \{w_i\}_{i \in \mathbb{N}_0}$ already that make the operators $F$ and $F'$ closable, and in particular, these weights ensure that the bounded functions are in the domains of both $F_w^*$ and $(F'_w)^*$.    Since  $\alpha$ is an isometry, these weights also ensure that $T^*_w$ has dense domain.
\end{proof}

\begin{lemma}\label{Lem:TStarT} Let $w$ be weights such that $F$
and $T$ are both closable operators on $\ell^2(w)$.  If the operator
$T^*_wT$ is self-adjoint and densely defined on $\ell^2(w)$, then
$T_w^*T = H'$. 
\end{lemma}

\begin{proof}
This proof repeats the argument from Lemma \ref{Lem:FStarFisH}:

\begin{eqnarray*}  \langle e_i^w | T^*_wTe_j^w \rangle_{\ell^2(w)} &=& \langle Te_i^w|Te_j^w \rangle_{L^2(\mu)}\\ &=&
\frac{1}{\sqrt{w_iw_j}}\langle \tau^i|\tau^j \rangle_{L^2(\mu)}
\\&=& \frac{1}{\sqrt{w_iw_j}}\langle v_i|v_j \rangle_{L^2(\mu \circ \tau^{-1})}
\\&=& M'_{i,j} = \frac{1}{\sqrt{w_iw_j}} M^{(\mu \circ \tau^{-1})}_{i,j}. \end{eqnarray*}

Therefore, we have \[Q_{M'}(c) = \langle c|M'c\rangle_{\ell^2(w)} =
\|(T^*T)^{1/2}c\|_{\ell^2(w)}^2\] for all $c \in \mathcal{D}$, which proves
$T^*_wT$ is exactly the Kato operator $H'$ for $Q_{M'}$.  This
proves $T^*_wT$ is self-adjoint and densely defined on
$\ell^2(w)$. \end{proof}

Let $A$ be the operator $j_w^{-1}\widetilde{A}j_w'$ on $\ell^2(w)$, so $T = F A$.

\begin{theorem}  The operator $A$ intertwines the
Kato operators $H$ and $H'$, and thus in matrix form intertwines
the moment matrices $M^{(\mu)}$ and $M^{(\mu \circ \tau^{-1})}$.
\end{theorem}
\begin{proof}
If the product $T=FA$ is defined on $\mathcal{D}$, we
have \[ T^*_wT = (FA)^*FA =
A^*(F^*_wF)A,\] which by Lemmas
\ref{Lem:FStarFisH} and \ref{Lem:TStarT} implies that \[H' =
A^*HA.\]

If we form the matrices for $A, H, $ and $H'$ with
respect to the orthonormal basis $\{e_i^w\}_{i \in \mathbb{N}_0}$, we have \[
A^* M^{(\mu)}A = M^{(\mu \circ
\tau^{-1})}.\] 
\end{proof}

Using the Kato approach we can therefore show that, under the appropriate conditions on 
$\tau$ and the measures $\mu$ and $\mu \circ \tau$, we can find a matrix $A$ such that \[
A^*M^{(\mu)}A = M^{(\mu \circ
\tau^{-1})}.\]

\section{Examples}
The following examples illustrate the situations in which infinite matrices do not have direct realizations as $\ell^2$ operators, but can be realized as operators on a certain weighted Hilbert space. The simplest such case arises when the measure $\mu$ is a Dirac mass at the point $b= 1$, which we observed earlier in Example \ref{Ex:PointMass}.  
\begin{example}\label{Ex:FStarDelta1}The Kato operator for $\mu = \delta_1$. \end{example}
Let $\mu = \delta_1$ be the Dirac measure on $\mathbb{R}$ with point
mass at $1$.  As we stated in Example \ref{Ex:PointMass}, the moment matrix $M=M^{(\mu)}$ will be $M_{i,j} = 1$ for all
$i,j\in\mathbb{N}_0$.  The associated quadratic form $Q_M$ is given on $\mathcal{D}$ by
\[ Q_M(c) = \Big|\sum_{k\in \mathbb{N}_0} c_k\Big|^2.\]
If we attempt to define the operator $F^*$ on the constant function $f(x)$ without weights, we have for each $i$, 
\[ (F^*f)_i = \int_{\mathbb{R}}f(x) x^i \,\mathrm{d}\delta_1(x) = f(1).\]
The weights are necessary in order to make any function with $f(1) \neq 1$ be in the domain of $F^*$.

Select  weights $w =\{w_i\}\subset \mathbb{R}^+$ such that
\[ \sum_{i\in\mathbb{N}_0} \frac{1}{w_i} < \infty.\]
Then $F^*_w$ and  $Q_M$ are closable, and the resulting Kato-Friedrichs operator $H_w$ then satisfies  Equation (\ref{Eqn:KatoProperties}).  Given $c$ is in the domain of $H_w$, we have
\[ Q_M(c) = \sum_j\sum_k \overline{c_j}c_k = \langle c|H_wc \rangle_{\ell^2(w)} = \sum_j w_j\overline{c_j}(H_wc)_j\] for all $j \in \mathbb{N}_0$. Thus, $H_w$  
is  defined by
\[ (H_wc)_j = \frac{1}{w_j} \sum_k c_k.\]
In other words, $H_w$ is a rank-one operator with range equal to the span of the vector \[
\xi = \left[\begin{matrix} \frac{1}{w_0} & \frac{1}{w_1} & \frac{1}{w_2} & \cdots \end{matrix}\right]. \] \hfill $\Diamond$

\begin{example}\label{Ex:ExpDecay} The Laplace transform and the operator $F^*$ associated with the measure $e^{-x}\,\mathrm{d}x$ on $\mathbb{R}^+$\end{example}

We first observed in Examples \ref{Ex:Laguerre} and \ref{Ex:Laguerre-2}  that the moment matrix entries 
\[M_{i,j} = (i+j)!\] grow quickly as $i$ and $j$ increase.  Therefore, we again need weights to make $F$ a closable operator.   

Let $w=\{w_i\}_{i \in \mathbb{N}_0}$ be weights given by 
\[ w_i = [(2i)!]^2.\]  We claim that with these weights, $F$ is closable.  
\begin{equation*} \sum_{i=0}^{\infty} \frac{1}{w_i}|M_{i+j}|^2 = \sum_{i=0}^{\infty} \frac{(i+j)!}{[(2i)!]^2}
\end{equation*} for each fixed $j \in \mathbb{N}_0$.  Using the ratio test, we see that the series converges for each $j$ so by Equation (\ref{Eqn:FCriterion}) we have the monomials in the domain of $F^*_w$, which makes $F$ closable on $\ell^2(w)$.

As a side note, we observe that $F_w^*$ has a connection, under appropriate conditions on
the function $f$, to the Laplace transform $\mathcal{L}$:
\[(F_w^*f)_i = \frac{1}{w_i} \int_0^{\infty} f(x) x^i e^{-x}\,\mathrm{d}x = \frac{(-1)^i}{w_i} \Bigl(\frac{\mathrm{d}}{\mathrm{d}s}\Bigr)^i \mathcal{L}_{f}(1),\]
where
\[\mathcal{L}_{f}(s):=\int_0^{\infty} e^{-sx}\xi(x)\,\mathrm{d}x\] for $s \in
\mathbb{R}^+$. 
\hfill$\Diamond$

%% file: integral_08_09_11.tex
\chapter{The integral operator of a moment matrix}\label{Sec:IntOperators}

In this chapter, we will first show how  the Hilbert matrix $M$, which is the moment matrix for Lebesgue measure on $[0,1]$, can be associated to an integral operator.  It turns out that both the operator associated to $M$ and the integral operator are bounded.   In fact, we use properties of Hardy space functions and the polar decomposition to show $M$ and its integral operator are unitarily equivalent.  The bounded operator properties of the Hilbert matrix are well-known.  See \cite{Hal67,  Wid66} for more details.  

   We will generalize Widom's technique, using a weighted Hilbert space when necessary, to find an integral operator on $L^2(\mu)$ associated with the moment matrix $M^{(\mu)}$, where $\mu$ is a positive Borel measure supported on $[-1,1]$.  In the more general setting, the integral operator may not be bounded.   In Chapter \ref{Sec:Spectrum} we will examine conditions under which the integral operator is bounded.   

 When we refer to the Hilbert matrix, where
$\mu$ is Lebesgue measure restricted to $[0,1)$, we will use the
symbol $M$; for any other measure $\mu$ supported in $[-1, 1]$,
the moment matrix will be denoted $M^{(\mu)}$.   
\section{The Hilbert matrix}\label{Sec:Hilbert}
The Hilbert matrix $M$ is the moment matrix for Lebesgue measure
restricted to $[0, 1)$.  $M$ has $(i,j)^{\textrm{th}}$ entry given
by
\begin{equation}
M_{i, j} :=\frac{1}{1 + i + j} = \int_0^1 x^{i+j} \,\mathrm{d}x.
\end{equation}
It is well-known \cite{Hal67, Wid66} that the Hilbert matrix defines a bounded operator on $\ell^2$, and the
operator norm of $M$ as an operator from $\ell^2$ to $\ell^2$ is
$\pi$. 

Using our definition of $\mathcal{D}$ to be the set of
sequences with only finitely many nonzero coordinates, recall the operator
$F:\mathcal{D}\rightarrow\mathcal{P}[0,1]$ given by
\[
Fc(x) = \sum_{n \in \mathbb{N}_0} c_n x^n.
\]
The operator $F$ defines a correspondence between the finite
sequences $\mathcal{D}$ and the polynomials on  $[0,1]$,
and this correspondence extends by Stone-Weierstrass and the Riesz-Fisher theorem to an
operator from $\ell^2$ to $L^2[0,1]$.  We will often call the
image of a sequence $c$ a \textit{generating function} $Fc = f_c$.

Suppose $f\in L^2(0,1)$ and $c\in \ell^2$.  By a Fubini argument and the Cauchy-Schwarz inequality, we can switch the sum and the integral in $\langle f
| Fc \rangle_{L^2(0,1)}$ to define the adjoint operator $F^*$:
\begin{eqnarray*}
\langle f|Fc \rangle_{L^2} &=& \int_0^1 \overline{f(x)}\sum_{n=0}^{\infty} c_nx^n \,\mathrm{d}x \\ &=& \sum_{n=0}^{\infty} c_n \int_0^1 \overline{f(x)}x^n \,\mathrm{d}x \\&=& \langle F^*f | c \rangle_{\ell^2}, \end{eqnarray*} where we have the adjoint now defined by
\[ (F^*f)_n = \int_0^1 f(x) x^n \,\mathrm{d}x.\]
Furthermore, if $f\in L^2(0,1)$, then $\{ (F^*f)_n\}_{n\in\mathbb{N}_0} \in \ell^2$.

\begin{lemma}\label{Lem:FFM} Let $M$ be the operator for the Hilbert matrix, and let $F$ and $F^*$ be the associated operators defined previously.  Then $F^*F = M$.
\end{lemma}
\begin{proof}
Given $c \in \ell^2$, the matrix product $Mc$ is well defined
since $M$ is a bounded operator on $\ell^2$.  We can therefore
write $(Mc)_i = \sum_{j=0}^{\infty} M(i,j)c_j =
\sum_{j=0}^{\infty} \frac{c_j}{i+j+1}$, where convergence is in
the $\ell^2$ norm.  Then

\begin{eqnarray*} (F^*Fc)_i &=& \int_0^1 x^i (Fc)(x) \,\mathrm{d}x\\
&=& \int_0^1x^i\sum_{j=0}^{\infty} c_jx^j \,\mathrm{d}x \\ &=&
\sum_{j=0}^{\infty} c_j \int_0^1 x^{i+j}\,\mathrm{d}x\\&=&
\sum_{j=0}^{\infty} \frac{c_j}{i+j+1}. \end{eqnarray*} The exchange
of the sum and integral above follows from Fubini since $c \in
\ell^2$ and $\left(\frac{1}{i+j+1}\right) \in \ell^2$ (as a sequence in $j$).
\end{proof}

We now turn to a relation between the operator $M$ and an integral
operator $K$.  This relationship is shown in a more general
setting in \cite{Wid66}, and we will also discuss a generalized version in Section \ref{Subsec:Integral}.

\begin{theorem}[\cite{Wid66}, Lemma 3.1]\label{Thm:HilbertKFFM}
Suppose $M = F^*F$ is the operator representing the Hilbert matrix.  The self-adjoint operator $K=FF^*$ on $L^2[0,1]$ is an integral operator with kernel
\[
k(x,y) = \frac{1}{1- xy}.
\]
Moreover, $K$ and $M$ satisfy the relation on
$\ell^2$:
\begin{equation}\label{Eqn:HilbertKFFM}
K F = F M.
\end{equation}

\end{theorem}
\begin{proof}Let $f \in L^2[0,1]$ and fix $x \in (0,1)$.  Then
\begin{eqnarray*} (FF^*f)(x) &=& \sum_{j=0}^{\infty}(F^*
f)_jx^j\\ &=& \sum_{j=0}^{\infty} \int_0^1 f(y) y^jx^j
\,\mathrm{d}y \\ &=& \int_0^1 \sum_{j=0}^{\infty} f(y) (xy)^j
\,\mathrm{d}y \\ &=& \int_0^1 \frac{f(y)\,\mathrm{d}y}{1-xy}.
\end{eqnarray*}
We again use Fubini to exchange the sum and integral above, so $FF^*$ is the desired integral operator.
Combining this equation with Lemma \ref{Lem:FFM}, we find
\[KF = FF^*F = FM.\]   \end{proof}


We next develop tools which we will use in the next section to study the spectrum of the integral operator $K$.
We will demonstrate the relationship between the
integral transform  and the Hardy
space $\mathbb{H}_2$ on the open unit disc $\mathbb{D}$ in $\mathbb{C}$.

The Hardy space $\mathbb{H}_2$ consists of functions which are analytic on $\mathbb{D}$.  The Hardy space norm of a function $F(z)$ with expansion $F(z) = \sum_{n=0}^{\infty} c_nz^n$ is given by \[ \|F\|^2_{\mathbb{H}_2} = \sum_{n=0}^{\infty}|c_n|^2.\]  The inner product on $\mathbb{H}_2$ is exactly the $\ell^2$ inner product of the expansion coefficients of two Hardy functions.

Given $z \in \mathbb{D}$, we let $k_z(x) = (1-zx)^{-1}$.   Then $\widetilde{K}$ denotes the integral operator with kernel $k_z$, which is the extension of $K$ defined in Theorem \ref{Thm:HilbertKFFM} to the complex unit disc $\mathbb{D}$.  Given $g \in L^2[0,1]$, we define
\[ (\widetilde{K}g)(z) = \int_0^1 \frac{g(x)}{1 - \overline{z}x} \,\mathrm{d}x = \langle k_z | g \rangle_{L^2[0,1]}.\]

Note that as a function of $z\in\mathbb{D}$, $k_z$ is analytic for $x\in [-1, 1]$:
\[
\frac{1}{1-zx} = \sum_{n \in \mathbb{N}_0} (zx)^n.
\]

Next, we prove that the range of $\widetilde{K}$ is contained in $\mathbb{H}_2$.  In particular, we show that if $f = \widetilde{K} g$, with $g\in L^2[0,1]$, then $f$ can be extended to an analytic function $F(z)$ on the open unit disc $\mathbb{D}$.

\begin{lemma}Suppose $f\in L^2[0,1]$.  Then
\[F(z) =  \int_0^1 \frac{f(x)}{1 - zx} \,\mathrm{d}x \in \mathbb{H}_2.\]
\end{lemma}
\begin{proof}
Given $x \in [0,1]$ and $|z|<1$, we have \[F(z) = \int_0^1 f(x) \sum_{n=0}^{\infty}x^nz^n \,\mathrm{d}x.\]   Next, observe that $(\int_0^1|f(x)|x^n \,\mathrm{d}x)$ is in $\ell^2$ as a sequence in $n$, since by Cauchy-Schwarz \[\int_0^1|f(x)|x^n \,\mathrm{d}x = \langle|f| | x^n \rangle_{L^2[0,1]} \leq \|f\|_{L^2[0,1]} \|x^n\|_{L^2[0,1]} = \frac{\|f\|_{L^2[0,1]}}{2n+1}.\]  This allows us to use Fubini's theorem to write \[F(z) = \sum_{n=0}^{\infty} \left(\int_0^1f(x) x^n \,\mathrm{d}x \right)z^n\] and also shows that $F \in \mathbb{H}_2$.  \end{proof}

The adjoint operator $\widetilde{K}^*$ to our extended integral operator $\widetilde{K}$ is a map from $\mathbb{H}_2$ to $L^2[0,1]$.   Let $\phi \in \mathbb{H}_2$ have expansion $\phi(z) = \sum_{n=0}^{\infty} c_nz^n$.  Then,

\begin{eqnarray*} \langle \widetilde{K}f|\phi \rangle_{\mathbb{H}_2} &=& \sum_{n=0}^{\infty} \left(\overline{\int_0^1 f(x)x^n\,\mathrm{d}x} \right) c_n \\&=& \int_0^1\overline{f(x)} \sum_{n=0}^{\infty} c_nx^n \,\mathrm{d}x \\ &=& \langle f | \phi \rangle_{L^2[0,1]}. \end{eqnarray*}  We use Fubini's Theorem above to switch the sum and the integral, and note that in the last line, $\phi$ is restricted to the real interval $[0,1]$.  Thus, we find that $\widetilde{K}^*\phi$ is the restriction of $\phi$ to $[0,1]$.
 We now provide an argument using the polar decomposition and Hardy
spaces to show that the operator $F$ is in fact a unitary
operator, so the operators $M$ and $K$ have the same spectrum.
Using the definitions of $F$ and $F^*$, we note that the polynomials are in the domain of $F^*$.  This proves that $F^*$ has a dense domain, and thus $F$ is closable.  We therefore have by a result of von Neumann that both $F^*F$ and $FF^*$ are self-adjoint positive operators with dense domains.

By polar decomposition, given the closable operator $F:\ell^2 \rightarrow L^2[0,1]$, there is a partial isometry $U:\ell^2 \rightarrow L^2[0,1]$ such that
\[F = U (F^*F)^{1/2} = (FF^*)^{1/2}U.\]

Because $U$ is a partial isometry, it is an isometry on the orthogonal complement of the kernel of $F$.  Our goal now is to show that  $U$ is, in fact, a unitary operator, i.e. the kernels of $F$ and $F^*$ are trivial.  This will give us a unitary equivalence between the operators $M$ and $K$.

Given $c \in \ell^2$, let $Fc=0$, that is, $\sum_{n=0}^{\infty} c_nx^n = 0$ for a.e. $x$ in $[0,1]$.  But, by the definition of Hardy functions, $\sum_{n=0}^{\infty} c_nx^n$ is the restriction to $[0,1]$ of an $\mathbb{H}_2$ function \[F(z) = \sum_{n=0}^{\infty}c_nz^n.\]  Since $F$ is analytic and is zero on a set with accumulation points, we have $F(z) = 0$ on $\mathbb{D}$.  Therefore $c_n = 0$ for all $n$, which gives $\mathrm{ker}(F) = \{0\}$.

For a function $f \in L^2[0,1]$, assume that $F^*f = 0$.  Thus, $\int_0^1 f(x) x^n \,\mathrm{d}x = 0$ for all $n \in \mathbb{N}_0$.  This means that $f$ is orthogonal to the monomials, and hence to the polynomials.  Since the polynomials are dense in $L^2[0,1]$ by Stone-Weierstrass, we have that $f(x)=0$ a.e.$x$.  Therefore, $\ker(F) = \{0\}$, and we can now conclude that the operator $U$ is a unitary operator from $\ell^2$ to $L^2[0,1]$.

Given our definition of $U$ from the polar decomposition, we have $U(F^*F)^{1/2} = (FF^*)^{1/2}U$.  This gives \[U(F^*F)^{1/2}U^* = (FF^*)^{1/2}, \] and squaring both sides gives \[ U(F^*F)U^* = UMU^* = FF^* = K.\]


\section{Integral operator for a measure supported on $[-1,1]$}\label{Subsec:Integral}

In this section we study measures $\mu$ with finite moments of all orders on $\mathbb{R}$. We generalize the results found in Section \ref{Sec:Hilbert} for the Hilbert matrix.  Given a measure $\mu$, we see that the moment matrix $M^{(\mu)}$ may be realized in two ways:

\begin{enumerate}[(1)]
\item as an operator $H$ having dense domain in $\ell^2$ or a weighted space $\ell^2(w)$ and acting via formal matrix multiplication by $M^{(\mu)}$; and
\item as an integral operator $K$ in the $L^2(\mu)$-space of all square integrable functions.
\end{enumerate}
While the resulting duality is also true in $\mathbb{R}^d$, for any value of $d$, for clarity we will present the details here just for $d = 1$.  The reader will be able to generalize to $d > 1$.  We will also assume our measures to be supported in $[-1,1]$.

Given the moment matrix $M^{(\mu)}$, define a quadratic form $Q_M$ as in Definition \ref{Def:QM} on the dense space $\mathcal{D}$ consisting of finite sequences.  Recall from Section \ref{Subsec:QClosable} that if the monomial functions $\{v_k\}_{k \in \mathbb{N}_0}$ are in $L^1(\mu) \cap L^2(\mu)$, there exist weights $w=\{w_i\}_{i \in\mathbb{N}_0}$ such that $Q_M$ is a closable quadratic form.  (Note: in some cases such as the Hilbert matrix, the weights are not required.) Moreover, the operator $F: \ell^2(w) \rightarrow L^2(\mu)$ which takes $c \in \mathcal{D}$ to the generating function (a polynomial) $f_c(x) = \sum_j c_jx^j$ is also made closable by the same weights. Then, as we showed previously, 
\[
Q_M(c) = \int |f_c(x)|^2 \, \mathrm{d}\mu(x)
\]
for all $c\in\mathcal{D}$.  By Kato's theory, $Q_M$ has a corresponding closable linear operator $H$ on $\ell^2(w)$ with dense domain satisfying $Q_M(c)= \|H^{1/2}c\|_{\ell^2(w)}^2$.  In particular, the Hilbert space completion $\mathcal{H}_Q$ with respect to the quadratic form $Q_M$ is the same as the completion with respect to the operator $H$. (See Remark \ref{Rem:Distributions} for another viewpoint on this space $\mathcal{H}_Q$.)  We also showed in Lemma \ref{Lem:FStarFisH} that the Kato operator $H$ is given by $F_w^*F$, which shows that the domain of $H$ contains $\mathcal{D}$.  Thus, for $c \in \mathcal{D}$, we have $Q_M(c) = \langle c|Hc \rangle_{\ell^2(w)}$.  

We will show below that the operator $FF^*_w$ on $L^2(\mu)$ is equal on its domain to the integral operator with kernel
\[
k(x,y) = \sum_k \frac{(xy)^k}{w_k}.
\]

Note that if weights are not required, the kernel above reduces to the same kernel which represented the Hilbert matrix in the previous section.  The integral operators corresponding to moment matrices $M^{(\mu)}$ for different measures $\mu$ might have the same kernel, but they will act on different Hilbert spaces $L^2(\mu)$. 

We now can state the association between the Kato operator $F^*_wF$ for a moment matrix $M^{(\mu)}$, the operator $K=F^*_w$, and an integral operator on $L^2(\mu)$.  As usual, we use the notation $v_j(x) = x^j$ for the monomials in $L^2(\mu)$.

\begin{proposition}\label{Lemma:FStar4}
Let $\mu$ be a Borel measure on $\mathbb{R}$ with support contained in $[-1,1]$ and $v_j \in L^1(\mu)$ for all $j \in \mathbb{N}_0$.  Let $F$ be the closure of the operator in (\ref{Eqn:DefnF}) which sends a sequence $c\in\mathcal{D}$ to the polynomial $f_c\in L^2(\mu)$.  Let $w=\{w_i\}_{i \in\mathbb{N}_0}$ be weights such that $F_w^*$ has dense domain as an operator into $\ell^2(w)$,  i.e. via Equation (\ref{Eqn:FCriterion}) we assume that 
\begin{equation}\label{Eqn:SumM}
\sum_{j\in\mathbb{N}_0} \frac{1}{w_j}|M_{j+k}|^2 < \infty \textrm{ for all } k\in\mathbb{N}_0.
\end{equation}
Then the operator $K=\widetilde{F}F^*_w$ is a self-adjoint operator with dense domain.  On its domain, $K$ is an integral operator with kernel 
\[ k(x,y) = \sum_j \frac{(xy)^j}{w_j}.\]
\end{proposition}
\begin{proof}
For the weights $w$,  $F$ is closable and the domain of $F_w^*$ is given by 
\begin{equation}
\Biggl\{\phi \in L^2(\mu) \Big| \sum_{j\in\mathbb{N}_0}w_j |(F_w^*\phi)_j|^2 < \infty\Biggr\}.
\end{equation}
We note in particular that the monomials $\{v_j\}_{j \in \mathbb{N}_0}$ are contained in the domain of $F_w^*$.  The domain of $FF^*_w$ is the set of all functions $\phi$ in the domain of $F^*_w$ such that $F^*_w(\phi)$ is in the domain of $F$.  This set is dense by von Neumann's polar decomposition but may not contain all of the polynomials.  

Let $\phi\in \textrm{dom}(FF^*_w)$. Then a Lebesgue dominated convergence argument gives
\begin{equation}
\begin{split}
({F}F_w^*)\phi(x)
& =\sum_{k\in\mathbb{N}_0} (F^*\phi)_k v_k(x)\\
& \underset{(\ref{Eqn:DefnFStar})}{=} \sum_{k\in\mathbb{N}_0}\frac{1}{w_k} \int_{\mathrm{supp}(\mu)} \phi(y) y^k \,\mathrm{d}\mu(y) x^k\\
& \underset{(\ref{Eqn:SumM})}{=} \int_{\textrm{supp}(\mu)} \phi(y) \sum_{k\in\mathbb{N}_0}\frac{(xy)^k}{w_k} \,\mathrm{d}\mu(y). \end{split}
\end{equation}
The requirement that $\mu$ is supported on $[-1,1]$ ensures that the sum above converges for all $x,y$.
\end{proof}

The polar decomposition of $F$ is
\begin{equation}\label{Eqn:TildeFUFStarTildeF}
F = U(F_w^*F)^{1/2},
\end{equation}
where $U$ is a partial isometry.  We therefore also have
\begin{equation}\label{Eqn:TildeFTildeFFStarU}
F = (FF^*)^{1/2}U.
\end{equation}
The partial isometry $U:\ell^2 \rightarrow L^2(\mu)$ is the same in both  (\ref{Eqn:TildeFUFStarTildeF}) and (\ref{Eqn:TildeFTildeFFStarU}).

\begin{corollary}
The two operators $H = F^*_wF$ and $K=FF^*_w$ have the same spectrum, apart from the point $0$.
\end{corollary}
\begin{proof}A rearrangement of (\ref{Eqn:TildeFUFStarTildeF}) and (\ref{Eqn:TildeFTildeFFStarU}) yields
\begin{equation}\label{Eqn:UHUKStar}
UHU = K^*.
\end{equation} \end{proof}

\begin{example}\label{Ex:ConvexDirac} Convex combinations of Dirac masses. \end{example} 
Let $\mu = \sum_{i=1}^n \alpha_i \delta_{x_i}$, where $\alpha_i \geq 0, \sum_{i=1}^n \alpha_i = 1$ and $x_i,  i=1, 2, \ldots, n$ are distinct real numbers.   Recall from Example \ref{Ex:FStarDelta1} that weights are required for any Dirac measure in order to make $F$ closable.  The same holds for convex combinations of these measures.   It turns out, however, that the integral operator is not so difficult to describe.  We show here that $L^2(\mu)$ has dimension $n$, so the corresponding integral operator must have a representation as a finite matrix.

Suppose the polynomial $f_c(x) = \sum_{k=0}^m c_kx^k$ is equal to the zero vector in $L^2(\mu)$.  In other words, \[ \|f_c\|_{L^2(\mu)}^2 = \int_{\mathbb{R}} \Bigr| \sum_{k=0}^m c_kx^k \Bigr| \mathrm{d}\mu(x) = 0.\]   Using our usual notation, $f_c$ is the image of a sequence $c \in \mathcal{D}$ under the map $F_Q$ from the Hilbert space $\mathcal{H}_Q$.  We recall from Lemma \ref{Lem:KatoIsometry} that $F_Q$ is an isometry, hence \[ \|f_c\|_{L^2(\mu)}^2 = Q_M(c) = \|c\|_{\mathcal{H}_Q}^2 \] where $M=M^{(\mu)}$ is the moment matrix for $\mu$.   The matrix $M$ is of the form \[ M_{j,k} = \sum_{ i=1}^n \alpha_i x_i^{j+k}.\]  
Therefore, if $f_c = 0$ in $L^2(\mu)$, we have
\begin{eqnarray*} \int_{\mathbb{R}} \Bigr| \sum_{k=0}^m c_k x^k \Bigr|^2\mathrm{d}\mu(x) &=& \sum_{k,j=0}^m \overline{c_j}c_k M_{j,k} \\ &=& \sum_{j,k=1}^m \overline{c_j}c_k \sum_{i=1}^n \alpha_i x_i^{j+k} \\ &=& \sum_{i=1}^n \alpha_i \Bigr| \sum_{k=0}^m c_kx_i^k \Bigr|^2 \\&=&0. \end{eqnarray*}
Since each $\alpha_i \geq 0$ and $\sum_{i=1}^n \alpha_i = 1$,  the above sum is zero if and only if \[ \sum_{k=0}^m c_kx_i^k = 0\] for all $i=1, 2, \ldots, n$.  Therefore, the distinct real numbers $x_1, \ldots, x_n$ are roots of any polynomial which is equal to zero in $L^2(\mu)$.  

Let $p(x)$ be the degree-$n$ polynomial \[ p(x) = \prod_{i=1}^n (x-x_i).\]   Using the Euclidean algorithm, any polynomial $q$ can be written in the form $q = ap+r$, where $a,r$ are polynomials and $r$ has degree less than $n$.  In the equivalence classes of $L^2(\mu)$, we therefore have that $q = r$.  Hence, the dimension of $L^2(\mu)$ is $n$.

 \hfill$\Diamond$

\begin{remark}[\cite{JoOl00}]\label{Rem:Distributions}   We can describe another way to generate the Hilbert space
completion $\mathcal{H}_Q$ of a quadratic form $Q$.  So far we have looked at two ways to generate isomorphic Hilbert spaces
$\mathcal{H}_Q$.  The first is to complete the finite sequences
$\mathcal{D}$ with respect to the quadratic form $Q_M$.  The second is the
closure of the polynomials in $L^2(\mu)$, where $\mu$ is the measure whose
moments are the entries of $M$.  Here, we briefly outline a third method
to generate the same Hilbert space which consists of distributions.  For
now, we will restrict ourselves to the concrete case $M=$ the Hilbert
matrix.

As in our first method, we start with the positive semidefinite quadratic
form generated by the moment matrix $M$:
\begin{equation*}
Q_M: [0,1)\times [0,1) \rightarrow \mathbb{R}.
\end{equation*}
We use $M$ to define linear functionals.  For each $x$ in $[0, 1)$, there
is a linear functional $v_x$ defined by
\begin{equation}\label{Defn:vxLinFun}
v_x(\cdot) = Q_M(\cdot, x).
\end{equation}
(The analog to $v_x$ is a sequence in $\mathcal{D}$ with zeros in every
position except one, where there is a $1$.)
We use the $v_x$s to build more linear functionals with finite sums (here,
the analog is $\mathcal{D}$):
\begin{equation}\label{Defn:uLinFun}
u(\cdot) = \sum_{x\textrm{ finite }}c_x v_x(\cdot) =  \sum_{x\textrm{
finite }} Q_M(\cdot, x).
\end{equation}
Finally, there is an inner product in the space of linear functionals of
the form (\ref{Defn:uLinFun}):
\begin{equation}\label{Defn:HtildeInnerProduct}
\Bigl\langle \sum_{x\textrm{ finite }}c_x v_x, \sum_{y\textrm{ finite
}}c_y v_y\Bigr\rangle = \sum_x \sum_y \overline{c_x} c_y Q_M(x,y).
\end{equation}
Now, we complete the inner product space to the Hilbert space
$\widetilde{\mathcal{H}}$.  The Hilbert space $\widetilde{\mathcal{H}}$ is
a reproducing kernel Hilbert space.  It is fairly easy to check that for
$v_x$ as in (\ref{Defn:vxLinFun}) and $u$ as in (\ref{Defn:uLinFun}),
$\langle v_x, u\rangle_{\widetilde{\mathcal{H}}} = u(x)$.

Recall that the Hilbert matrix is related to the integral kernel
$(1-xy)^{-1}$.  We will next explain how this kernel appears in the
Hilbert space of distributions $\widetilde{\mathcal{H}}$.

The tensor product of two distributions $\overline{u_1}$ and $u_2$ is
straightforward; we apply the linear functional $\overline{u_1}\otimes
u_2$ to the pair of functions $(f,g)$ by
\begin{equation}
(\overline{u_1}\otimes u_2)(f(x),g(y)) = \overline{u_1}(f)u_2(g).
\end{equation}
Then, we can extend this definition to functions $h(x,y)$ because the span
of functions of the form $f(x)g(y)$ is dense in $L^2[0,1]$.

One important example of a tensor product of distributions which is
related to Widom's theorem is the following.  Suppose $\phi_1$ and
$\phi_2$ are locally integrable functions on $[0,1)$, and $u_1 = \phi_1
\mathrm{d}x$ and $u_2 = \phi_2 \mathrm{d}x$.  Then
\begin{equation}
(\overline{u_1}\otimes u_2)\Biggl(\frac{1}{1-xy}\Biggr) = \int_0^1
\int_0^1 \frac{\overline{\phi_1}(x)\phi_2(y)}{1-xy} \,\mathrm{d}x \mathrm{d}y.
\end{equation}

The Hilbert space $\widetilde{\mathcal{H}}$ is spanned by distributions
$u$ on such that
\[
(\overline{u}\otimes u) \Biggl(\frac{1}{1-xy}\Biggr) < \infty.
\]
It can be shown that the Dirac mass at $0$ and all its derivatives belong
to $\widetilde{\mathcal{H}}$; moreover, derivatives of the Dirac mass at
$0$ (when scaled appropriately) form an orthonormal basis for
$\widetilde{\mathcal{H}}$ with respect to the inner product
(\ref{Defn:HtildeInnerProduct}).  The specific ONB is
\[
u_n = \frac{(-1)^n}{n!} \delta_0^{(n)}, n = 0, 1, 2, \ldots.
\]
\end{remark}

%% file: spectrum_08_09_11.tex
\chapter{Boundedness and spectral properties}\label{Sec:Spectrum}
In cases such as the Hilbert matrix from Section \ref{Sec:Hilbert}, the operators $F$ and $F^*$ (which might be weighted or unweighted) are bounded operators, although in general they will be unbounded densely defined operators.   In the first part of this chapter, we give sufficient conditions such that these operators are bounded operators, and hence the Kato operator $F^*F$ is also bounded.   In Section \ref{Sec:Spectra}, we use the theory of projection-valued measures described in Section \ref{Sec:pvm} to analyze connections among the spectrum of the Kato operator for a moment matrix $M^{(\mu)}$, the measure $\mu$ itself, and the associated integral operator.  We will demonstrate these connections via the examples in Section \ref{Subsec:SpectrumExamples}.  We define the rank of a measure in Section \ref{Subsec:RankTransformation} and demonstrate with examples which are convex combinations of Dirac measures.

\section{Bounded Kato operators}
Given a measure $\mu$ with finite moments of all orders, denote its moment matrix by $M=M^{(\mu)}$.  We will use the generating function $G_{\mu}(x)$ having the even moments $M_{k,k}$ as coefficients to express conditions which ensure the operator $F^*_w$ on the weighted space $\ell^2(w)$ is bounded.   

\begin{proposition}Let $\mu$ be a measure with compact support on $\mathbb{R}$ with moments of all orders.  Let the generating function
\[
G_{\mu}(x) = \sum_{k=0}^{\infty} M_{k,k} x^k
\]
have a finite and positive radius of convergence $R$.  Select $t$ such that $t > R$ and $\frac{1}{t} < R$.  If we choose weights $w_k := t^k$, then $F^*_w$ is a bounded operator with 
\[
\|F_w^*\|_{op} \leq \Biggl( G_{\mu}\Biggl(\frac{1}{t}\Biggr)\Biggr)^{1/2}.
\]
\end{proposition}
\begin{proof}  Note that since $\frac{1}{t} < R$, then $\frac{1}{t}$ is within the radius of convergence, hence
\[
G_{\mu}\Bigr(\frac{1}{t}\Bigr):=\sum_{k=0}^{\infty} \frac{M_{k,k}}{t^k} < \infty.
\]
We take $w_k:=t^k$ for our weights. 

Now,
\begin{equation}\label{Eqn:FStarWPhi}
(F^*_{w}\phi)_k = \frac{1}{w_k} \int x^k \phi(x) \,\mathrm{d}\mu(x), k\in\mathbb{N}_0.
\end{equation}
So
\[
\Big|
\int x^k \phi(x) \,\mathrm{d}\mu(x)
\Big|^2 
\leq M_{2k}\|\phi\|^2_{L^2(\mu)}
\]
and  
\begin{equation} 
\begin{split}
\|F^*_{w}\phi\|^2_{\ell^2(w)} 
&=\sum_{k=0}^{\infty} \frac{1}{w_k} \Big|\int x^k \phi(x) \,\mathrm{d}\mu(x) \Big|^2\\
&\leq \sum_{k=0}^{\infty} \frac{1}{w_k} M_{2k} \|\phi\|^2_{L^2(\mu)}\\
&=  G_{\mu}\Bigr(\frac{1}{t}\Bigr) \|\phi\|^2_{L^2(\mu)}.
\end{split}
\end{equation}
Therefore, $F^*_w$ is a bounded operator with 
\[
\|F_{w}^*\|^2_{\ell^2(w)\rightarrow L^2(\mu)} 
= \sup_{\|\phi\|=1} \|F_{w}^*\phi\|^2_{\ell^2(w)}
\leq G_{\mu}\Bigr(\frac{1}{t}\Bigr).
\]
Using the properties of the adjoints of bounded operators, we also can conclude that $F$, $F^*_wF$ and $FF^*_w$ are bounded operators with norms given by
\[
\|F_{w}^*\|_{op}^2 = \|F\|_{op}^2 = \|F_{w}^* F_{w}\|_{op} = \|F_{w} F^*_{w}\|_{op} \leq G_{\mu}\Bigr(\frac{1}{t}\Bigr).
\]
\end{proof}

\begin{lemma}Suppose $\mu$ is a measure on $\mathbb{R}$ with compact support and finite moments of all orders.  Suppose in addition that $\textrm{supp}(\mu) \subset [-1,1]$, with $\,\mathrm{d}\mu(x) = B(x) \,\mathrm{d}x$, where $B$ is a nonnegative bounded function.  Then the operators $F$, $F^*$, $F^*F$ and $FF^*$ are bounded, and 

\[ \|F^*\|_{op}^2 = \|F\|_{op}^2 = \|F^* F\|_{op} = \|FF^*\|_{op} \leq \pi \|B\|_{\infty}.\]
\end{lemma}
\begin{proof} By \cite{Wid66} (a straightforward generalization of Theorem \ref{Thm:HilbertKFFM}), we know $FF^*$ is an integral operator of the form
\[
(FF^* \phi)(x)
=\int_{-1}^1 \frac{\phi(y)}{1-xy}\,\mathrm{d}\mu(y) 
=\int_{-1}^1  \frac{\phi(y)B(y)}{1-xy}\,\mathrm{d}y.
\] 
Let $K$ be the positive integral operator defined in Section \ref{Sec:Hilbert} which is equivalent to the moment matrix for Lebesgue measure.  Then
\begin{equation*}
\begin{split}
\langle \phi | FF^*\phi \rangle_{L^2(\mu)}
& = \int_{-1}^1 \int_{-1}^1 \frac{\overline{\phi(x)}B(x)\phi(y)B(y)}{1-xy}\,\mathrm{d}y \mathrm{d}x\\
& = \langle \phi B | K\phi B \rangle_{L^2[-1,1]} \\
& = \langle K^{1/2} \phi B | K^{1/2} \phi B \rangle_{L^2[-1,1]} \\
& \leq \|K^{1/2}\|_{op} \|\phi B\|_{L^2[-1,1]} \\
& \leq \pi \int_{-1}^1 |\phi(x) B(x)|^2 \mathrm{d}x.
\end{split}
\end{equation*}
We recall that the operator norm of $K$ is $\pi$.  In turn, the last expression is bounded above by
\[
\pi\|B\|_{\infty}\int_{-1}^{1}|\phi(x)|^2 B(x) \,\mathrm{d}x = \pi\|B\|_{\infty}\|\phi\|_{L^2(\mu)}.
\]
Therefore, $\|FF^*\|_{op}\leq \pi \|B\|_{\infty}$, and $\|F\|=\|F^*\| \leq \sqrt{\pi\|B\|_{\infty}}$.
\end{proof}

\begin{theorem}\label{Thm:HKBound}Suppose there exists a finite $t>1$ such that $\textrm{supp}(\mu)\subset [-t, t]$, $\mu$ has finite moments of all orders, and $\,\mathrm{d}\mu(x) = B(x) \,\mathrm{d}x$, with $B$ a bounded nonnegative function.  Define weights $w = \{w_k\}_{k=0}^{\infty}$ where $w_k:=t^{2k}$ on $\ell^2(w)$.  Then the operators $F$, $F^*$, $F^*_wF$, $FF^*_w$ are bounded, and the operator norm for $F^*_wF$ and $FF^*_w$ is bounded above by $t\pi \|B\|_{\infty}$.
\end{theorem} 
\begin{proof}For each $k\in\mathbb{N}_0$,
\begin{equation*}
\begin{split}
(F_w^*\phi)_k
& = \frac{1}{w_k}\int_{-t}^t x^k \phi(x) \,\mathrm{d}\mu(x)\\
& = \frac{1}{t^{2k}}\int_{-t}^t x^k \phi(x) B(x) \,\mathrm{d}x\\
& = \frac{t}{t^k}\int_{-1}^1 u^k \phi(tu) B(tu) \,\mathrm{d}u,
\end{split}
\end{equation*}
and
\begin{equation*}
\begin{split}
\|F_w^*\phi\|^2_{\ell^2(w)} 
& = \sum_k \frac{1}{w_k} \Bigg| \int_{-t}^t x^k \phi(x) B(x) \,\mathrm{d}x\Bigg|^2\\
& = t^2 \sum_k \Bigg| \int_{-1}^1 u^k \phi(tu) B(tu)\,\mathrm{d}u\Bigg|^2.
\end{split}
\end{equation*}
But by the previous lemma, this last expression is less than or equal to
\begin{equation*}
\begin{split}
&  t^2 \pi \|B\|_{\infty} \int_{-1}^1 |\phi(tu)|^2 B(tu) \,\mathrm{d}u\\
& = t\pi \|B\|_{\infty} \int_{-t}^t |\phi(x)|^2 B(x) \,\mathrm{d}x\\
& = t\pi \|B\|_{\infty} \|\phi\|^2_{L^2(\mu)}.
\end{split}
\end{equation*}
So, $\|F_w^*\|_{op}^2 = \|FF^*_w\|_{op} \leq t\pi\|B\|_{\infty}$. \end{proof}

The following is an immediate result of Theorem \ref{Thm:HKBound}.
\begin{corollary}\label{Cor:BWeights}
Suppose there exists a finite $t>1$ such that $\textrm{supp}(\mu)\subset [-t, t]$, $\mu$ has finite moments of all orders, and $\,\mathrm{d}\mu(x) = B(x) \,\mathrm{d}x$, with $B$ an nonnegative bounded function.  Define weights $w=\{w_k\}_{k=0}^{\infty}$ where $w_k:=t^{2k}$ on $\ell^2(w)$.  The operator  $FF_w^*$ is an integral operator of the form
\[
(FF^*_w\phi)(x) 
= \int_{-t}^t \frac{1}{1- \frac{xy}{t^2}} \phi(y) B(y)\,\mathrm{d}y.\]
\end{corollary}

We can use Theorem \ref{Thm:HKBound} to study operators from the examples in Chapter \ref{Sec:MomentTheory} where the moments were related to the Catalan numbers. 
\begin{corollary}
Suppose $\mu$ is the Wigner semicircle measure in Example \ref{Ex:Wigner}.  Then under the weights $w$ from Corollary \ref{Cor:BWeights}, $FF^*_w$ is bounded, with $\|FF^*_w\|_{op} \leq 4\pi$.
\end{corollary}

We would also like to apply Theorem \ref{Thm:HKBound} to Example \ref{Ex:Secant}, but we cannot since the function $B$ is not in $L^{\infty}$.  Still, we can provide an estimate.   We use a generating function $G$ defined by
\begin{equation}\label{Eqn:DefnGmu}
G(\zeta) = \sum_{k\in\mathbb{N}_0} \zeta^k \int |x|^k \,\mathrm{d}\mu(x) = \sum \zeta^k \mathcal{M}_k,
\end{equation}
where we recall the use of the notation $\mathcal{M}_k$ from Lemma \ref{Lemma:RClosable} in Section \ref{Sec:GeneralA}.   

\begin{lemma}
Suppose $\textrm{supp}(\mu)\subset [-t, t]$, where $\mu$ has finite moments of all orders.  Then the radius of convergence $R$ of the generating function $G$ from Equation (\ref{Eqn:DefnGmu}) satisfies $R \geq \frac{1}{t}$.
\end{lemma} 
\begin{proof}We have
\[
\Biggl( \int_{-t}^t |x|^k \,\mathrm{d}\mu\Biggr)^{1/k} \leq t ,
\]
where we use the fact that $t := \textrm{ess supp}|x|$ on $L^{\infty}(\mu)$ and also the fact that $\mu$ is a probability measure.  Therefore, we are assured that $G$ converges absolutely for $|\xi| t < 1$.   Therefore, $R \geq \frac{1}{t}$.  \end{proof}

Pick $w_k:=s^k$ for some $t^2 < s < \infty$.   We can now establish an estimate for $\|F^*\phi\|_{\ell^2(w)}$ in terms of the generating function $G$:

\begin{equation*}
\begin{split}
\|F_w^* \phi\|^2_{\ell^2(w)}
& = \sum_{k} \frac{1}{s^k}\Bigg|  \int x^k \phi(x) \,\mathrm{d}\mu(x) \Bigg|^2\\ & \leq \sum_k \frac{1}{s^k} \Biggr(\int_{-t}^t |x|^k |\phi(x)|\mathrm{d}\mu(x) \Biggr)^2. \end{split}\end{equation*}
Then, using Cauchy-Schwarz, we have
\begin{equation*} \begin{split} \|F^*_w\phi\|^2_{\ell^2(w)}  & \leq \sum_k \frac{1}{s^k} \mathcal{M}_k \int |x|^k |\phi(x)|^2 \mathrm{d}\mu(x)\\
& = \int_{-t}^t G\Biggl(\frac{|x|}{s}\Biggr) |\phi(x)|^2 \,\mathrm{d}\mu(x).
\end{split}
\end{equation*}
The switch of the sum and integral above is justified by a Fubini argument, since the power series $G(x)$ is continuous.

We apply this observation to Example \ref{Ex:Secant}, the secant measure.  In this case, $G(x) = 2(\sqrt{1-4x^2})^{-1}$ and $t = 2$.  Choose $s > 4$ and set $w=\{w_k\}_{k=0}^{\infty}$ where $w_k = s^k$.  Then in $\ell^2(w)$, we have 
\[
\|F_w^* \phi\|^2_{\ell^2(w)} \leq \int_{-2}^2 \frac{2}{\sqrt{1-(\frac{2x}{s})^2}} |\phi(x)|^2 \,\mathrm{d}\mu(x).
\]

We next describe a result related to affine iterated function system measures.  In particular, we show a sufficient condition under which the matrix $A$ (see Section \ref{Sec:GeneralA}) encoding an affine map $\tau$ is a Hilbert-Schmidt operator.

\begin{proposition}Consider the transformation $\tau(z) = cx+b$, and suppose $|c|\leq |b| < \frac{1}{4}$.  Then the operator defined by the matrix $A$ which satisfies \[M^{(\mu \circ \tau)} = A^* M^{(\mu)}A,\] with is Hilbert-Schmidt; i.e.  $A$ is bounded, and $\textrm{trace}(A^*A) < \infty$.
\end{proposition}
\begin{proof}Let $k$ denote the row index in $A$, and let $j$ denote the column index in $A$; as before (\ref{Eqn:UpTriA}), we have 
\begin{equation*}
A_{k,j}=
\begin{cases} \binom{k}{j} c^k b^{j-k} & 0 \leq k \leq j\\
                                                              0 & k > j\\
\end{cases}.
\end{equation*}
The $(i,j)^{\textrm{th}}$ entry of the infinite matrix $A^*A$ is therefore
\begin{equation*}
(A^*A)_{i,j}
= \sum_{0\leq k \leq i\wedge j} \overline{A_{k,i}}A_{k,j}
= \sum_{0\leq k \leq i\wedge j} \binom{i}{k} \binom{j}{k} |c^k|^2 \overline{b^{i-k}}b^{j-k},
\end{equation*}
and if $i=j$, the diagonal entries are
\begin{equation}
(A^*A)_{j,j} = \sum_{k=0}^j \binom{j}{k}^2 |c^k|^2 |b^{j-k}|^2.
\end{equation}
Set $\alpha:=|c|^2$ and $\beta:=|b|^2$.
Then, using the assumption that $\frac{\alpha}{\beta}\leq 1$, we have
\begin{equation}\begin{split}
\textrm{trace}(A^*A) 
& = \sum_{j=0}^{\infty} (A^*A)_{j,j}
    = \sum_{j=0}^{\infty} \sum_{k=0}^j \binom{j}{k}^2 \alpha^k\beta^{j-k}\\
&\leq \sum_{j=0}^{\infty} \beta^j \sum_{k=0}^j \binom{j}{k}^2
    = \sum_{j=0}^{\infty} \beta^j \binom{2j}{j}
\end{split}
\end{equation}
The last series has radius of convergence $\frac{1}{4}$.
\end{proof}

\section{Projection-valued measures}\label{Sec:pvm}
This section provides an overview of the properties of
projection-valued measures, leading to the spectral decomposition
for unbounded self-adjoint operators on a Hilbert space
$\mathcal{H}$.  These ideas will be used in Section
 \ref{Sec:Spectra} and later in Chapter \ref{Ch:Extensions} in order to discuss the
Kato-Friedrichs extension of operators.  For far more detailed
treatment of this theory, we refer the reader to
\cite{Bag92,ReSi80,Rud91}.

Recall from Definition \ref{Def:SelfAdjoint} that an operator $H$
with dense domain $\mathrm{dom}(H)$ in a Hilbert space
$\mathcal{H}$ is said to be \textit{self adjoint} if $H^*=H$ and
$\mathrm{dom}(H^*) = \mathrm{dom}(H)$.

\begin{definition} A Borel \textit{projection-valued measure} $E$ on $\mathbb{R}$ is a
map from the $\sigma$-algebra $\mathcal{B}$ of Borel subsets of
$\mathbb{R}$ to the orthogonal projections on a Hilbert space
$\mathcal{H}$ such that
\begin{enumerate}[(i)] \item $E(\emptyset) = 0$, \item if $S =
\cup_{i=1}^{\infty} S_i$ and $S_i \cap S_j = \emptyset$ for $S_i
\in \mathcal{B}$ and $i \neq j$, then $$E(S) = \sum_{i=0}^{\infty}
E(S_i) = \lim_{i\rightarrow \infty} \sum_{j=1}^i E(S_j),$$ where
the limit is in the strong operator topology.\end{enumerate}   A projection-valued measure $E$ is said to be \textit{orthogonal} if
\begin{equation}\label{Eqn:pvmMult} E(S_1 \cap S_2) = E(S_1)E(S_2)
\end{equation} for all Borel sets $S_1,S_2$.  We will assume all projection-valued measures are orthogonal unless otherwise stated.
\end{definition}
Recall from the definition of an orthogonal projection that for
every Borel set $S$, we have $E(S) = E(S)^* = E(S)^2$.

There is a well-known theory which develops the notion of
integration of measurable real-valued functions against a
projection-valued measure $E$, resulting in an operator on the
Hilbert space $\mathcal{H}$.  We will denote such an integral by
$\int_{\mathbb{R}} f(\lambda) E(\mathrm{d}\lambda)$. Quite
naturally, we start by defining integrals of characteristic
functions,
$$E(S) = \int_{\mathbb{R}} \chi_S(\lambda) E(\mathrm{d}\lambda),$$
then extend appropriately to measurable functions.  The operator
is a bounded operator if and only if there exists an $M >0 $ such
that $E(|f|^{-1}(M,\infty)) = 0$ (the trivial projection). We say
that $E$ is a \textit{resolution of the identity} if
\begin{equation} I_{\mathcal{H}} = \int_{\mathbb{R}}
E(\mathrm{d}\lambda) = E(\mathbb{R}). \end{equation}

\begin{example}\cite{Bag92} Given $\mu$ a $\sigma$-finite Borel measure on $\mathbb{R}$, let
$\mathcal{H} = L^2(\mu)$.  Given a Borel set $S \subseteq
\mathbb{R}$, define $E(S)$ to be a multiplication operator which
multiplies by the characteristic function $\chi_S$: $$E(S)(f) =
\chi_S f.$$ Then $E$ is a projection-valued measure on
$\mathbb{R}$, called the \textit{canonical projection-valued
measure}. Given $f$ a measurable real-valued function on
$\mathbb{R}$, $\int_{\mathbb{R}} f(\lambda) E(\mathrm{d}\lambda)$
has domain $\{g \in \mathcal{H}\,:\, fg \in \mathcal{H} \}$ and on
that domain,
$$\int_{\mathbb{R}} f(\lambda) E(\mathrm{d}\lambda)(g) = fg. $$ \hfill
$\Diamond$
\end{example}

One useful property of projection-valued measures is that
integration is multiplicative, where the product of operators is
composition.

$$\int f(\lambda)g(\lambda)E(\mathrm{d}\lambda) = \int f(\lambda)
E(\mathrm{d}\lambda) \int g(\lambda)E(\mathrm{d}\lambda) $$
 Corresponding to a projection-valued measure $E$, there is a parameterized family of real-valued
measures $\{\nu_v\}_{v \in \mathcal{H}}$ given by the
$\mathcal{H}$-inner product:
\begin{equation}\label{Eqn:RealMeasure} \nu_v(S) = \langle v | E(S)v \rangle. \end{equation}
Since orthogonal projections are positive operators, the quantity
$\langle v | E(S)v \rangle$ is nonnegative for any choice of $v$
and any set $S$.  Given a projection-valued measure which is a
resolution of the identity, we have $$ \int_{\mathbb{R}}
\mathrm{d}\nu_v(\lambda) = \langle v | E(\mathbb{R})v \rangle =
\|v\|^2.$$

The following theorem gives a spectral decomposition for
self-adjoint operators.

\begin{theorem}[\cite{DS88,Rud91}] \label{Thm:pvm}An operator $H$ in a Hilbert
space $\mathcal{H}$ is self adjoint if and only if there exists an
orthogonal projection-valued measure $E$ supported on the spectrum
$\sigma(H)$ such that
\begin{equation}\label{Eqn:pvmH} Hv =\left[\int_{\mathbb{R}}\lambda
E(\mathrm{d}\lambda)\right]v =  \left[\int_{\sigma(H)}\lambda
E(\mathrm{d}\lambda)\right]v \end{equation} for all $v \in
\mathrm{dom}(H)$.  Moreover, $v \in \mathrm{dom}(H)$ if and only
if \begin{equation}\label{Eqn:pvmNorm} \|Hv\|^2 =
\int_{\sigma(H)}\lambda^2 \mathrm{d}\nu_v(\lambda) < \infty,
\end{equation} where the measure $\nu_v$ is defined in Equation (\ref{Eqn:RealMeasure}).
\end{theorem}

While we leave the details of the proof to the cited literature,
it may be helpful here to include the computation which
demonstrates the equality given in Equation (\ref{Eqn:pvmNorm}).
Given $v \in \mathrm{dom}(H)$,
\begin{eqnarray*}  \left\|\int \lambda E(\mathrm{d}\lambda)v\right\|^2 &=&
\left\langle \int \lambda E(\mathrm{d}\lambda)v \Bigr| \int
\lambda E(\mathrm{d}\lambda)v \right\rangle \\ &=& \left\langle
\Bigr[\int \lambda E(\mathrm{d}\lambda)\Bigr]^* \Bigr[\int \lambda
E(\mathrm{d}\lambda)\Bigr]v \Bigr| v \right\rangle \\ &=&
\left\langle\Bigr[\int \lambda E(\mathrm{d}\lambda)\Bigr]^2 v
\Bigr| v \right\rangle \quad \text{since the op. is self adjoint}
\\ &=& \left\langle \Bigr[\int
\lambda^2 E(\mathrm{d}\lambda)\Bigr]v \Bigr|v \right\rangle \quad \text{multiplicative prop. of integrals}\\
&=& \int \lambda^2 \mathrm{d}\nu_v(\lambda). \end{eqnarray*}
\hfill $\Box$

The formulas from Theorem \ref{Thm:pvm} together help us define
the domains of (possibly unbounded) self-adjoint operators defined
by integrals of measurable real-valued functions $f$ against a
projection-valued measure $E$.   We say that $v$ is in the domain
of an operator $H = \int f(\lambda) E(\mathrm{d}\lambda)$ if
\begin{equation} \left\|Hv- \Bigr[\int_{-n}^n f(\lambda) E(\mathrm{d}\lambda)\Bigr] v\right\| \rightarrow 0 \end{equation} as
$n \rightarrow \infty$.  If $H$ is a self-adjoint operator, we can
use the equality in (\ref{Eqn:pvmNorm}) which gives
\begin{equation*} \left\|Hv- \Bigr[\int_{-n}^n
\lambda E(\mathrm{d}\lambda)\Bigr] v \right\|^2 =
\int_{|\lambda|>n} \lambda^2 \mathrm{d}\nu_v(\lambda).
\end{equation*}
Therefore, a given vector $v \in \mathcal{H}$ is in
$\mathrm{dom}(H)$ if and only if \begin{equation} \lim_{n
\rightarrow \infty} \int_n^{\infty} \lambda^2
\mathrm{d}\nu_v(\lambda) = 0.
\end{equation}

\section{Spectrum of the Kato operator}\label{Sec:Spectra}

In this section, we show how spectral analysis of the Kato-Friedrichs operator translates into spectral data for the initially given moment matrix $M^{(\mu)}$ and the measure $\mu$.  We then illustrate the theory in several examples.  Our main result (Theorem \ref{Thm:AtomsPair}) is that atoms in the spectrum of the Kato-Friedrichs operator pair up with atoms in the spectrum of the measure $\mu$ itself. Using this we arrive at our own proof of the well-known fact (Theorem \ref{Thm:HilbertSpectrum}) that the Hilbert matrix has continuous spectrum.

In fact, in their spectral picture, the Kato-Friedrichs operators
are directly related to the associated moment matrices
$M^{(\mu)}$; this correspondence is especially transparent for
Examples
\ref{Subsec:KFExampleDelta1} and \ref{Subsec:KFConvex2Dirac}.  In
these examples, the agreement of  the rank of the Kato-Friedrichs
operators with the rank of the measure, which we define in Section \ref{Subsec:RankTransformation}, can be seen by inspection.

\begin{theorem}\label{Thm:AtomsPair} Let $\mu$ be a measure with compact support on the real line. Then the atoms in the spectrum of the Kato-Friedrichs operator of the moment matrix $M =M^{(\mu)}$ (if any) pair up with atoms in the spectrum of the measure $\mu$ itself.
\end{theorem}

Before we begin the proof, we will a few key points about
spectral resolutions and operator theory.

Let  $H$ be a densely defined, positive self-adjoint operator on the
Hilbert space $\mathcal{H}$.  Then there exists (unique, up to
unitary equivalence) a projection-valued measure $E$ such that

\[I_{\mathcal{H}} = \int E(\mathrm{d}\lambda)\] and
\[H = \int_{\sigma(H)} \lambda E(\mathrm{d}\lambda).\]
In addition, for any $x \in \mathcal{H}$, given the spectral measure $\nu_x$ defined in as in Equation (\ref{Eqn:RealMeasure}), we have
\[ \|Hx\|^2 = \int_{\sigma(H)} \lambda^2 \mathrm{d}\nu_x.\]
\begin{definition}\label{Defn:Atom}
An \textit{atom} in $H$ or in $E$ is a point $\lambda_1\in\mathbb{R}$ such
that $E(\{\lambda_1\})\neq 0$, i.e. is not the trivial projection.
\end{definition}

If $\mathcal{H}$ is a complex Hilbert space, then we can define an
order on self-adjoint operators on $\mathcal{H}$:
\[
H\leq K \textrm{ if and only if } \langle x|Hx\rangle \leq \langle
x|Kx\rangle \textrm{ for all } x\in\mathcal{H}.
\]  With this order, we observe that
 if $E_1$ and $E_2$ are projections on $\mathcal{H}$, then
\[ E_1\leq E_2 \Leftrightarrow E_1 = E_1E_2 = E_2E_1.\]
In this case, we say that $E_1$ is a subprojection of $E_2$.

\begin{proof} Let $\textrm{supp}(\mu)\subset\mathbb{R}$, where the measure $\mu$ has compact support.  Denote  the Kato-Friedrichs operator for the moment matrix $M=M^{(\mu)}$ by
$H$.  Let $E$ be the projection-valued measure for $H$.  Recall that in the unweighted case, the mapping $F^*:L^2(\mu)\rightarrow
\ell^2$ which takes $f_c(x)=\sum_{i\in\mathbb{N}_0} c_ix^i$ to
$\{c_i\}_{i\in\mathbb{N}_0}$ is an isometry.

Next, assume that $\lambda_1$ is an atom in the spectrum of $H$.  Therefore, $E({\lambda_1})$ is a nontrivial projection, hence there exists $\xi_1\in
\mathrm{Range}(E({\lambda_1}))$, $\|\xi_1\|_{\ell^2} = 1$.  We denote the rank-one projection onto $\xi_1$ (using Dirac notation;
see Chapter \ref{Sec:Notation}) by
\begin{equation}\label{Eqn:EKetBra}
E_{\lambda_1} = |\xi_1\rangle\langle\xi_1|.
\end{equation}

 $H$ is a positive operator, hence $\lambda_1 > 0$.  Since $H$ is the Kato-Friedrichs operator, we have \[\int |f_c|^2\mathrm{d}\mu = \|H^{1/2}c\|^2_{\ell^2}\] for all $c \in \mathcal{D}$.   Therefore,  there exists
$A_1\in \mathbb{R}^+$ such that
\begin{equation}\label{Eqn:AEH}
A_1E_{\lambda_1} \leq H.
\end{equation}
   Equivalently, by the idempotent property of projections, for all
$c\in\mathcal{D}$ we have
\begin{equation}\label{Eqn:AEH2}
A_1\|E_{\lambda_1}c\|_{\ell^2}^2 \leq \|H^{1/2}c\|_{\ell^2}^2 = \int
|f_c|^2\,\mathrm{d}\mu.
\end{equation}
 Recall that for all $c\in\ell^2$,
\begin{equation}
\langle c|E_{\lambda_1}c\rangle_{\ell^2} = \|E_{\lambda_1} c\|_{\ell^2}^2 = |
\langle\xi_1|c\rangle_{\ell^2}|^2.
\end{equation}

We now make use of the Hankel property from Definition \ref{Defn:HankelN0d} in the moment matrix $M$.  Let
\[ S:\{c_0, c_1, c_2, \ldots\} \rightarrow \{0, c_0, c_1, \ldots\}\]
be the right shift operator in $\ell^2$.  Then the Hankel condition is equivalent
on $\mathcal{D}$ to
\begin{equation}
HS = S^*H.
\end{equation}
By the spectral theorem, then
\begin{equation}
E_{\lambda_1}S = S^*E_{\lambda_1}.
\end{equation}
Looking ahead to the argument in Lemma \ref{Lem:indices}, we see that this implies that the vector $\xi_1$ is of the form \[ \xi_1 = [\begin{matrix} 1&b&b^2&b^3&\cdots \end{matrix}]^{tr}. \]

Equation (\ref{Eqn:AEH2}) gives that for all $c \in \mathcal{D}$, \begin{equation}\label{Eqn:NormEc} A_1\|E_{\lambda_1}c\|_{\ell^2}^2 = A_1 \sum_{i=0}^{\infty} \left| \sum_{j=0}^{\infty} \overline{\xi(i)}\xi(j)c_j \right|^2 = A_1 \sum_{i=0}^{\infty} \left| \sum_{j=0}^{\infty} b^{i+j}c_j \right|^2 \leq \int|f_c|^2\mathrm{d}\mu.\end{equation}

Using Example \ref{Ex:Rank1} regarding the Dirac point-mass measure $\delta_b$ at the point $b$, we have
\[ A_1 \int |f_c|^2\mathrm{d}\delta_b = A_1|f_c(b)|^2 = A_1\left|\sum_{j=0}^{\infty} c_jb^j \right|^2.\]
Since the expression above is equal to the $i=0$ term from Equation (\ref{Eqn:NormEc}), we also have
 \begin{equation}\label{Ineq:Aff}
A_1\int |f_c|^2 \,\mathrm{d}\delta_b \leq \int |f_c|^2\,\mathrm{d}\mu \textrm{ for all
}c\in\mathcal{D}.
\end{equation}

Since the support of $\mu$ is compact, the set
$\{f_c\}_{c\in\mathcal{D}}$ is dense in $L^2(\mu)$.  Therefore, for all Borel sets
$E\in\mathcal{B}(\mathbb{R})$ there exists a sequence $\{c_k\}_{k \in \mathbb{N}_0} \subset \mathcal{D}$ such that
\[
f_{c_k}(x) \underset{L^2(\mu)}{\longrightarrow} \chi_{E}(x).
\]
Therefore, using the sequence $\{f_{c_k}\}$ to approximate
$\chi_{E}$, we get
\begin{equation}
A_1\int_E \,\mathrm{d}\delta_b \leq \int_E \,\mathrm{d}\mu,
\end{equation}
or equivalently,
\begin{equation}\label{Ineq:Adm}
A_1\delta_b(E) \leq \mu(E).
\end{equation}
Since $\delta_b(\{b\}) = 1$, it must be true that $\mu(\{b\}) > 0$.  Therefore, when the Kato-Friedrichs operator $H$ has an atom, the measure $\mu$ has a corresponding atom.  \end{proof}

       The Kato-Friedrichs operator for the moment matrix for Lebesgue measure on the interval $[0,1]$ is the  bounded operator defined relative to the standard ONB in $\ell^2$ by the Hilbert matrix.  In particular, we need not introduce weights into the $\ell^2$ space in order to produce a closable quadratic form.  This is a case when the Kato-Friedrichs operator has infinite rank (see Section \ref{Subsec:RankTransformation}).   We now see that the well-known fact that the spectrum of this operator is purely continuous follows as an immediate corollary of Theorem \ref{Thm:AtomsPair}.  (This result was first shown in \cite{Mag50}.)

\begin{theorem}\label{Thm:HilbertSpectrum}
The spectrum of the Hilbert matrix is continuous, i.e., there are
no atoms in the spectral resolution of the Hilbert matrix.
\end{theorem}
\begin{proof}
       We saw that the action of the Hilbert matrix $M$ on $\ell^2$ defines a bounded (see \cite{Hal67}) self-adjoint operator, and $M$ is the moment matrix $M^{(\mu)}$ where  $\mu$ is taken to be Lebesgue measure on the unit
interval $(0,1)$.  Moreover, the quadratic forms below coincide on the space
$\mathcal{D}$ of finite vectors in $\ell^2$:  
\begin{equation}
\| f_c \|_{L^2(\mu)} = Q_M(c) = \langle c | Mc\rangle_{\ell^2}
\textrm{ for all } c \in \mathcal{D}.
\end{equation}
Therefore, $M$ is the Kato-Friedrichs operator for the quadratic form $Q_M$.  

If $M=M^{(\mu)}$ contained an atom, then there would be a
rank-one projection in the spectral resolution for $M^{(\mu)}$, and by Theorem \ref{Thm:AtomsPair} this would imply that Lebesgue measure restricted to $(0,1)$ would
contain an atom.   That is a clear 
contradiction, and the therefore $M$ has continuous spectrum. \end{proof}

In Section \ref{Subsec:SpectrumExamples}, we explicitly compute the Kato-Freidrichs operators in several examples and study their spectra.  Recall that in general, the Kato-Friedrichs operator $H$ is a  self-adjoint operator on $\ell^2(w)$ for some set of weights $w=\{w_k\}_{k \in \mathbb{N}_0}$.  Note that the operator $H (= H_w)$ depends on the choice of $w$, and the choice of $w$ depends on the measure $\mu$.  The lemma below demonstrates how the weights in a  Hilbert space $\ell^2(w)$ arise in the equations governing the point spectrum of the Kato-Friedrichs operator.

Recall our notation for the weighted Hilbert spaces, where $w=\{w_i\}_{i \in \mathbb{N}_0}$ are  the weights:
\[
\ell^2(w) = \Bigl\{c = \{c_i\}_{i\in\mathbb{N}_0} \Big|
\sum_{i\in\mathbb{N}_0} w_i|c_i|^2 < \infty\Bigr\}
\]
and
\[
\langle b|c\rangle_{\ell^2(w)}:=\sum_{i\in\mathbb{N}_0} w_i \overline{b_i}
c_i.
\]
In the following computations, we assume that  \begin{enumerate} \item the measure $\mu$ is a
positive Borel measure on $\mathbb{R}$ with compact support having
moments of all orders,  \item the weights $\{w_i\}_{i \in \mathbb{N}_0}$ satisfy  $w_i > 0$ for all $i\in\mathbb{N}_0$, \item the monomials $x^i$ are in $L^2(\mu)\cap L^1(\mu)$ for all
$i \in \mathbb{N}_0$. \end{enumerate}
\begin{lemma}\label{Lemma:FundEqns}  Let $\mu$ be a Borel measure on $\mathbb{R}$ with  compact support and finite moments of all orders.  Let $M = M^{(\mu)}$ be the moment matrix for $\mu$, and let $H$ be the Kato-Friedrichs operator for $M$ under an appropriate choice of weights $w = \{w_i\}_{i \in \mathbb{N}_0}$.  There exists $\lambda$ in the point spectrum of $H$, i.e. 

\[
Hc = \lambda c 
\]
for some nonzero $c$ in the domain of $H$, if and only if for each $i \in \mathbb{N}_0$,
\begin{equation}\label{Eqn:FundEqnPtSpec}
\sum_{j\in\mathbb{N}_0}M_{i,j}c_j = \lambda w_ic_i.
\end{equation}
\end{lemma}

\begin{proof}[Proof of Lemma \ref{Lemma:FundEqns}:  ]
The (nonnegative) real number $\lambda$ is in the point spectrum of $H$ if and only if there is a nonzero $c$ in the domain of $H$ such that $Hc = \lambda c$.  By definition of the Kato-Friedrichs operator, $H$ is related to the quadratic form $Q_M$ arising from $M$ via 
\[ Q_M(c) = \sum_j\sum_k \overline{c_j}M_{j,k}c_k = \|H^{1/2}c\|^2_{\ell^2(w)}, \quad \forall c\in\mathcal{D}.\]
When we pass to the completion, given $c\in \textrm{dom}(H^{1/2})$ and $f_c\in
L^2(\mu)$, we have the identity 
\[ \int |f_c|^2 \,\mathrm{d}\mu = \| H^{1/2} c\|^2_{\ell^2(w)}.\]
Recall that $f_c\in L^2(\mu)$ if and only if $c\in
\textrm{dom}(H^{1/2})$; this follows from the completion in the
Kato-Friedrichs construction.

There is also a sesquilinear form $S_M$ determined by $M$ (connected to $Q_M$ by the polarization identity) which satisfies 
\begin{equation}\label{Eqn:Sesq} S_M(b,c) = \sum_j\sum_k \overline{b}_j M_{j,k} c_k = \langle b|Hc \rangle_{\ell^2(w)} \end{equation} for all $b \in \mathcal{D}, c \in \mathrm{dom}(H)$.  

For each $i \in \mathbb{N}_0$, let $b = e_i$, the $i$-th standard basis vector.  Equation (\ref{Eqn:Sesq}) becomes 
\[
\sum_j M_{i,j}c_j = \langle e_i|Hc\rangle_{\ell^2(w)} = w_i(Hc)_i .\]  We now have that $Hc=\lambda c$ for some nonzero vector $c$ in the domain of $H$ 
if and only if
\[\sum_j M_{i,j} c_j = w_i(Hc)_i = w_i \lambda c_i.\]
 \end{proof}

In Equations (\ref{Eqn:FundEqnPtSpec}), we will seek
solutions for $\lambda$, $w$, and $c$.  We require
\[ c\in \ell^2(w) \quad \textrm{and}\quad \|c\|_{\ell^2(w)}=1,\]
i.e.,
\begin{equation}\label{Eqn:SumCWOne}
\sum_{i\in\mathbb{N}_0} w_i |c_i|^2 = 1.
\end{equation}

In the examples shown in Section \ref{Subsec:SpectrumExamples},  we will find that it is relatively easy to solve
(\ref{Eqn:SumCWOne}) for a single Dirac mass $\mu = \delta_b$ for any $b\in \mathbb{R}$.
However, in the case of convex combinations of point masses,
finding solutions to (\ref{Eqn:FundEqnPtSpec}) and
(\ref{Eqn:SumCWOne}) is more tricky.

\section{Rank of measures}\label{Subsec:RankTransformation}
Consider a probability measure $\mu$ on $\mathbb{C}$ or on
$\mathbb{R}^d$, and assume that $\mu$ has finite moments of all orders.
Pick an order of the index set for the monomials, and let
$\{p_k\}_{k \in \mathbb{N}_0^d}$ be the associated orthogonal polynomials in $L^2(\mu)$
defined from $\mu$; further, let $G$ be the corresponding Gram
matrix.  Let $\mathcal{P}$ be the closed subspace in $L^2(\mu)$
spanned by the monomials.
\begin{definition}\label{Defn:Rank}
  We say that the \textit{rank of $G$}, and hence the \textit{rank of the measure $\mu$}, is the
dimension of the Hilbert space $\mathcal{P}$.
\end{definition}

We note two things in Definition \ref{Defn:Rank}.  First,
$\{p_k\}_{k \in \mathbb{N}_0^d}$ is an ONB in $\mathcal{P}$.  Second, $\mathcal{P} =
L^2(\mu)$ if and only if the monomials are dense in $L^2(\mu)$;
this is known to be true if $\mu$ has compact support, but it need
not be true in general.  We will see this more in Chapter \ref{Ch:Extensions}.

\begin{example}\label{Ex:Rank1} A Dirac measure has rank 1.  \end{example}
Let $b$ be a fixed complex number, and let $\mu:=\delta_b$ be the
corresponding Dirac measure.  Then the moments $m_k$ of $\mu$ are
\begin{equation}
m_k = \int_{\mathbb{C}} z^k \,\mathrm{d}\mu(z) = b^k, \:\:k\in\mathbb{N}_0;
\end{equation}
the moment matrix is
\begin{equation}
M^{(\mu)}_{j,k} = \overline{b}^jb^k.
\end{equation}
If $\mathcal{D}$ is the space of finite sequences, $c =
\{c_j\}_{j\in\mathbb{N}_0}$, then
\begin{equation}\label{Eqn:DiracQFb}
\langle c|M^{(\mu)}c\rangle_{\ell^2} = \Big|
\sum_{j\in\mathbb{N}_0}c_jb^j\Big|^2.
\end{equation}

We showed in Example \ref{Ex:ConvexDirac} that $\text{dim}(L^2(\mu)) = 1$, so it follows that $p_0(z) \equiv
1$ and $p_k \equiv 0$ for $k\geq 1$.  We use Equations (\ref{Eqn:GramMatrix}) and (\ref{eqn:grammatrix}) together with Lemma \ref{Lem:GramInverse} to compute the
following formulas for the infinite matrices $G$
and $G^{-1}$:

\begin{equation}
G = \begin{bmatrix}
1 & 0 & 0 & \cdots\\
0 & 0 & 0 & \cdots \\
0 & 0 & 0 & \\
\vdots & \vdots &  & \ddots
\end{bmatrix}
\quad \text{and}\quad G^{-1} = \begin{bmatrix}
1 & 0 & 0 & \cdots\\
b & 0 & 0 & \cdots \\
b^2 & 0 & 0 & \\
\vdots & \vdots &  & \ddots
\end{bmatrix}.\end{equation}
We see here that $G$ has rank $1$.  Note that we are using the notion of inverse described in Section \ref{Sec:Inverses}.  We see that $GG^{-1}$ is an idempotent rank-one matrix.  A direct computation further yields
\begin{equation}
G^{-1}(G^{-1})^* = (b^j\overline{b}^k) = (M^{(\mu)})^{tr},
\end{equation}
as predicted by Lemma \ref{Lem:GramInverse} and Equation (\ref{Eqn:RStarR}). 

\begin{example}\label{Ex:Rank2}A measure with rank 2.\end{example}
Let $\mu:=\frac{1}{2}(\delta_0 + \delta_1)$.  The first moment of
$\mu$ is $1$, and all the other moments are all $\frac{1}{2}$.
The associated orthogonal polynomials are
\begin{equation}
p_0(z) \equiv 1, \qquad p_1(z)  = 2z - 1, \quad \text{and}\quad
p_2(z) = p_3(z) = \ldots  \equiv 0.
\end{equation}
We then compute the Gram matrix and an inverse (in the sense of Definition \ref{def:inverse} and Example \ref{Ex:Inverses}): 
\begin{equation}
G = \begin{bmatrix}
1         & 0         & 0 & \cdots\\
-1        & 2         & 0 & \cdots \\
0         & 0         & 0 &\cdots\\
\vdots & \vdots &       &\ddots
\end{bmatrix}
\quad\text{and}\quad G^{-1} = \begin{bmatrix}
1 & 0 & 0  & \cdots\\
\frac{1}{2} & \frac{1}{2}  & 0 &\cdots \\
\frac{1}{2} & \frac{1}{2}  & 0 &\cdots \\
\vdots        & \vdots        &\vdots &\ddots\\
\end{bmatrix}.
\end{equation}
Again we can compute directly the two idempotents $E_1$
and $E_2$.
\begin{equation}
E_1 = GG^{-1} = \begin{bmatrix}
1 & 0 & 0 & \cdots\\
0 & 1 & 0  & \cdots\\
0 & 0 & 0  & \cdots\\
\vdots & \vdots &\vdots & \ddots
\end{bmatrix},
\end{equation}
and
\begin{equation}
E_2 = G^{-1}G = \begin{bmatrix}
1 & 0 & 0 & 0 & \cdots\\
0 & 1 & 0  & 0 & \cdots\\
0 & 1 & 0  & 0 & \cdots\\
0 & 1 & 0  & 0 & \cdots\\
\vdots & \vdots & \vdots & \vdots &\ddots\\
\end{bmatrix}.
\end{equation}
Note that $E_1$ is a projection; i.e., $E_1 = E_1^* = E_1^2$,
while $E_2$ is only an idempotent; i.e., $E_2 ^2 = E_2$.  Note
that these rules are statements about infinite matrices.  In fact,
if $E_2$ is viewed as an operator in $\ell^2$, it is unbounded
($E_2(\textbf{e}_2)\not\in\ell^2$.) So,
\[\textrm{rank}(G) = \textrm{rank}(\mu) = \textrm{rank}(M^{(\mu)}) = \textrm{dim}(L^2(\mu)) = 2.\]
\hfill$\Diamond$

\section{Examples}\label{Subsec:SpectrumExamples}

In Example \ref{Exam:LebBounded}, we use Lebesgue measure and the corresponding Hilbert matrix to illustrate the case where the moment matrix is a bounded operator.  We also can demonstrate the IFS techniques from Chapter \ref{Sec:Exist} here.  Following that,  we use the measures  $\mu = \delta_1$ and $\mu = \frac12 (\delta_0 + \delta_1)$ to illustrate the spectral results from this chapter.   In both cases, we find appropriate choices for the weights which allow us to find Kato-Friedrichs operators on the Hilbert space $\ell^2(w)$ space.  Since the operator depends on the choice of weights, we will denote the Kato-Friedrichs operator by $H_w$.  In Example \ref{Subsec:KFExampleDelta1}, for a particular choice of $w$,  we get the associated Kato-Friedrichs operators to be a rank-one projection.  We similarly analyze the spectrum of the Kato-Friedrichs operator (which has rank $2$) in Example \ref{Subsec:KFConvex2Dirac}.

\begin{example}\label{Exam:LebBounded}  Lebesgue measure and the Hilbert matrix. \end{example}
Let $\mu$ be Lebesgue measure restricted to $[0,1]$, hence the moment matrix $M^{(\mu)}$ is the Hilbert matrix previously discussed in Example \ref{Ex:Lebesgue} and Section \ref{Sec:Hilbert}.   Recall that it is known (\cite{Wid66}) that the Hilbert matrix represents a bounded operator on $\ell^2$ with operator norm $ \|M^{(\mu)}\|_{op} = \pi$.  

It is well known that $\mu$ is exactly the equilibrium measure arising from the real affine IFS $\{\tau_i\}_{i=0,1}$ where \[\tau_0(x) = \frac{x}{2} \; \textrm{and} \; \tau_1(x) = \frac{x+1}{2}, \] together with equal weights $\frac12$.   The IFS invariance property for moment matrices is \begin{equation}\label{Eqn:HilbertInvar} M^{(\mu)} = \frac12 \Bigr( A_0^*M^{(\mu)}A_0 + A_1^*M^{(\mu)}A_1\Bigr), \end{equation} where Lemma \ref{Lemma:Amatrix} gives the matrices $A_0$ and $A_1$ (from \cite{EST06}):

\[A_0 =  \left[ \begin{matrix} 1&0&0&\cdots\\ 0&\frac12 & 0& \cdots\\0&0& \frac14 & \cdots\\ \vdots & \vdots & & \ddots \end{matrix}\right] \qquad A_1 = \left[ \begin{matrix}  1&\frac12&\frac14&\cdots\\ 0&\frac12 & \frac12& \cdots\\0&0& \frac14 & \cdots\\ \vdots & \vdots & & \ddots \end{matrix}\right].  \]
Since the infinite matrices $A_0$ and $A_1$ are upper triangular (and hence $A_0^*$ and $A_1^*$ are lower triangular), the matrix multiplication computations in (\ref{Eqn:HilbertInvar}) involve only \textit{finite} sums.  In addition, when we check the appropriate convergence issues, each of the matrices $A^*_iM^{(\mu)}A_i = M^{(\mu \circ \tau_i^{-1})}$, for $i=0,1$, represents a bounded operator on $\ell^2$.    We compute directly $A_0^*M^{(\mu)}A_0$:  
\begin{eqnarray*}  (A_0^*M^{(\mu)}A_0)_{i,j} &=&  \sum_{k=0}^{\infty} \sum_{\ell=0}^{\infty} (A_0^*)_{i,k}M^{(\mu)}_{k,\ell}(A_0)_{\ell,j} \\ &=& (A^*_0)_{i,i} \frac{1}{1+i+j} (A_0)_{j,j}\\ &=& \Bigr(\frac{1}{2^{i+j}} \Bigr)\frac{1}{1+i+j} = M^{(\mu \circ \tau_0^{-1})}. \end{eqnarray*}
The entries of the matrix $A_0M^{(\mu)}A_0$ together with the IFS invariance  (\ref{Eqn:HilbertInvar}) property 
\begin{eqnarray*} \frac12 \Bigr( A_0^*M^{(\mu)}A_0 + A_1^*M^{(\mu)}A_1\Bigr)_{i,j} &=& \frac12 \Bigr[ \Bigr(\frac{1}{2^{i+j}}\Bigr) \frac{1}{1+i+j} + \Bigr(1-\frac{1}{2^{i+j}}\Bigr)\frac{1}{1+i+j} \Bigr]\\ &=& \frac{1}{1+i+j} = M^{(\mu)}_{i,j} \end{eqnarray*}
allow us to compute the entries of the matrix $A_1^*M^{(\mu)}A_1$:
\begin{eqnarray}\label{Eqn:A1Matrix}   (A_1^*M^{(\mu)}A_1)_{i,j} &=&  \sum_{k=0}^{\infty} \sum_{\ell=0}^{\infty} (A_1^*)_{i,k}M^{(\mu)}_{k,\ell}(A_1)_{\ell,j}\nonumber\\ &=& \frac{1}{2^{i+j}}\sum_{k=0}^{i} \sum_{\ell=0}^{j} \binom{i}{k} \binom{j}{\ell} \frac{1}{1+k+\ell} \\ &=&
\Bigr(1-\frac{1}{2^{i+j}} \Bigr) \frac{1}{1+i+j}.\nonumber\end{eqnarray}

We note that (\ref{Eqn:A1Matrix}) yields an interesting identity on binomial coefficients.  Given $i,j \in \mathbb{N}_0$, 
\begin{equation}\label{Eqn:Binomial} \sum_{k=0}^i \sum_{\ell=0}^j \binom{i}{k}\binom{j}{\ell} \frac{1}{1+k+\ell} = \frac{2^{i+j}-1}{1+i+j}. \end{equation}

The infinite matrices $A_0$ and $A_1$ represent bounded operators on the Hilbert space $\ell^2$ with \[ \|A_i\|_{op} = \|A_i^*\|_{op} = 1\quad i=0,1 .\]  These are therefore contractive operators on $\ell^2$.
\hfill $\Diamond$

\begin{example}\label{Subsec:KFExampleDelta1} The spectrum of the Kato-Friedrichs operator for  $\mu = \delta_1$. \end{example}

Let $\mu = \delta_1$, the Dirac point mass at $1$.  We computed the Kato-Freidrichs operator $H_w$ in Example \ref{Ex:FStarDelta1}, where the weights $w = \{w_j\}_{j \in \mathbb{N}_0} \subset \mathbb{R}^+$ are chosen such that $\sum_j \frac{1}{w_j} < \infty$.   We found that $H_w$ is a (bounded) rank-one operator on $\ell^2(w)$ defined by
\[ (H_wc)_j = \frac{1}{w_j} \sum_k c_k.\]
Recall that we also know from Example \ref{Ex:ConvexDirac} that the dimension of $L^2(\mu)$ must be exactly $1$.

If we choose weights $\{w_i\}\subset\mathbb{R}^+$
such that
\[ \sum_{i\in\mathbb{N}_0} \frac{1}{w_i} = 1,\]
then $H_w:\ell^2(w)\rightarrow \ell^2(w)$
is in fact a rank-one projection---that is,
\begin{equation}
H_w^2 = H_w = (H_w)^*.
\end{equation}

  In this case, we know from the spectral theory of projections that the spectrum of $H_w$ is exactly $\{0,1\}$, and the operator norm of $H_w$ is $1$.

If we generalize to the case where $\mu = \delta_b$ for some real value $b$, then the moment matrix is \[M_{j,k} = b^{j+k}, \] and the Kato-Friedrichs operator $H_w$ for weights $w$ is a rank-one operator with range the span of \[ \xi_b = \left[ \begin{matrix} \frac{1}{w_0} & \frac{b}{w_1} & \frac{b^2}{w_2} & \cdots \end{matrix}\right]. \]
\hfill $\Diamond$

\begin{example}\label{Subsec:KFConvex2Dirac} The Kato-Friedrichs operator associated with $\mu = \frac{1}{2}(\delta_0 + \delta_1)$ \end{example}
When $\mu = \frac{1}{2}(\delta_0 + \delta_1)$, the moment matrix
$M$ is given by
\begin{equation}\label{Eqn:MConvex2Dirac}
M = \frac{1}{2}
\begin{bmatrix}
1 & 0 & 0 & 0 & \cdots\\
0 & 0 & 0 & 0 & \cdots\\
0 & 0 & 0 & 0 & \\
\vdots &\vdots & & & \ddots
\end{bmatrix}
+ \frac{1}{2}\begin{bmatrix}
1 & 1 & 1 & 1 & \cdots\\
1 & 1 & 1 & 1 & \cdots\\
1 & 1 & 1 & 1 & \\
\vdots &\vdots & & & \ddots
\end{bmatrix}.
\end{equation}
We wish to solve the eigenvalue equations
(\ref{Eqn:FundEqnPtSpec}), which in this case become
\begin{align}\label{Eqn:FirstSystem}
c_0 + \frac{1}{2}(c_1+c_2+\cdots)  &= \lambda w_0 c_0, \quad k =0 \notag\\
\frac{1}{2}(c_0+c_1 + c_2 + \cdots)  &= \lambda w_kc_k, \quad k \geq 1.\notag\\
\end{align}
which can be transformed into
\begin{equation}\label{Eqn:SecondSystem}
\frac{1}{2}c_0 = \lambda(w_0c_0 - w_kc_k), \quad k \geq 1.
\end{equation}
by subtracting the later equations in (\ref{Eqn:FirstSystem}) from
the first.

To solve for $c$ in Equation (\ref{Eqn:SecondSystem}), we consider the
cases where $\lambda = 0$ and where $\lambda\neq 0$.

\noindent \textbf{Case 1:  }If $\lambda = 0$, $c_0$ must also be
$0$.  In addition, referring back to the equations in the system
(\ref{Eqn:FirstSystem}), we easily see that
Equation (\ref{Eqn:FundEqnPtSpec}) is true whenever
\begin{equation}\label{Eqn:SumK1Ck}
 \sum_{k = 1}^{\infty} c_k = 0.
\end{equation}
(For example, $c = (0, 1, -1, 0, 0, \ldots)$.)  There is no restriction on the weights
$w= \{w_k\}_{k \in \mathbb{N}_0}$ coming from the equations here.  So for any choice of weights, the eigenspace for $\lambda=0$ is infinite-dimensional. 

\noindent \textbf{Case 2:  }If $\lambda \neq 0$,  the
$k^{\mathrm{th}}$ equation for $k > 0$ in the system of equations
(\ref{Eqn:SecondSystem}) can be solved for $c_k$ in terms of $\lambda$, $c_0$, and the weights $w
= \{w_k\}_{k \in \mathbb{N}}$:
\begin{equation}\label{Eqn:CkDirac2Convex}
c_k = \frac{c_0}{w_k}\Bigl(w_0 - \frac{1}{2\lambda}\Bigr), \quad k
\geq 1.
\end{equation}
We see that  $c_0 = 0$ implies that $c_k = 0$ for all $k$, so we
require $c_0 \neq 0$.  Substituting (\ref{Eqn:CkDirac2Convex})
into the first equation in (\ref{Eqn:FirstSystem}), and cancelling
$c_0$ from both sides, we obtain a condition on $\lambda$ and the
weights $w$:
\begin{equation}\label{Eqn:lambdaW}
1 +
\frac{1}{2}\Bigl(\sum_{k=1}^{\infty}\frac{1}{w_k}\Bigr)\Bigl(w_0
- \frac{1}{2\lambda}\Bigr) = \lambda w_0.
\end{equation}
Define $T = \sum_{k=1}^{\infty} \frac{1}{w_k}$, where this sum being finite imposes a restriction on our choice of weights $w$.  Then the equation becomes \[ 1+\frac{1}{2}T\Bigr(w_0-\frac{1}{2\lambda}\Bigr) = \lambda w_0.\]

Observe that the two conditions
\[ 1 = \lambda w_0 \quad \textrm{ and }\quad 2\lambda w_0 = 1\]
will lead to inconsistent systems, so we must rule these conditions out.  Otherwise, there are many solutions to Equations (\ref{Eqn:FundEqnPtSpec}). 

We now demonstrate that there will be two distinct nonzero eigenvalues for the Kato-Friedrichs operator $H_w$, where the weights $w$ are chosen so that $T$ is finite.  Without loss of generality, take $w_0 = c_0 = 1$.  Our eigenvalues $\lambda \neq 0$ must satisfy Equation (\ref{Eqn:lambdaW}), which is now \begin{equation}\label{Eqn:SimplelambdaW}1+\frac12 T \Bigr(1-\frac{1}{2\lambda} \Bigr) =\lambda. \end{equation}
This yields the quadratic equation in $\lambda$:
\[
\lambda^2 - \Bigl( 1 + \frac{T}{2} \Bigr)\lambda + \frac{T}{4} = 0.
\]
This equation has distinct positive roots $\lambda_+, \lambda_-$ for $T>0$, and they are given by
\begin{equation}\label{Eqn:lambdapm}
\lambda_{\pm} = \frac{1+ \frac{T}{2} \pm \sqrt{1 +
\Bigl(\frac{T}{2}\Bigr)^2}}{2};
\end{equation}
and we also see 
\begin{equation}\label{Eqn:TraceDet}
\lambda_+\lambda_- = \frac{T}{4} \textrm{ and } \lambda_+
+\lambda_- = 1 + \frac{T}{2}.
\end{equation}

Since the vector $c$ satisfying $Hc = \lambda c$ (where $\lambda \neq 0$) is uniquely determined from $\lambda, c_0$, and $w$, given a choice of weights and taking $c_0=1$, each eigenspace for nonzero $\lambda$ is one-dimensional.  Denote by $c_+$ and $c_-$ these eigenvectors in the eigenspaces for $\lambda_+$ and $\lambda_-$ respectively.  


Going back to (\ref{Eqn:CkDirac2Convex}) we have the
simplification for the eigenvectors:
\begin{equation}\label{Eqn:CkSub}
c_k = \frac{1}{w_k}\Bigl(1 - \frac{1}{2\lambda}\Bigr) =
\frac{1}{w_k}\Bigl(\frac{\lambda - \frac{1}{2}}{\lambda}\Bigr),
\quad k \geq 1.
\end{equation}
With respect to  our eigenvalues $\lambda_{\pm}$, this gives 
\[
c_{\pm} = \Bigl( 1,
\frac{\lambda_{\pm}-\frac{1}{2}}{w_1\lambda_{\pm}},
\frac{\lambda_{\pm}-\frac{1}{2}}{w_2\lambda_{\pm}},\ldots\Bigr);
\]
then
\[
\|c_{\pm}\|^2_{\ell^2(w)} \underset{(\ref{Eqn:CkSub})}
{=} 1 +
\Biggl(\frac{\lambda_{\pm}-\frac{1}{2}}{\lambda_{\pm}}\Biggr)^2 T.
\]
The reader can verify the orthogonality of the eigenvectors. 

In order to do more explicit computations, let us choose the weights  $w =
\{w_k\}_{k\in\mathbb{N}_0}$ such that $T = 2$.  This yields
\[
 \lambda_{\pm} = 1 \pm
\frac{1}{\sqrt{2}},
\]
which are both positive. 

Substituting back into (\ref{Eqn:CkSub}) we get
\begin{equation}\label{Eqn:NotNormEigv}
c_{\pm} = \Bigl( 1, \pm\frac{1}{\sqrt{2}w_1},
\pm\frac{1}{\sqrt{2}w_2}, \ldots\Bigr),
\end{equation}
with
\[
\|c_{\pm}\|^2_{\ell^2(w)} = 2.
\]
Normalize to obtain unit vectors
\begin{equation}\label{Eqn:NormEigv}
\xi_{\pm}: = \frac{1}{\sqrt{2}}c_{\pm}
\end{equation}
which yield the rank-one projections
\[
E_{\pm}:=|\xi_{\pm}\rangle \langle\xi_{\pm}|
\]  onto the eigenspaces.  In other words, the spectral resolution of the self-adjoint operator $H$ can
be written
\begin{equation}\label{Eqn:SpecResConvex2Dirac}
H = \Bigl(1 + \frac{1}{\sqrt{2}}\Bigr)E_+ + \Bigl(1 -
\frac{1}{\sqrt{2}}\Bigr)E_-.
\end{equation}

We next make two observations regarding the connections between $H^{1/2}$, the square root of the Kato-Friedrichs opertator, and the Hilbert space $L^2(\mu)$.  First, we recall from Equation (\ref{Eqn:KatoProperties}) and Lemma \ref{Lem:KatoIsometry} the isometry
\begin{equation}\label{Eqn:KatoIsometry}
\int |f_c|^2\,\mathrm{d}\mu = \|H^{1/2} c\|^2_{\ell^2(w)};
\end{equation}
that is, the identification $f_c \leftrightarrow H^{1/2}c$ is
isometric.  By choosing specific weights, we can verify this isometry using our Kato-Friedrichs operator for $\mu = \frac12(\delta_0+\delta_1)$.  

Let $w = \{w_k\}$ be given by the sequence with $w_0 = 1$ and 
\[
w_k = 2^{k-1}, \quad k \geq 1,
\]
which gives $T = 2$.  Denote $F(\xi_{\pm})$ by the shorthand  $f_{\pm}$.  Using (\ref{Eqn:NotNormEigv}) and (\ref{Eqn:NormEigv}), we get
\begin{equation}
f_{\pm}(x) = \sum_{k=0}^{\infty} (\xi_{\pm})_k x^k =
\frac{1}{\sqrt{2}} \pm \frac{x}{2-x},
\end{equation}
and, because $\mu = \frac{1}{2}(\delta_0 + \delta_1)$, we have
\[
\int |f_{\pm}(x)|^2\,\mathrm{d}\mu(x) = 1 \pm \frac{1}{\sqrt{2}} =
\frac{\sqrt{2}\pm 1}{\sqrt{2}}.
\]
Using Equation (\ref{Eqn:SpecResConvex2Dirac}), we have
\[
H^{1/2} = \Bigl(1 + \frac{1}{\sqrt{2}}\Bigr)^{1/2}E_+ + \Bigl(1 -
\frac{1}{\sqrt{2}}\Bigr)^{1/2}E_-.
\]
Moreover,
\begin{equation}
H^{1/2}\xi_{\pm} = \Bigl( 1 \pm
\frac{1}{\sqrt{2}}\Bigr)^{1/2}\xi_{\pm}
\end{equation}
and hence
\[
\|H^{1/2}\xi_{\pm}\|^2_{\ell^2(w)} = 1 \pm \frac{1}{\sqrt{2}} \]
which verifies the isometry in (\ref{Eqn:KatoIsometry}).

Define the map $W: L^2(\mu) \rightarrow \ell^2(w)$ given by
\begin{equation}\label{Eqn:WConvex2Dirac}
W(f_c) = H^{1/2}c \textrm{ for }c\in\mathcal{D}\subset \ell^2(w).
\end{equation}
Recall from Example \ref{Ex:ConvexDirac} that the dimension of $L^2(\mu)$ is exactly $2$.  Our second observation is that the map $W$ is an isometry into a two-dimensional subspace of $\ell^2(w)$.

We will check this fact directly. Recall $E_{\pm}
= |\xi_{\pm}\rangle \langle \xi_{\pm}|$, with $\xi_{\pm}$ as in
(\ref{Eqn:NormEigv}).  We only need to check that $W$ takes an ONB
in $L^2(\mu)$ to an orthonormal family in $\ell^2(w)$.  First, observe that the orthogonal
polynomials in $L^2(\mu)$ are $p_0(x)\equiv 1$, $p_1(x) = 2x -1$,
and $p_k(x)\equiv 0$ for $k \geq 2$.  Moreover,
\begin{equation*}
\begin{split}
Wp_0
& = Wf_{(1, 0, 0, \ldots)}\\
& \underset{(\ref{Eqn:WConvex2Dirac})}{=} \Bigl( 1 +
\frac{1}{\sqrt{2}} \Bigr)^{1/2}E_+(1, 0, 0, \ldots) +
\Bigl( 1 - \frac{1}{\sqrt{2}} \Bigr)^{1/2}E_-(1, 0, 0, \ldots)\\
& \underset{(\ref{Eqn:NormEigv})}{=} \Bigl( 1 + \frac{1}{\sqrt{2}}
\Bigr)^{1/2} \frac{1}{\sqrt{2}}\xi_+ + \Bigl( 1 -
\frac{1}{\sqrt{2}} \Bigr)^{1/2} \frac{1}{\sqrt{2}}\xi_-,
\end{split}
\end{equation*}
and
\begin{equation*}
\begin{split}
Wp_1
& = Wf_{(-1, 2, 0, 0, \ldots)}\\
& = \Bigl( 1 + \frac{1}{\sqrt{2}}\Bigr)^{1/2} \Bigl(1 -
\frac{1}{\sqrt{2}}\Bigr)\xi_+ +
\Bigl( 1 - \frac{1}{\sqrt{2}}\Bigr)^{1/2} \Bigl(-1 - \frac{1}{\sqrt{2}}\Bigr)\xi_-\\
& = \Bigl( 1 - \frac{1}{\sqrt{2}} \Bigr)^{1/2}
\frac{1}{\sqrt{2}}\xi_+ - \Bigl( 1 + \frac{1}{\sqrt{2}}
\Bigr)^{1/2} \frac{1}{\sqrt{2}}\xi_-.
\end{split}
\end{equation*}

Since $Wp_0$ and $Wp_1$ are orthonormal in $\ell^2(w)$, we have that $W$ is an isometry.  

 \hfill $\Diamond$
 
\begin{example} \label{Subsec:KFExampleDeltab} The Kato-Friedrichs operator associated with $\mu = \frac{1}{2}(\delta_0 + \delta_b)$, $0 < b <1$ \end{example}

This example is a more general case where $\mu$ is a finite convex combination of Dirac masses, hence the associated infinite Hankel matrix will be of finite rank. To simplify matters we pick the atoms in $\mu$ in such a way that no weights will be needed. When we work with
an unweighted $\ell^2$ space, $w_k = 1$ for all $k \in
\mathbb{N}_0$, and $\ell^2 = \ell^2(w)$.
The moment matrix $M = M^{(\mu)}$ for
the convex combination $\mu$ is equal to the Kato-Friedrichs
operator, as in the case of the Hilbert matrix. 
This last example also has the advantage of illustrating an ONB in $L^2(\mu)$ consisting of rational functions; not
polynomials. The choice of the rational functions is dictated by
our Kato-Friedrichs operator $H = M$.

As in the previous Example \ref{Subsec:KFConvex2Dirac},
$H$ has a spectrum consisting of $0$ and two points on the
positive real line, and $L^2(\mu)$ is two-dimensional. The point
$0$ has infinite multiplicity, and the two positive eigenvalues
are simple, i.e., have multiplicity one.

For $\mu = \frac{1}{2}(\delta_0 + \delta_b)$, $0 < b < 1$, the
moment matrix $M$ is
\begin{equation}
M = H  =
\begin{bmatrix}
1 & \frac{1}{2}b & \frac{1}{2}b^2 & \cdots\\
\frac{1}{2}b & \frac{1}{2} b^2 & \frac{1}{2} b^3 & \cdots\\
\frac{1}{2}b^2 & \frac{1}{2} b^3 & \frac{1}{2} b^4 & \cdots\\
\vdots & \vdots & & \ddots\\
\end{bmatrix}.
\end{equation}
The orthogonal polynomials in $L^2(\mu)$ are
\begin{equation}
p_0(x)\equiv 1,\quad  p_1(x) = 1 - \frac{2x}{b},\quad p_k(x)
\equiv 0, k \geq 2.
\end{equation}

However, we also have an orthogonal basis in $L^2(\mu)$ consisting of the two
rational functions $f_{\pm}(x)$, where
\begin{equation}
f_{\pm}(x):= \alpha_{\pm} + \frac{bx}{1 - bx}.
\end{equation}
These correspond to the eigenvectors for $H$.  We see that if \[ \xi_{\pm} = [ \begin{matrix} \alpha_{\pm} & b &b^2&b^3& \cdots \end{matrix}]^{tr}, \] then $H\xi_{\pm} = \lambda \xi_{\pm}$, where  the parameters $\alpha_{\pm}$ are given by 
\begin{equation*} \alpha_{\pm} = 1-\frac{p}{2} \pm \sqrt{ \left(\frac{p}{2}\right)^2 + 1} \end{equation*} and the two positive eigenvalues of $H$ are \begin{equation*}
\lambda_{\pm} = \frac{1}{2}\left( 1+\frac{p}{2} \pm \sqrt{\left(\frac{p}{2}\right)^2+1} \right), \end{equation*}
where $p$ is the constant  \[ p = \frac{b^2}{1-b^2}.\]
The functions above are the images of $\xi_{\pm}$ under the operator $F$:
\[ f_{\pm}(x) = (F\xi_{\pm})(x) = \alpha_{\pm} + \frac{bx}{1-bx}.\] 
The eigenvectors $\xi_{\pm}$ are orthogonal, and it follows that $f_{\pm}$ are orthogonal in $L^2(\mu)$:
\begin{eqnarray*} \langle f_+|f_-\rangle_{L^2(\mu)} &=& \langle F\xi_+| F\xi_- \rangle_{L^2(\mu)}\\
 &=& \langle \xi_+|F^*F \xi_- \rangle_{\ell^2} \\ &=& \langle \xi_+|H\xi_- \rangle_{\ell^2}\\
 &=& \lambda_- \langle \xi_+|\xi_- \rangle_{\ell^2} = 0 \end{eqnarray*} 
Direct computation of $\langle f_+|f_-\rangle_{L^2(\mu)}$ also verifies that these functions are orthogonal.

This example also demonstrates the correspondence from Theorem \ref{Thm:AtomsPair} between atoms of the measure $\mu$ and atoms in the projection-valued measure $E$ associated to $H$.   The measure $\mu$ has two atoms at $0$ and $b$, while the projection valued measure $E$ for $H$ has atoms at the two nonzero eigenvalues $\lambda_{\pm}$.  
\hfill $\Diamond$

%% file: extensions_08_09_11.tex
\chapter{The moment problem revisited}\label{Ch:Extensions}

 Given a positive definite
Hankel matrix $M$ with real entries, we showed in Section
\ref{Subsec:exist} that  there exists a measure $\mu$ (not unique in general) such that $M$ is the moment matrix of $\mu$, i.e., $M =
 M^{(\mu)}$.  In this chapter, we will describe a setting in which
 the solution measure is not unique.  Differing measures solving the same moment problem will arise from nontrivial
 self-adjoint extensions of a symmetric shift operator $S$.  In
 fact, we will show that if $\mu$ is not unique, the self-adjoint extensions of $S$ yield a one-parameter family of  measures
 satisfying the moment problem for $M$.  This will yield a necessary condition for non-uniqueness of measure for a given Hankel matrix $M$.  For more details on the
 theory of self-adjoint extensions of unbounded symmetric
 operators and moments, see \cite{Con90,ReSi75,Rud91}.

Given a Hankel matrix $M$ with real entries, let $Q_M$ be the
quadratic form from Equation (\ref{Def:QM}) and let
$\mathcal{H}_Q$ be the Hilbert space completion of $Q_M$. Then the
inner product on $\mathcal{H}_Q$ is defined on the finite
sequences $\mathcal{D}$ by $$\langle c | d \rangle_{\mathcal{H}_Q} =
\sum_i\sum_j \overline{c}_i M_{i,j} d_j.$$

We define the shift operator $S$ by
\begin{equation}\label{Eqn:DefnS}
Sc = S(c_0,c_1,\ldots,) = (0,c_0,c_1,\ldots).
\end{equation} The domain of $S$ contains
$\mathcal{D}$ which is dense in $\mathcal{H}_Q$, so $S$ is densely
defined.

In Section \ref{Sec:ThreeInc} we observe some of the interplay
between the matrix $M$, the shift operator $S$, and the isometry
$F$.  Then, in Sections \ref{Sec:Extensions} and \ref{Sec:Nonunique} we describe how the
self-adjoint extensions of $S$ yield solutions to the moment
problem $M = M^{(\mu)}$.  We also use the shift operator in Section \ref{Sec:Banded} in order to find a Jacobi matrix corresponding to a given Hankel matrix $M$.

\section{The shift operator and three incarnations of symmetry}\label{Sec:ThreeInc}

Suppose $M$ is a positive definite Hankel matrix and $\mu$ is a measure such that $M = M^{(\mu)}$.  Recall that we have defined an isometry $F:\mathcal{H}_Q\rightarrow L^2(\mu)$ by
\[Fc(x) = f_c(x) = \sum_{n\in\mathbb{N}_0}c_n x^n \textrm{ for all
}c\in\mathcal{D}.\]
  The shift operator $S$
plays a major role in this chapter, and in this section we study
how $F$ and $S$ behave with respect to each other.  The operators
$F$ and $S$ reveal three different incarnations of symmetry in the
Hankel matrix $M$.  In turn, these incarnations will be used in
subsequent sections to prove results about non-uniqueness of
measures which solve the moment problem, particularly in Theorem
\ref{Thm:measures}.

\noindent\textbf{Incarnation 1:  A symmetric operator.  } We see
here the connection between the Hankel property of $M$ and the
shift operator in the Hilbert space $\mathcal{H}_Q$.
\begin{lemma}\label{Lemma:Symmetry} The shift operator $S$ is symmetric in
$\mathcal{D}\subset\mathcal{H}_Q$ if and only if the matrix $M$ is a Hankel matrix.
\end{lemma}

\begin{proof}
($\Rightarrow$):  Setting $b = e_i$, $c = e_j$, and $\langle
 Sb|c\rangle_{\mathcal{H}_Q} = \langle b|Sc \rangle_{\mathcal{H}_Q}$, we see that $M_{i-1, j} = M_{i,
 j-1}$.

($\Leftarrow$):  Let $b,c\in\mathcal{D}$.  Then
\begin{equation}
\begin{split}
\langle Sb|c\rangle_{\mathcal{H}_Q} & = \sum_i \sum_j \overline{b_{i-1}} M_{i+j}
c_j = \sum_{i}\sum_{j}
 \overline{b_i} M_{i+1+j}c_j\\
& = \sum_i \sum_j \overline{b_i} M_{i+j} c_{j-1} = \langle
 b|Sc\rangle_{\mathcal{H}_Q}.
\end{split}
\end{equation}\end{proof}

\noindent\textbf{Incarnation 2:  Multiplication by $x$.  }We notice that if $p$ and $q$ are polynomials in $L^2(\mu)$, then
\[
\int \overline{x p(x)} q(x) \,\mathrm{d}\mu(x) = \int \overline{p(x)} xq(x) \,\mathrm{d}\mu(x).
\]
We can therefore state the interaction between the isometry $F$
and the shift $S$.
\begin{lemma}\label{Lemma:MultiplicationByX}Define $\widetilde{S}:=FSF^*$.  Then $\widetilde{S}$ is a
symmetric operator on $\mathcal{P}\subset L^2(\mu)$, and in addition, $\widetilde{S}$ is a restriction {\rm(}to its domain{\rm)}
of the multiplication operator $M_x$:\[[M_xf](x) = xf(x).\]
\end{lemma}

\begin{proof}
Let $M_x$ be the operator which takes $f(x)$ to $xf(x)$.  Using the definitions of $F$ and $S$, we see that
\begin{equation}\label{Eqn:MxFFS}
M_x(Fc) = FSc  \textrm{ for all }c\in\mathcal{D},
\end{equation}
or stated equivalently,
\[ xf_c(x) = f_{Sc}(x) \textrm{ for all }c\in\mathcal{D}.\]

Since $F$ is an isometry, we know that $F^*F$ is the identity on
$\mathcal{H}_Q$. As a result, applying $F^*$ on the left to both
sides of Equation (\ref{Eqn:MxFFS}) yields
\[ S = F^*M_x F \textrm{ on } \mathcal{D}.\]

Also, $P = FF^*$ is the projection of $L^2(\mu)$ onto the closure
of the polynomials $\mathcal{P}\subset L^2(\mu)$.  Applying $F^*$
on the right to both sides of Equation (\ref{Eqn:MxFFS}) yields
\[ M_xP = FSF^* = \widetilde{S} \textrm{ on }\mathcal{P},\]
which is the desired conclusion.
\end{proof}

\noindent\textbf{Incarnation 3:  Jacobi matrices and orthogonal
polynomials.  } Looking ahead in Section \ref{Sec:Banded}, we
see that given the space $L^2(\mu)$, there is a
Jacobi matrix $J$ which encodes the multiplication operator $M_x$ in terms of a recursion relation for
orthogonal polynomials $\{p_k\}_{k \in \mathbb{N}_0}$ in $L^2(\mu)$.  Specifically,  \[J \left[\begin{matrix} p_0\\p_1\\ \vdots \end{matrix} \right] = \left[ \begin{matrix} M_x p_0 \\ M_xp_1\\ \vdots \end{matrix}\right] . \]  Lemma \ref{Lemma:WIntertwine} then gives an intertwining relationship between the shift $S$ and this Jacobi matrix $J$.


\section{Self-adjoint extensions of a shift
operator}\label{Sec:Extensions} Recall from Definition
\ref{Def:SelfAdjoint} that a densely defined operator $S$ on a
Hilbert space $\mathcal{H}$ is \textit{symmetric} on
$\textrm{dom}(S) \subseteq \mathcal{H}$ if $\langle Sh| k\rangle =
\langle h| Sk\rangle$ for all $h,k\in
\textrm{dom}(S)$.  Also recall that we can define the
\textit{adjoint} $S^*$ of $S$, as in Definition \ref{Def:Adjoint},
and we call $S$ a \textit{self-adjoint} operator if $S=S^*$, in
particular, if $\mathrm{dom}(S) = \mathrm{dom}(S^*)$.

\begin{definition}\label{Defn:SAExt} Suppose $S$ is a densely defined
  operator on $\mathcal{D}$ in a Hilbert space $\mathcal{H}$.  A
  \textit{self-adjoint} extension $T$ of $S$ satisfies the following
  properties:
\begin{enumerate}
\item $T$ is self-adjoint with $\textrm{dom}(T) =
\textrm{dom}(T^*)$ \item $Tc = Sc$ for all $c \in \mathcal{D}$.
\end{enumerate}
It may be the case that no self-adjoint extensions of $S$ exist.
\end{definition}
If $S$ is essentially self-adjoint, we call the closure of $S$ a
trivial self-adjoint extension of $S$. Finally, we say that $S$ is
\textit{maximally symmetric} if $S$ has no proper self-adjoint
extensions.

Let $S$ be the closure of the shift operator from Equation (\ref{Eqn:DefnS}).   By
Lemma \ref{Lemma:Symmetry}, $S$ is symmetric in $\mathcal{D}$ if
and only if the matrix $M$ is a Hankel matrix.  The deficiency indices of $S$ will allow us to describe its self-adjoint extensions, if any exist.

\begin{definition}\label{Def:DeficiencySpace} Let $S$ be a closed
symmetric operator on a Hilbert space $\mathcal{H}$ and let
$\alpha \in \mathbb{C}$ with $\mathrm{Im}(\alpha) \neq 0$.  We
define the \textit{deficiency subspace} of $S$ at $\alpha$ by
\begin{equation}\label{Eqn:Halpha} \mathcal{L}(\alpha) = \mathrm{null}(S^*-\alpha) = \{\xi \in \mathrm{dom}(S^*)\,:\,
S^*\xi = \alpha \xi\}.\end{equation} \end{definition}

Due to a beautiful argument of von Neumann, the dimension of the
deficiency subspace will be the same for any $\alpha$ with
$\mathrm{Im}(\alpha) > 0$, and the dimension will be the same for any
$\alpha$ with $\mathrm{Im}(\alpha) < 0$.  It is therefore sufficient
to consider $\alpha = \pm i$.  We will denote the respective
deficiency subspaces for the shift operator $\mathcal{L}_+$ and
$\mathcal{L}_-$.

\begin{definition}
The dimensions of $\mathcal{L}_+$ and
 $\mathcal{L}_-$ for a closed symmetric operator $S$ are called the \textit{deficiency indices} of $S$ and are denoted by the ordered pair $(\textrm{dim}(\mathcal{L}_+),
 \textrm{dim}(\mathcal{L}_-))$.  Note that the indices can take
 on any value in $\mathbb{N}_0$ or $\infty$.
\end{definition}

\begin{definition} Let $S$ and $T$ be linear operators with dense domains in a Hilbert space, and let $\mathrm{Gr}(S), \mathrm{Gr}(T)$ be their corresponding graphs.  We say that $S \subseteq T$ if $\mathrm{Gr}(S) \subseteq \mathrm{Gr}(T)$.  \end{definition}

Note that an operator $S$ is self-adjoint if and only if $S \subseteq S^*$.  If $T$ is a self-adjoint extension of $S$, then \[ S \subseteq T \subseteq T^* \subseteq S^*.\]  Given a bounded operator $J$, we write $JS \subseteq SJ$ if $J$ maps the domain of $S$ into itself and $JSv = SJv$ for all $v \in \mathrm{dom}(S)$.  

Another theorem of von Neumann gives a
condition under which the deficiency indices of a symmetric
operator are equal.
\begin{theorem}\rm (von Neumann, as stated in \cite[Prop. 7.2, p. 343]{Con90}) \label{Thm:EqualIndices} \it Given an
operator $S$ on a Hilbert space $\mathcal{H}$, if there exists a
function $J: \mathcal{H} \rightarrow \mathcal{H}$ satisfying the
following properties:

\begin{enumerate}[(1)]
\item $J^2$ is the identity on $\mathcal{H}$,  \item $J$ is
conjugate linear---that is, for all $\alpha\in\mathbb{C}$,
 $J(\alpha h) = \overline{\alpha}J(h)$,
\item $\|Jh\| = \|h\|$ for all $h\in\mathcal{H}$,  \item $J
 \mathrm{dom}(S) \subseteq \mathrm{dom}(S)$ and $JS \subseteq SJ$,
\end{enumerate}
then $S$ has equal deficiency indices, i.e.
 $\dim(\mathcal{L}_+) = \dim(\mathcal{L}_-)$.
\end{theorem}
The main idea in the proof of this theorem is that the operator
$J$ restricts as an
 isometry between the deficiency subspaces $\mathcal{L}_+$ and
 $\mathcal{L}_-$, thus showing they have equal dimension.

We apply this theorem to our shift operator $S$ on the space
$\mathcal{H}_Q$.  Let $J$ be the conjugation operator:
\begin{equation}\label{Eqn:DefnJ}
J:\mathcal{H}_Q \rightarrow \mathcal{H}_Q \textrm{ with }J(c_0,
c_1,
 c_2, \ldots):=(\overline{c_0}, \overline{c_1}, \overline{c_2},
 \ldots),.
\end{equation}
If $\xi \in \mathcal{L}_+$, then $J\xi \in \mathcal{L}_-$.  It
is readily verified that $J$ satisfies the properties in Theorem
\ref{Thm:EqualIndices}. In particular we see that on $\mathcal{D}
= \mathrm{dom}(S)$, $J$ commutes with $S$.

We now calculate the deficiency indices for $S$, the closure of
the shift operator.
\begin{lemma}\label{Lem:indices}
The closed shift operator $S$ {\rm(}\ref{Eqn:DefnS}{\rm)} defined on
$\mathcal{D}\subset
 \mathcal{H}_{M}$ is either self-adjoint or
has deficiency indices $(1, 1)$---i.e.
\[
\dim\{\xi \in \mathrm{dom}(S^*)\,:\, S^*\xi = i\xi\}=1.
\]
 In fact,
if $\xi\in\mathcal{L}_+$ with $\xi \neq 0$, then $H\xi$ is a
multiple of the vector $(1, i, i^2, i^3, \ldots)$, where $H$ is
the self-adjoint Kato operator for the quadratic form $Q_M$  on $\mathcal{H}_Q$.
\end{lemma}
\begin{proof}  If $S$ is self-adjoint, its deficiency indices are $(0,0)$.  Suppose there exists $\xi \neq 0$
  such that $\xi\in\mathcal{L}_+$.
We will compute $(H\xi)_i$ via the inner product $\langle \cdot |
 \cdot\rangle_{\mathcal{H}_Q}$ to show that
\begin{equation}\label{Eqn:AlphaVector}
H\xi \in\mathbb{C}(1, i, i^2, i^3, \ldots).
\end{equation}
Note $e_i$ belongs to $\mathcal{D} = \textrm{dom}(S)$ for every
 $i\in\mathbb{N}_0$.  We now compare $(H\xi)_i=\langle e_i | \xi\rangle_{\mathcal{H}_Q}$
 and $(H\xi)_{i+1}= \langle e_{i+1} | \xi\rangle_{\mathcal{H}_Q}$:
\begin{equation}
 \langle e_{i+1} | \xi\rangle_{\mathcal{H}_Q}
= \langle S e_i | \xi\rangle_{\mathcal{H}_Q} = \langle e_i | S^*\xi \rangle_{\mathcal{H}_Q} =
i \langle e_i | \xi\rangle_{\mathcal{H}_Q},
\end{equation}
which implies by induction that $H\xi$ is a multiple of the vector
 (\ref{Eqn:AlphaVector}).

Since $H$ has trivial kernel in $\mathcal{H}_Q$, we can conclude
that $\dim \mathcal{L}_+ = 1$ and
 by Theorem \ref{Thm:EqualIndices}, $\dim \mathcal{L}_- = 1$ as well.  \end{proof}

We next use a theorem, again quoted almost verbatim from \cite{Con90}, which states that the
self-adjoint extensions of $S$ are determined by the partial
isometries from $\mathcal{L}_+$ to $\mathcal{L}_-$.
\begin{theorem}\cite[Theorem 2.17, p. 314]{Con90}\label{Thm:Isom} Let $S$ be a closed
  symmetric operator.  If $W$ is a partial isometry with initial space
  in $\mathcal{L}_+$ and final space in $\mathcal{L}_-$, then there
  is a closed symmetric extension $S_W$ of $S$ on the domain
  \[\{f+g+Wg\,:\, f \in \mathrm{dom}(S), g \in \mathrm{initial}(W)\}\]
  given by
 \[S_W(f+g+Wg) = Sf+ig-iWg.\]  Conversely, if $T$ is any closed
 symmetric extension of $S$, then there is a unique partial
 isometry $W$ such that $T=S_W$ as defined above.
 \end{theorem}

 In the case of our shift operator $S$, we see that the only
 nontrivial partial isometries from $\mathcal{L}_+$ to
 $\mathcal{L}_-$ are isometries between the one-dimensional spaces.
 These isometries are given by multiplication by $z \in \mathbb{C}$ where $|z|=1$.
 \begin{theorem}\label{Thm:extensions} Given a Hankel matrix $M$ and the associated
 Hilbert space $\mathcal{H}_Q$, let $S$ be the closure of the symmetric shift operator
 {\rm(}\ref{Eqn:DefnS}{\rm)}.  If $S$ is not self-adjoint, then it has
self-adjoint extensions $T_z$ which have domain
$\mathcal{D}+\mathcal{L}_+ + \mathcal{L}_-$ and are exactly
given by
\[T_z(c+\xi+J\xi) = Sc+i\xi-izJ\xi \quad \mathrm{for}\, z \in
\mathbb{C}.\]
\end{theorem}
\begin{proof}
By Lemma \ref{Lem:indices}, if $S$ is not self-adjoint, it has
deficiency indices $(1,1)$.  If $\xi \in \mathcal{L}_+$, then
$J\xi \in \mathcal{L}_-$.  Given $z \in \mathbb{C}$ with
$|z|=1$, we can define an isometry $T_z:\mathcal{L}_+ \rightarrow
\mathcal{L}_-$ by $T\xi = zJ\xi$.  In fact, every such isometry
is of this form since the spaces are one-dimensional.  The result follows from Theorem \ref{Thm:Isom}.
\end{proof}

\section{Self-adjoint extensions and the moment problem}\label{Sec:Nonunique}
Next, we describe how the self-adjoint extensions $T_z$ to $S$
described in Theorem \ref{Thm:extensions} yield solutions $\mu_z$
to the moment problem $M=M^{(\mu)}$.  (For background, see
\cite{Con90}.)

As we discussed in Section \ref{Sec:pvm}, every self-adjoint
operator $T$ can be written in terms of a projection-valued
measure $E$ such that
\begin{equation}\label{Eqn:ResolutionT}
T = \int_{\mathbb{R}} \lambda E(\mathrm{d}\lambda).
\end{equation}

\begin{proposition}\cite[Prop. 7.2, p. 343]{Con90}\label{Prop:NonuniqueMeasure}
Given an infinite Hankel matrix $M$, suppose $T$ is a self-adjoint
nontrivial extension of the shift operator $S$ on the Hilbert
space $\mathcal{H}_Q$ with corresponding
 projection-valued measure $E$ {\rm(}\ref{Eqn:ResolutionT}{\rm)}.  Given $E$, we can define a real measure as we did in Equation {\rm(}\ref{Eqn:RealMeasure}{\rm)}, by
\[
\mu(\cdot)  := \langle e_0 | E(\cdot)e_0\rangle_{\mathcal{H}_Q},
\]
where $e_0$ is the first standard basis vector $(1, 0, 0,
\ldots)$.  Then the real-valued measure $\mu$ is a solution to the moment
problem for $M$---that is,
\begin{equation*}
\int x^i \,\mathrm{d}\mu(x) = M_i \textrm{ for all }i\in\mathbb{N}_0.
\end{equation*}
\end{proposition}
\begin{proof} From the multiplicative property of integrals against projection-valued measures,
\[\int_{\mathbb{R}} \lambda^i E(\mathrm{d}\lambda) = T^i\] for each $i \in
\mathbb{N}_0$.  This gives the corresponding result for $\mu$:\[
\int_{\mathbb{R}} x^i \mathrm{d}\mu(x) = \langle e_0 | T^i e_0
\rangle .\]  Since $e_0 \in \mathcal{D}$ and $T$ maps $\mathcal{D}$ to $\mathcal{D}$, we have the following
calculation:

\begin{equation}
\begin{split}
\int x^i \,\mathrm{d}\mu(x) & = \langle e_0 | T^i e_0\rangle_{\mathcal{H}_Q}
 = \langle e_0 |S^i e_0\rangle_{\mathcal{H}_Q}
= \langle e_0|e_i\rangle_{\mathcal{H}_Q}\\
& = M_{0, i} = M_i.
\end{split}
\end{equation}\end{proof}

We can now associate to each self-adjoint extension to the shift operator $S$ a measure which satisfies
the moment problem for the matrix $M$, in the case where the shift
operator is not essentially self-adjoint.

\begin{theorem}\label{Thm:measures} Given an infinite Hankel matrix $M$ with $M_{0,0} = 1$, let $S$ be
the closure of the shift operator on $\mathcal{H}_Q$.
Then the
following statements are equivalent. \begin{enumerate}
\item $S$ is not self-adjoint.

\item  The set of distinct solutions to the moment problem $M=M^{(\mu)}$ is a one-parameter family of probability measures $\{\mu_z\,:\, z\in \mathbb{C}, |z|=1\}$.

\item There exist two nonequivalent probability measures $\mu_1$ and $\mu_2$ which are both solutions to the moment problem, i.e.  $M=M^{(\mu_1)} = M^{(\mu_2)}$.

\item Given any measure $\mu$ solving the moment problem $M = M^{(\mu)}$, the polynomials are not dense in the space $L^2(\mu)$.

\end{enumerate}

\end{theorem}
\begin{proof}

  \noindent $(1) \Rightarrow (2)$: If $S$ is not self-adjoint, then by
  Theorem \ref{Thm:extensions} there is a one-parameter family
  $\{T_z\,:\, z \in \mathbb{C}, |z|=1\}$ of distinct (and not
  unitarily equivalent) self-adjoint extensions to $S$.  For each
  $T_z$, we can define a measure $\mu_z$ as in Proposition
  \ref{Prop:NonuniqueMeasure} which satisfies the moment problem
  $M=M^{(\mu)}$.  It remains to be shown that these measures are
  distinct.

  The isometry $F:\mathcal{H}_Q \rightarrow L^2(\mu_z)$ which maps
  $T_z^ke_0$ to $x^k$ extends to map $\psi(T_z)(e_0)$ to the function
  $\psi$.  Thus, because the span of the functions
  $\{\frac{1}{x-\alpha}\,:\, \alpha \in \mathbb{C},
  \mathrm{Im}(\alpha) \neq 0\}$ is dense in each space $L^2(\mu_z)$,
  the measure $\mu_z$ is uniquely determined by these functions.  The
  extensions $T_z$ are cyclic operators on $\mathcal{H}_Q$, which
  means (see \cite{Con90}) that the set of vectors
  $\{(T_z-\alpha I)^{-1}e_0 \,:\, \alpha \in \mathbb{C},
  \mathrm{Im}(\alpha) \neq 0\}$ is dense in $\mathcal{H}_Q$.  It
  follows that $\mu_z$ is determined uniquely by the isomorphism $F$
  mapping the vector $(T_z-\alpha)^{-1}e_0 \in \mathcal{H}_Q$ to the
  function $\frac{1}{x-\alpha} \in L^2(\mu)$.  Since the operators
  $T_z$ are distinct, this proves the measures $\{\mu_z\,:\, z \in
  \mathbb{D}, |z|=1\}$ are distinct.

\noindent $(2) \Rightarrow (3)$: Follows because each $\mu_z$ is distinct.

\noindent $(3) \Rightarrow (4)$: We prove the contrapositive. Suppose
the space of polynomials $\mathcal{P}$ is dense in $L^2(\mu)$ where
$M=M^{(\mu)}$.  Recall the map $F:\mathcal{D} \rightarrow \mathcal{P}$
given by $Fc = \sum_i c_ix^i$ is an isometry on $\mathcal{D}$ which
extends to an isometry on $\mathcal{H}_Q$.  Assume $\xi \in
\mathcal{L}_+$, so $S^*\xi = i\xi$ and let $c \in \mathcal{D}$.  Then
\begin{equation*}
\begin{split}
& \int_{\mathbb{R}}\overline{xFc(x)}F\xi(x) \mathrm{d}\mu(x)
  = \langle FSc|F\xi \rangle_{L^2(\mu)} \\
& = \langle Sc | \xi \rangle_{\mathcal{H}_Q}
  = \langle c| S^*\xi \rangle_{\mathcal{H}_Q}
  = i \langle c|\xi \rangle_{\mathcal{H}_Q}\\
& = i\langle Fc | F\xi \rangle_{L^2(\mu)}
  =  i \int_{\mathbb{R}} \overline{Fc}(x)F\xi(x)\mathrm{d}\mu(x).
\end{split}
\end{equation*} Because $F$ gives a one-to-one correspondence between
$\mathcal{D}$ and $\mathcal{P}$, we can say for any polynomial $p \in
\mathcal{P}$ and $\xi \in \mathcal{L}_+$,
\begin{equation}\label{Eqn:XP} \int_{\mathbb{R}}
  (x-i)\overline{p(x)}F\xi(x) \mathrm{d}\mu(x) = 0.\end{equation}
The function $\frac{1}{x-i} \in L^{\infty}(\mu)$, hence
$\frac{F\xi}{x-i} \in L^2(\mu)$.  Let $\{p_n\} \subset \mathcal{P}$
converge in $L^2(\mu)$ to $\frac{F\xi}{x-i} \in L^2(\mu)$.  Substituting into Equation (\ref{Eqn:XP}) then gives \[\int_{\mathbb{R}}
|F\xi(x)|^2 \mathrm{d}\mu(x) = 0,\] hence $\xi = 0$.  Therefore, $S$
has deficiency indices $(0,0)$.

\noindent $(4) \Rightarrow (1)$: Assume the polynomials are not dense in
$L^2(\mu)$.  As before, $M = M^{(\mu)}$.  Let $\psi \in L^2(\mu)$ be a
nonzero bounded vector orthogonal to the subspace $\mathcal{P}$
spanned by the polynomials.  Recall that $L^2(\mu)\ominus \mathcal{P}\neq 0$ if and only if $\mathcal{P}$ is not dense in $L^2(\mu)$.  Given $\alpha \in \mathbb{C}$ with
 $\mathrm{Im}(\alpha) \neq 0$, define \[\xi_{\alpha} =
\frac{\psi(x)}{x-\alpha}.\] Then, because each function
$\frac{1}{x-\alpha}$ is bounded, we have $\xi_{\alpha} \in
L^2(\mu)$.  Note that the function $[M_x\xi_{\alpha}](x) =
x\xi_{\alpha}(x)$ is also an $L^2(\mu)$ function because
$x\xi_{\alpha}(x) = \psi(x)+\alpha \xi_{\alpha}(x)$.

The isometry $F$ between $\mathcal{H}_Q$ and $L^2(\mu)$ resulting
from Equation (\ref{Eqn:DefnF}) ensures that $FF^*$ is the projection
onto the range of $F$, which is the closed space spanned by the
polynomials in $L^2(\mu)$.  Therefore, we know $F^*\psi = 0$.  Given
the function $\xi_{\alpha}$, let $\phi_{\alpha} \in \mathcal{H}_Q$ be
defined by $\phi_{\alpha} = F^*\xi_{\alpha}$.  We first must show that
for at least one choice of $\alpha$, $\phi_{\alpha} \neq 0$.

Suppose that for all $\alpha \in \mathbb{C} \setminus \mathbb{R}$,
$\xi_{\alpha}$ is orthogonal to the polynomials.  Define $\Psi$ to be
the linear span of the functions $\{\frac{1}{x-\alpha}: \alpha \in
\mathbb{C}, \mathrm{Im}(\alpha) \neq 0\}$.  It is well-known that
$\Psi$ is dense in $L^2(\mu)$ when $\mu$ is a finite real measure, so
let $\{\psi_k\}_{k \in \mathbb{N}_0} \subset \Psi$ be a sequence of functions which
converges to $\overline{\psi}$ in $L^2(\mu)$.  This gives
\begin{equation*}
\begin{split}
& \int_{\mathbb{R}} x^{\ell} \psi(x) [\overline{\psi(x)}- \psi_k(x)]\mathrm{d}\mu(x)\\
& =\int_{\mathbb{R}} x^{\ell} |\psi(x)|^2 \mathrm{d}\mu(x) - \int_{\mathbb{R}} x^{\ell} \psi(x)\psi_k(x) \mathrm{d}\mu(x)\\
&\rightarrow 0 \; \mathrm{as}\, k \rightarrow \infty.
\end{split}
\end{equation*}
The terms $\psi(x)\psi_k(x)$ are linear combinations of $\xi_{\alpha}$
functions, so the second integral in the sum above is zero for all
$k,\ell \in \mathbb{N}_0$.  Therefore, for all $\ell \in
\mathbb{N}_0$, \[\int_{\mathbb{R}} x^{\ell} |\psi(x)|^2
\mathrm{d}\mu(x)=0.\] In particular, the $\ell=0$ case implies that
$\psi = 0$.  This contradicts our choice of $\psi$, hence we know
there must be some function $\xi_{\alpha}$ which is not orthogonal to
the polynomial space $\mathcal{P}$.

  Because \[ \int_{\mathbb{R}} (x-\alpha)\overline{p(x)}\xi_{\alpha}(x) \mathrm{d}\mu(x) = \int_{\mathbb{R}} \overline{p(x)} \psi(x) \mathrm{d}\mu(x) = 0\] for all $p \in \mathcal{P}$, we have \[
\langle M_xp|\xi_{\alpha} \rangle_{L^2(\mu)} = \alpha \langle p| \xi_{\alpha} \rangle_{L^2(\mu)},\] where $M_x$ is the multiplication operator by $x$.  Every polynomial $p$ is the image of a finite sequence $d \in \mathcal{D}$ under $F$.  Recall that $FS=M_xF$ for all $d \in \mathcal{D}$, and let $\phi_{\alpha} = F^*\xi_{\alpha} \neq 0$.  Applying $F^*$ in the above equation gives \[\langle F^*M_xp|F^*\xi_{\alpha} \rangle_{\mathcal{H}_Q} = \langle Sd|\phi_{\alpha} \rangle_{\mathcal{H}_Q} = \alpha\langle d|\phi_{\alpha} \rangle_{\mathcal{H}_Q} = \alpha\langle F^*p|F^*\xi_{\alpha} \rangle_{\mathcal{H}_Q}. \]Therefore, $\phi_{\alpha}$ satisfies the equation $S^*\phi = \alpha\phi$, hence $\phi_{\alpha} \in \mathcal{L}(\alpha)$ for $S$ which proves $S$ is not self-adjoint.

\end{proof}

The following example shows a Hankel matrix $M$ which does not
have a unique moment problem solution.

\begin{example}A non-unique measure.\end{example}\label{Ex:NotUnique}
Set
\[
f(x) := \int_{\mathbb{R}} \cos(xt) e^{-(t^2 + 1/t^2)} \,\mathrm{d}t.
\]
Then
\[
\int x^i f(x) \,\mathrm{d}x
= \Bigl(\frac{\mathrm{d}}{\,\mathrm{d}t}\Bigr)^i e^{-(t^2 +1/t^2)}\Big|_{t=0}
= 0 \textrm{ for all }i\in\mathbb{N}_0.
\]
Let $f_{+}$ be the function $\max (f,0)$ and let $f_{-}$ be the
 function $-\min (f,0)$.
Set $f = f_{+} - f_{-}$ and $\,\mathrm{d}\mu_{\pm}(x) = f_{\pm}(x) \,\mathrm{d}x.$ Then
\[ \int x^i \mathrm{d}\mu_{+}(x) = \int x^i \mathrm{d}\mu_{-}(x).\]\hfill$\Diamond$

Next, we see that the converse of Lemma \ref{Lem:indices} also holds.

\begin{theorem} Let $M$ be an infinite Hankel matrix with
$M_{0,0}=1$.  The following statements are equivalent.

\begin{enumerate}

\item The moment problem $M=M^{(\mu)}$ does not have a unique solution.

\item Given any $\alpha \in \mathbb{C}$ with $0 < \mathrm{Im}(\alpha)< 1$, there exists a vector $\xi \in \mathcal{H}_Q$ such that $H\xi = \lambda(1, \alpha, \alpha^2, \alpha^3, \cdots)$ for some $\lambda \in \mathbb{C}$.
\end{enumerate}

\end{theorem}

\begin{proof}
Let $\mathcal{H}_Q$ be the Hilbert space completion of the
quadratic form $Q_M$ as defined previously, and let $S$ be the
closure of the shift operator on $\mathcal{H}_Q$.  The solution to
the moment problem is unique if and only if the shift operator has
no self-adjoint extensions, by Theorem \ref{Thm:measures}. Using Lemma
\ref{Lem:indices}, this is true if and only if $S$ is self-adjoint, i.e. its deficiency indices are $(0,0)$.

\noindent $(1 \Rightarrow 2)$: This is a restatement of Lemma \ref{Lem:indices}, using $\alpha$ for the deficiency spaces instead of $i$.

\noindent $(2 \Rightarrow 1)$: Fix $\alpha \in \mathbb{C}$ such that $0< \mathrm{Im}(\alpha) < 1$ and denote \[s = (1,\alpha,\alpha^2, \alpha^3, \ldots ) \in \mathcal{H}_Q.\]  Assume there exists a nonzero $\xi \in \mathcal{H}_Q$ such that $H\xi = \lambda s$ for some nonzero scalar $\lambda \in \mathbb{C}$.  Then for $e_k$ a standard basis vector, \[
\langle Se_k|\xi \rangle_{\mathcal{H}_Q} = \langle Se_k|H\xi \rangle_{\ell^2} = \lambda \alpha^{k+1},\] where we have selected $\alpha$ so that $H\xi$ is in $\ell^2$ as well as in $\mathcal{H}_Q$.  We then compute \[\langle e_k|\xi\rangle_{\mathcal{H}_Q} = \langle e_k|H\xi \rangle_{\ell^2} = \lambda \alpha^k.\]  By linearity, we have for every $d \in \mathcal{D}$,\[ \langle Sd|\xi \rangle_{\mathcal{H}_Q} = \alpha \langle d|\xi\rangle_{\mathcal{H}_Q}, \] hence $\xi \in \mathcal{L}(\alpha)$ and we know the moment problem for $\mu$ does not have a unique solution.
\end{proof}


\section{Jacobi representations of matrices}\label{Sec:Banded}
The big picture in our work is an analysis of measures, passing from
moments to spectra. It turns out that a number of our problems may be
studied with the use of unbounded operators in Hilbert space, which
fits in with our multi-faceted operator-theoretic approach to moment
problems.

In this section we focus on the special relationship
between banded matrices which represent unbounded operators and
their associated moment problems.  Beginning with a Hankel matrix $M$, we find a 
(nonunique) banded Jacobi matrix $T$ which encodes the information in $M$ (Theorem \ref{Thm:MTClass}).  On the other hand, given a banded matrix $T$, we can find an associated moment matrix $M$ (Theorem \ref{Thm:JIndepOfMu}).  We conclude with a discussion of higher-dimensional analogues and banded matrices which arise in quantum mechanics.

\begin{definition}
We say that a matrix $T$ indexed by $\mathbb{N}_0 \times
\mathbb{N}_0$ is \textit{banded} if its nonzero entries are
restricted to the main diagonal and some number of diagonal bands
adjacent to the main diagonal.  Specifically, $T$ is banded if
there exist $b_1,b_2$ such that for all $j,k \in \mathbb{N}_0$, \[
T_{j,k} \neq 0 \Rightarrow  j-b_1 \leq k \leq j+b_2.\]
\end{definition}
When $T_1$ and $T_2$ are banded matrices, then
\[(T_1T_2)_{s,t}:= \sum_{n\in S} (T_1)_{s,n} (T_2)_{n, t}\] for each pair
$(s,t)\in \mathbb{N}_0\times \mathbb{N}_0$.  It is not hard to show that the banded matrices
form an algebra with composition as the multiplication operation.

Recall that a matrix $A$ with real entries is called
\textit{symmetric} if $A = A^{\mathrm{tr}}$ and if $A$ has complex
entries, it is called \textit{hermitian} if $A =
\overline{A}^{\mathrm{tr}}$.  Banded hermitian (or symmetric)
matrices $T:\mathbb{N}_0\times \mathbb{N}_0 \rightarrow
\mathbb{C}$ define symmetric \textit{operators} on
$\ell^2(\mathbb{N}_0)$ with $\mathcal{D}$ as dense domain.

\begin{example}\label{Ex:P}The matrix $P$ which represents the momentum operator in quantum mechanics is a banded symmetric matrix:
\begin{equation}\label{Eqn:MomentumP}
P=\frac{1}{2}
\begin{bmatrix}
0 & 1        & 0       & 0        &        &           &   &&\\
1 & 0        & \sqrt{2}&0         &        &           &   &&\\
0 & \sqrt{2} & 0       & \sqrt{3}        &        &           &   &&\\
  &          &\ddots   & \ddots   &  \ddots&           &   &&\\
  &          &         &\sqrt{n-2}& 0      & \sqrt{n-1}&   &&\\
  &          &         & 0        & \sqrt{n-1}& 0      &\sqrt{n} &&\\
  &          &         &          &\ddots    & \ddots    &  \ddots&         &\\
\end{bmatrix}.
\end{equation}

Given $v \in \ell^2$, we have
\[(Pv)_n = \frac{1}{2}(\sqrt{n-1}v_{n-1} + \sqrt{n}v_{n+1}).\]
\hfill$\Diamond$
\end{example}

The following definition generalizes to higher dimensions, but for
clarity we will state everything in a one-dimensional form.  We begin with a Hankel matrix $M$, and again work in the Hilbert space  $\mathcal{H}_Q$.  We
denote the standard orthonormal basis in $\ell^2(\mathbb{N}_0)$ by
$\{\delta_n\}_{n \in \mathbb{N}_0}$.  Even though $\{\delta_n\}_{n \in\mathbb{N}_0}$ is not an ONB in $\mathcal{H}_Q$, we also use $\delta_n$ to denote the element of $\mathcal{H}_Q$ with $0$'s in every place except the $(n+1)^{\textrm{st}}$, which contains a $1$.

\begin{definition}We say that a Hankel matrix $M$ is of \textit{$T$-class} if there
exists some banded hermitian matrix $T$, representing a symmetric
operator, such that
\begin{equation}\label{Eqn:TClass}
 M_{j,k} = \langle \delta_0 | T^{j+k}\delta_0\rangle_{\ell^2}.
\end{equation}
\end{definition}

Recall from Proposition \ref{Prop:NonuniqueMeasure} that if $M$ is
Hankel and positive semidefinite, then $M$ satisfies the equation \[ M_{j,k} = \langle \delta_0|\widetilde{S} 
^{j+k}\delta_0 \rangle _{\mathcal{H}_Q}, \]
where $\tilde{S}$ is a self-adjoint extension
of the shift operator on the Hilbert space $\mathcal{H}_Q$.  We will demonstrate the correspondence between this shift operator $S$ (when a moment matrix $M$ is given) and a symmetric Jacobi matrix $T$.  

\begin{lemma}\label{Lem:TClassPD}Every Hankel matrix of $T$-class is positive
semidefinite.
\end{lemma}
\begin{proof} Let $v\in\mathcal{D}$.  Then because $T$ is banded, all
summations are finite:
\begin{equation}
\begin{split}
\sum_j \sum_k \overline{v_j} M_{j,k} v_k
& = \sum_j \sum_k \overline{v_j}\Bigl\langle \delta_0|T^{j+k}\delta_0\Bigr\rangle_{\ell^2} v_k\\
& = \Bigl\langle \sum_j v_j T^j\delta_0 | \sum_k v_k T^k\delta_0\Bigr\rangle_{\ell^2}\\
& = \Big\| \sum_{j} v_j T^j\delta_0\Big\|_{\ell^2}^2\geq 0.
\end{split}
\end{equation}
\end{proof}

\begin{lemma}[One-dimensional version]\label{Lemma:WIsometry}
Suppose $M$ is Hankel of $T$-class.  Define $W:\mathcal{H}_Q \rightarrow \ell^2(\mathbb{N}_0)$ by
\[W(\delta_k):=T^k\delta_0 \qquad \forall k \in \mathbb{N}_0.\]
Then $W$ is an isometry.
\end{lemma}
\begin{proof}We first show that $\|W(\delta_k)\|_{\ell^2}^2 = \|\delta_k\|^2_{\mathcal{H}_Q}$.
Recall from the norm on $\mathcal{H}_Q$ that
\[\|\delta_k\|^2_{\mathcal{H}_Q} = Q_M(\delta_k),\] and
\[Q_M(\delta_k) = M_{k,k} = \langle \delta_0 | T^{k+k}\delta_0\rangle_{\ell^2}.\]  On the other hand,
\[ \|W(\delta_k)\|_{\ell^2}^2 = \langle T^k \delta_0 |T^k\delta_0\rangle_{\ell^2} = \langle \delta_0|T^{k+k}\delta_0\rangle_{\ell^2}.\]

Since every element $v$ of $\mathcal{D}$ is a finite linear
combination of $\delta_k$s, the rest of the proof follows from the
same computation used in the proof of Lemma
\ref{Lem:TClassPD}.\end{proof}

As in Equation (\ref{Eqn:DefnS}), let $S$ will refer to the
closure of the shift operator, which is defined on $\mathcal{D}$
by
\[ S(c_0, c_1, c_2, \ldots ) := (0, c_0, c_1, c_2, \ldots ).\]

Because we will consider deficiency indices of two different
operators $S$ and $T$, we let $\mathcal{L}_{\pm}(S)$ denote the dimension
of the deficiency subspace of $S$, and we
let $\mathcal{L}_{\pm}(T)$ denote the dimension of the deficiency
subspace of $T$.

\begin{lemma}\label{Lemma:WIntertwine} Let $M$ be of $T$-class and let the Hilbert space $\mathcal{H}_Q$ be as defined above.
 On the dense domain $\mathcal{D} \subset \mathcal{H}_Q$, $WS =
 TW$.
\end{lemma}

\begin{proof} Let $c=(c_0, c_1, \ldots)\in\mathcal{D}$. Then
\begin{equation}
\begin{split}
WS(c)
& = W(0, c_0, c_1, \ldots)\\
& = c_0 T\delta_0 + c_1 T^2 \delta_0 + c_2 T^3 \delta_0 + \cdots\\
& = T( c_0 \delta_0 + c_1 T \delta_0 + c_2 T^2 \delta_0 + \cdots
=TW(c).
\end{split}
\end{equation}\end{proof}
 An immediate result of Lemma \ref{Lemma:WIntertwine} is the
following:
\begin{lemma}Suppose the set $\{Wc\}_{ c\in \mathcal{D}}$ is dense in $\text{dom}(T)$.
Then the isometry $W$ maps the $S$-defect subspaces
$\mathcal{L}_{\pm}(S)$ into the $T$-defect subspaces
$\mathcal{L}_{\pm}(T)$.
\end{lemma}

\begin{proof} Suppose $f_{\pm}$ denotes an element of
$\mathcal{L}_{\pm}(S)$. We know that $S^*f_{\pm} =  \pm if_{\pm}$
if and only if $\langle Sc|f_{\pm}\rangle_{\mathcal{H}_Q} = \pm i\langle
c|f_{\pm}\rangle_{\mathcal{H}_Q}$ for all $c\in\mathcal{D}$.  With the assumption
of density, $T^*Wf_{\pm} = \pm iWf_{\pm}$ if and only if
\[\langle TWc|Wf_{\pm}\rangle_{\ell^2} = \pm i \langle Wc|Wf_{\pm}\rangle_{\ell^2}\]
for all $c\in\mathcal{F}$.

Let $c\in\mathcal{D}$.  Then by Lemmas \ref{Lemma:WIntertwine} and \ref{Lemma:WIsometry},
\begin{equation}
\begin{split}
\langle TWc|Wf_{+}\rangle_{\ell^2}
& = \langle WSc|Wf_{\pm}\rangle_{\ell^2} = \langle Sc|f_{\pm}\rangle_{\mathcal{H}_Q} \\
& = \langle c|S^*f_{\pm}\rangle_{\mathcal{H}_Q} = \pm i\langle c|f_{\pm}\rangle_{\mathcal{H}_Q} = \pm i\langle Wc|Wf_{\pm}\rangle_{\ell^2}.
\end{split}
\end{equation}\end{proof}

Recall that we established in Lemma \ref{Lem:indices} that the deficiency indices of $S$ are either $(0,0)$ or $(1,1)$.
\begin{corollary}\label{Cor:STdefects}Suppose the set $\{Wc \:\:|\:\: c\in \mathcal{F}\}$ is dense in $\text{dom}(T)$.  Then
\[  \mathcal{L}_{\pm}(T) \geq \mathcal{L}_{\pm}(S).\]
\end{corollary}

Table \ref{Table:STW} shows the relationships among $S$, $T$, and $W$.

\begin{table}\caption{Relationships among $S$, $T$, and $W$.}
\begin{tabular}{cccc}\label{Table:STW}
   $\mathcal{H}_{\text{min}}$ & $\overset{W}{\longrightarrow}$ & $\ell^2(\mathbb{N}_0)$           \\
$S \curvearrowright$         &                              & $\curvearrowright T$\\
   $\mathcal{H}_{\text{min}}$& $\overset{W}{\longrightarrow}$ & $\ell^2(\mathbb{N}_0)$            \\
\end{tabular}
\end{table}

Because
\begin{equation}\label{Eqn:TWWS}
TW = WS,
\end{equation}
when $S(c_0, c_1, \ldots):=(0, c_0, c_1, \ldots)$ in $\mathcal{H}_{\text{min}}$, we can take adjoints in (\ref{Eqn:TWWS}) to see that
\begin{equation}\label{Eqn:WTSW}
W^*T^* = S^*W^*.
\end{equation}
Finally, $W^*$ maps $\mathcal{L}_{\pm}(T)$ into $\mathcal{L}_{\pm}(S)$.  To see this, suppose $T^* f_{\pm} = \pm i f_{\pm}$.  Apply $W^*$ and use (\ref{Eqn:WTSW}):
\[W^*T^*f_{\pm} = \pm iW^*f_{\pm}= S^*W^*f_{\pm},\]
which implies that $W^*f_{\pm} \in \mathcal{L}_{\pm}(S)$.

We now come to the main result of this section:  every positive definite Hankel matrix $M$ is of $T$-class, and the matrix $T$ such that
\[M_{j,k} = \langle \delta_0 | T^{j+k}\delta_0\rangle_{\ell^2}\]
can be chosen to be banded with respect to the canonical ONB in
$\ell^2(\mathbb{N}_0)$.
\begin{theorem}\label{Thm:MTClass}
Let $M$ be a positive definite Hankel matrix, where $M$ is normalized so that $M_{0,0} = 1$.
\begin{enumerate}[\rm(a)]
\item Then there exists a banded symmetric matrix $T$ operating on $\ell^2(\mathbb{N}_0)$ such that (\ref{Eqn:TClass}) is satisfied.
\item We may choose $T$ of the banded form
\begin{equation}
T=
\begin{bmatrix}
b_0    & a_0    & 0   & 0   & 0  &\cdots &\\
\overline{a_0}    & b_1    & a_1 & 0   & 0  &\cdots &\\
0      & \overline{a_1}    & b_2 & a_2 & 0  &       &\\
0      & 0      & \overline{a_2} & b_3 & a_3&       &\\
\vdots & \vdots &     &     &    &\ddots &\\
\end{bmatrix}.
\end{equation}

\end{enumerate}
\end{theorem}

\begin{proof}By Theorem \ref{thm:real} there exists $\mu$ such that
\[M_{j,k} = \int_{\mathbb{R}} x^{j+k} \,\mathrm{d}\mu(x).\]
Now select the orthogonal polynomials $p_0(x), p_1(x), p_2(x), \ldots$ in $L^2(\mu)$ with $p_0(x) = 1$.  Then the mapping
\begin{equation}\label{Eqn:Fisometry}
\sum_{k=0}^{\infty} c_k p_k(x) \mapsto \{c_k\}_{k\in\mathbb{N}_0}
\end{equation}
is an isometry of a subspace in $L^2(\mu)$ onto $\ell^2(\mathbb{N}_0)$.  Specifically,
\begin{equation}\label{Eqn:SquareSum}
\Big\|\sum_k c_k p_k \Big\|^2_{L^2(\mu)} = \sum_{k}|c_k|^2.
\end{equation}

It is well-known (see, for example \cite{Akh65, AAR99}) that there exist sequences $\{a_n\}_{n \in \mathbb{N}_0}$ and $\{b_n\}_{\mathbb{N}_0}$ such that the following three-term recursion formulas are satisfied:
\begin{align*}
xp_0   & = b_0 p_0 + \overline{a_0}p_1\\
xp_1   & = a_0 p_0 + b_1p_1 + \overline{a_1} p_2\\
\vdots & \qquad \qquad \vdots\\
xp_j   &= a_{j-1}p_{j-1} + b_jp_j + \overline{a_j}p_{j+1}\\
\vdots & \qquad \qquad \vdots
\end{align*}  Because of the
Hankel property assumed for $M$, in proving (\ref{Eqn:TClass}), it
is enough to show that
\begin{equation}\label{Eqn:44}
M_{0,k} = \langle \delta_0 | T^k\delta_0\rangle_{\ell^2}.
\end{equation}
Using (\ref{Eqn:Fisometry}) and (\ref{Eqn:SquareSum}), we have for all $k\in\mathbb{N}_0$,
\begin{equation*}
\begin{split}
\langle \delta_0 | T^k\delta_0\rangle_{\ell^2}
& = \int_{\mathbb{R}} p_0 x^k p_0 \,\mathrm{d}\mu(x)\\
& = \int_{\mathbb{R}} x^k\,\mathrm{d}\mu(x) = M_{0,k},
\end{split}
\end{equation*}
which is the desired conclusion.\end{proof}

\begin{remark} There are many other choices of banded symmetric or hermitian matrices which solve (\ref{Eqn:TClass}).  The candidates for $T$ are dictated by applications. \end{remark}

Consider a fixed positive definite Hankel matrix $M$ such that the associated symmetric operator $S$ has deficiency indices $(1,1)$.  Then by Theorem \ref{Thm:measures}, there is a one-parameter family of inequivalent measures $\{\mu_z: |z| = 1\}$ such that $M^{(\mu_z)} = M$.  Further, for each measure $\mu_z$, we may compute an associated symmetric Jacobi matrix $T_z$, as in Theorem \ref{Thm:MTClass}.  The question we ask next is whether the Jacobi matrix $T_z$ depends on the measure $\mu_z$ used to compute the orthogonal polynomials.

Consider the Hilbert space $\mathcal{H}_Q$, in which the shift operator $S$ has dense domain.  The Jacobi matrix $T_{\mu}$ we just computed is also a symmetric operator, this time with dense domain in $\ell^2(\mathbb{N}_0)$.  Moreover, for each $\mu$ solving the $M$-moment problem, we have an isometry $F = F_{\mu}$ which maps $\mathcal{H}_Q$ into $L^2(\mu)$ and which intertwines $S$ with multiplication by $x$ (Lemma \ref{Lemma:MultiplicationByX}).  Recall also that $T_{\mu}$ encodes multiplication by $x$.

By looking at the relevant formulas for the orthogonal polynomials $\{p_n\}\subset L^2(\mu_z)$, we see that $T_z$ does not depend on $z$.  In particular, we use the following well-known formulas for the orthogonal polynomials to justify our answer \cite{Akh65}.

Define \[D_k = \det\begin{bmatrix}m_0 & \ldots &m_k\\
                          m_1 &\ldots &m_{k+1}\\
                          \vdots&      &\vdots\\
                          m_k & \ldots & m_{2k}\end{bmatrix}\]
and let $S_M = \{k\in \mathbb{N}: D_k\neq 0\}$.  Then if $k\in S_M$, then $p_k(x) = (D_{k-1}{D_k})^{-1/2}D_k(x)$, where
\[D_k(x) = \det\begin{bmatrix}m_0 &  m_1 &\ldots &m_k\\
                              m_1 &  m_2 &\ldots &m_{k+1}\\
                          \vdots  &      &       &\vdots\\
                          m_{k-1} & m_k &\ldots & m_{2k-1}\\
                          1 & x& \ldots & x^k\end{bmatrix}.\]

As a result, the isometry $F_{\mu}$ is also independent of $\mu$; the closed subspace spanned by the polynomials in $L^2(\mu_z)$ is independent of $z$, and it is only the relative orthogonal complement in $L^2(\mu_z)$ which depends on $z$.  The orthogonal complement can be described by
\[ \Bigl\{\psi\in L^2(\mu) \Big| \int \psi(x)x^k \textrm{d}\mu(x) = 0 \textrm{ for all }k\in\mathbb{N}_0\Bigr\}.\]  Lemma \ref{Lemma:MultiplicationByX} tells us that $S$ and $T$ are unitarily equivalent.  Moreover,
\[ \langle e_0 | S^ke_0\rangle_{\mathcal{H}_Q}  = \langle \delta_0 | T^k\delta_0\rangle_{\ell^2},\]
so the spectral measures derived from the self-adjoint extensions $\widetilde{S}$ of $S$ and $\widetilde{T}$ of $T$ produce the same measures $\mu$ which solve the $M$-moment problem.  Thus, if we are thus given the Jacobi matrix $T$, the self-adjoint extensions of $T$ in $\ell^2(\mathbb{N}_0)$ correspond to spectral measures $\mu$ which will have a common moment matrix $M$.   

We summarize with a theorem:
\begin{theorem}\label{Thm:JIndepOfMu}
Let $M$ be a positive definite Hankel matrix, and let $\mu$ be a measure such that $M^{(\mu)} = M$.  The associated Jacobi matrix $T$ is independent of the choice of measure $\mu$ which solves the $M$-moment problem.  Conversely, every symmetric Jacobi matrix $T$ gives rise to a moment problem. 

The following three conditions are equivalent:
\begin{enumerate}[\rm(a)]
\item The solution $\mu$ to the $M$-moment problem is unique.
\item The deficiency indices for $T$ are $(0,0)$.
\item The polynomials are dense in $L^2(\mu)$.
\end{enumerate}
\end{theorem}

\begin{proof}  The forward direction is shown in the discussion preceding the statement of the theorem.  Given a symmetric  Jacobi matrix $T$, define a Hankel matrix $M$ by taking \[M_{j,k} = \langle \delta_0|T^{j+k}\delta_0\rangle_{\ell^2}.\]  Since powers of banded matrices remain banded, the entries of $M$ are all finite.  By Lemma \ref{Lem:TClassPD}, $M$ is a positive-definite Hankel matrix, hence yielding a moment problem $M = M^{(\mu)}$ which has at least one solution.

The equivalent statements follow from Theorem \ref{Thm:measures}, Corollary \ref{Cor:STdefects}, and the discussion above showing that $\mathcal{L}_{\pm}(T) \leq \mathcal{L}_{\pm}(S)$.
\end{proof}

\section{The triple recursion relation and extensions to higher dimensions}
Let $p_0(x) \equiv 1,p_1(x), p_2(x), \ldots$ be the orthogonal polynomials with respect to $\mu$, where $\textrm{deg}(p_k) = k$.  (Here we assume that $\mu$ corresponds to a positive definite, not positive semidefinite, linear functional on $L^2(\mu)$.)  We can use Gram-Schmidt on $\{1, x, x^2, \ldots\}$, so that
\[
\textrm{span}\{1, x, x^2, \ldots\} = \textrm{span}\{p_0(x), p_1(x), p_2(x), \ldots\}
\]
and
\[
\int p_j(x) p_k(x) \,\mathrm{d}\mu(x) = \delta_{j,k} \textrm{ (the Kronecker delta)}.
\]
Then for $k\in\mathbb{N}_0$ we have
\begin{equation}\label{Eqn:PolyRecur1}
xp_{n-1}(x) = b_{n-1}p_{n-2}(x) + a_{n-1} p_{n-1}(x) + \overline{b_n}p_n(x)
\end{equation}
and
\begin{equation}\label{Eqn:PolyRecur2}
x p_n(x) = b_n p_{n-1}(x) + a_n p_n(x) + \overline{b_{n+1}}p_{n+1}(x).
\end{equation}
Note that the $\overline{b_n}$ appearing in (\ref{Eqn:PolyRecur1}) is the conjugate of the $b_n$ appearing in (\ref{Eqn:PolyRecur2}).  To see this that this is true, suppose
\begin{align*}
x p_n(x)      &= Bp_{n-1}(x) + \textrm{  terms in }p_{n}(x) \textrm{ and }p_{n+1}(x) \\
xp_{n-1}(x) & = Cp_n(x) + \textrm{ terms in }p_{n-2}(x) \textrm{ and }p_{n-1}(x).
\end{align*}
We now take advantage of the orthogonality:
\begin{equation*}
\begin{split}
B
& = \langle p_{n-1}(x) | xp_n(x)\rangle_{L^2(\mu)}
    = \langle xp_{n-1}(x) | p_n(x)\rangle_{L^2(\mu)}\\
& = \langle Cp_n(x) | p_n(x)\rangle_{L^2(\mu)}
    = \overline{C}\langle p_n(x) | p_n(x)\rangle_{L^2(\mu)} = \overline{C}.
\end{split}
\end{equation*}
The triple recursion relation can be considered in terms of projections onto finite-dimensional subspaces.  Set
\[
\mathcal{H}_n:
=\textrm{span}\{1, x, \ldots, x^n\}
= \textrm{span}\{p_0(x), p_1(x), \ldots, p_n(x)\}
\]
and let $Q_n$ be the projection onto $\mathcal{H}_n$ with
\[
Q_n^* = Q_n = Q_n^2
\]
and
\[ Q_n(L^2(\mu)) = \mathcal{H}_n.
\]
(Note:  $Q_n$ is a projection, and is \textit{not} related to our earlier quadratic form.)  We are interested in the shift operator, which is the same as multiplication by $x$ on $\mathcal{P}\subset L^2(\mu)$.  We have
\[
x\mathcal{H}_n \subset \mathcal{H}_{n+1},
\]
so
\[
Q_{n+1} xQ_n = xQ_n.
\]
Set $Q_n^{\perp}:=I - Q_n$---that is, $Q_n^{\perp}$ is the projection onto $\mathcal{H}_n^{\perp}$, where $\mathcal{H}_n^{\perp} = L^2(\mu) \ominus \mathcal{H}_n$.

If $k < n-1$, then $\langle p_k, xp_n\rangle_{L^2(\mu)} = 0$.  To see this, write
\[
\langle p_k, xp_n\rangle_{L^2(\mu)} = \langle xp_k, p_n\rangle_{L^2(\mu)},
\]
and $p_n$ is orthogonal to any polynomial with degree less than $n$.

Now,
\begin{equation*}
\begin{split}
xp_n & = Q_n(xp_n) + Q_n^{\perp}(xp_n)\\
& = \underbrace{b_n p_{n-1} + a_n p_n}_{Q_n(xp_n)} + \underbrace{\overline{b_{n+1}}p_{n+1}}_{Q_n^{\perp}(xp_n)}.
\end{split}
\end{equation*}

If we restrict $J^{*}$ to $\mathcal{P}$, the subspace of all polynomials, then
\[
J^{*}|_{\mathcal{P}} = J,
\]
since $J\subset J^*$.  Moreover, since the projections $Q_n$ are self-adjoint,
\[
Q_{n+1} J Q_n = JQ_n \Rightarrow  Q_n J Q_{n+1} = Q_n J.
\]
Finally,
\[
Q_{n-2} J (Q_n - Q_{n-1}) =0,
\]
since
\[
Q_{n-2} (JQ_n) = Q_{n-2} J Q_{n-1} Q_n  = Q_{n-2} J Q_{n-1}.
\]

The projection approach to the recursion relation can be extended to $\mathbb{R}^d$ for $d > 1$.  For each of the $d$ coordinate directions, there is a Jacobi matrix
\[
J_k =
 \begin{bmatrix}
a_0^{(k)}   &b_1^{(k)}   &0          &0      &0       &\cdots\\
\overline{b_1^{(k)} }   &a_1^{(k)}   &b_2^{(k)}      &0       &0       &\cdots\\
0        &\overline{b_2^{(k)} }  &a_2^{(k)}     &b_3^{(k)}    &0       &\cdots\\
0        &0      & \overline{b_3^{(k)} }    &a_3^{(k)}    &b_4^{(k)}   &         \\
0        &0      &0          &\overline{b_4^{(k)} }  &a_4^{(k)}    &     \\
\vdots  &\vdots & \vdots           &          &         &\ddots \\
\end{bmatrix}
\]
and a shift operator $S_k$, which is realized by multiplication in the $k^{th}$ coordinate
\[
S_k p(x_1, \ldots, x_d) = x_k p(x_1, \ldots,x_d).
\]
Recall that the degree of $x^{\alpha}=x_1^{\alpha_1}\cdots x_d^{\alpha_d}$ is $\alpha_1 + \ldots + \alpha_d$.  With this notation, the finite-dimensional subspaces are
\[
\mathcal{H}_n = \{ p \in \mathcal{P} | \textrm{deg}(p) \leq n\}.
\]

\section{Concrete Jacobi matrices}

We examine the momentum
and position operators from quantum mechanics for one degree of freedom and then study some
Hamiltonian operators (for example, the polynomials
in the momentum and position operators in Table
\ref{Table:Polynomials}).  The operators we consider all have a common
property: in a natural orthonormal basis their representations
take the form of infinite banded matrices; that is, the matrices have
zeros outside a band around the diagonal of finite width.

The banded property of the matrices makes matrix multiplication
easy; under multiplication, the banded matrices form an algebra of
unbounded operators.  While such banded matrices follow simple
algebraic rules, their spectral theory can be subtle. For example,
we show that these operators may not have a well-defined spectral
resolution. Using von Neumann's deficiency indices, we showed above the
connection of Jacobi matrices to the theory of extensions of
symmetric operators with dense domain, and thereby to moment problems.

One often encounters problems in physics, such as the Heisenberg
banded matrices $T$, where the nature of the bands is dictated by
applications.  A particular infinite matrix $T$ represents an
operator in an $\ell^2$ sequence space.  In fact, in a particular
application, $T$ may be realized in a different Hilbert space, for
example an $L^2$ function space, but the function version will be
unitarily equivalent to the matrix model.  In particular, we refer to the fact that
Heisenberg's matrix formulation of quantum mechanics is unitarily
equivalent to Schr\"{o}dinger's wave formulation in function
space.  For example, the momentum operator $P$ is represented by
the matrix (\ref{Eqn:MomentumP}) in $\ell^2$ and by the operator
$\frac{1}{i}\frac{\mathrm{d}}{\,\mathrm{d}x}$ on a dense subspace of
$L^2(\mathbb{R})$.  See Examples \ref{Ex:P} and \ref{Ex:Q} and the
remark following the two examples. 

From our banded matrix $T$ we then get a Hankel matrix $M$, and we
apply our theory to $M$. In particular, we find the measures $\mu$
which solve the moment problem for $M$.  We use operator theory in
constructing the family of measures $\mu$ which solve the moment
problem at hand. 

\begin{table}\caption{Two approaches to moments and banded matrices.}
\begin{tabular}{l}\label{Table:AB}
\boxed{\begin{minipage}{.15\linewidth}Banded $T$\end{minipage}}$\longrightarrow$\boxed{\begin{minipage}{.2\linewidth}Hankel $M_T$\end{minipage}}$\longrightarrow$\boxed{\begin{minipage}{.18\linewidth} {measures $\mu$}\end{minipage}}
\\

\\

\boxed{\begin{minipage}{.15\linewidth}Hankel $M$\end{minipage}}$\longrightarrow$\boxed{\begin{minipage}{.15\linewidth}$M = M^{(\mu)}$\end{minipage}}$\longrightarrow$\boxed{\begin{minipage}{.17\linewidth}Banded $T_M$\end{minipage}}
\end{tabular}
\end{table}

We emphasize further that in applications, one
typically encounters a much richer family of banded symmetric or
hermitian infinite matrices $T$; for example those from
Heisenberg's quantum mechanics.  In these matrices, the band-size
will typically be more than three.  In fact the band can be any
size, and the deficiency indices can be anything. However this
wider class of banded matrices, including for example anharmonic
oscillators, may be studied with the aid of the associated Jacobi
matrices.

We begin with two Jacobi matrices, the matrix $P$ given in Example \ref{Ex:P} and $Q$ here.

\begin{example}\label{Ex:Q}
The operator $Q$ is represented by multiplication by $x$ on $L^2(\mathbb{R})$.  The operator $Q$ can be represented on $\ell^2$ by a matrix defined by
\[(Qv)_n = \frac{1}{2i}(\sqrt{n-1}v_{n-1} - \sqrt{n}v_{n+1}).\]
\end{example}

\begin{remark}The two operators $P$ and $Q$ in Examples \ref{Ex:P} and \ref{Ex:Q} have dense domains in $L^2(\mathbb{R})$:
\[ \textrm{dom}(P) = \{f\in L^2(\mathbb{R}) : f'\in L^2(\mathbb{R})\}\]
and
\[ \textrm{dom}(Q) = \{f\in L^2(\mathbb{R}) : xf(x)\in L^2(\mathbb{R})\}.\]
The operators $P$ and $Q$ are both self-adjoint, and they are unitarily equivalent via the Fourier transform in $L^2(\mathbb{R})$.

Setting
\[ A_{\pm}:=\frac{1}{\sqrt{2}}(P\pm i Q),\]
we get $A_{+}^* = A_{-}$, and the commutator
\begin{equation}\label{Eqn:Commutator}
[A_{+}, A_{-}] =-I.
\end{equation}  The Hermite function $h_0(x) = c_0 e^{-x^2/2}$ satisfies
\[ A_{-}h_0 = 0.\]

An application of Equation (\ref{Eqn:Commutator}) yields
\[(A_{+}A_{-})A_{+}^n h_0 = n A_{+}^n h_0 \textrm{ for each }n\in\mathbb{N}.\]
The functions $h_n := c_n A_{+}h_0$ diagonalize the harmonic oscillator Hamiltonian
\[H:=A_{+}A_{-} = \frac{1}{2}(P^2 + Q^2 - I),\]
and the constants $c_n$ can be chosen so that $\{h_n|n\in\mathbb{N}_0\}$ is an orthonormal basis in $L^2(\mathbb{R})$ consisting of Hermite functions.

Using this ONB, we arrive at the two matrix representations for $P$ and $Q$ in Examples \ref{Ex:P} and \ref{Ex:Q}.  Specifically,
\[\langle h_{n-1}|Ph_n\rangle = \frac{1}{2}\sqrt{n},\]
\[\langle h_{n}|Ph_n\rangle = 0,\]
and
\[\langle h_{n+1}|Ph_n\rangle = \frac{1}{2}\sqrt{n+1},\]
which is the Jacobi matrix in (\ref{Eqn:MomentumP}).
\end{remark}
\begin{example}If $T = QPQ$ in $\ell^2(\mathbb{N}_0)$, then $T$ has deficiency indices both equal to $1$.

\begin{proof}A differential equations problem.
\end{proof}
\end{example}

\begin{table}\caption{Some polynomials in the position and momentum operators.}
\begin{tabular}{c|c|c}\label{Table:Polynomials}
   $T$    & index   & spectrum\\
&&\\
\hline
$QPQ$    & $(1,1)$ & depends on the choice \\
         &         & of selfadjoint extension\\
&&\\

$P^2+Q^4$&$(0.0)$  & discrete, anharmonic oscillator\\

&&\\
$P^2-Q^4$ & $(2,2)$& repulsive potential\\ && quantum particle shoots to infinity in finite time \\

&&\\
$P^2+Q^2$&$(0,0)$  & $\{2n+1 \:\:|\:\: n\in\mathbb{N}_0\}$\\

&&\\

$P^2-Q^2$&$ (0,0)$ & continuous $\mathbb{R}$ (easier to see in $L^2(\mathbb{R})$\\
&&than by using  matrix calculations)\\
\hline
\end{tabular}
\end{table}